\documentclass[11pt,reqno,letterpaper]{amsart}
\usepackage{amssymb}
\usepackage{graphicx,color}
\usepackage{float}

\usepackage[normalem]{ulem}
\usepackage{fullpage} 

\usepackage{comment}
\usepackage{color}
\usepackage{amsxtra}
\usepackage{amssymb}
\usepackage{mathtools}
\usepackage{enumerate}
\usepackage{nicefrac}
\mathtoolsset{showonlyrefs} 
\usepackage{ mathrsfs }
\usepackage{tikz}
\usetikzlibrary{positioning, calc}
\usetikzlibrary{shadings}
\usepackage{color}
\usepackage{subcaption}
\usepackage{hyperref}
\hypersetup{
colorlinks   = true, 
urlcolor     = blue, 
linkcolor    = blue, 
citecolor   = red 
}

\newtheorem{theorem}{Theorem}[section]
\newtheorem{proposition}[theorem]{Proposition}
\newtheorem{corollary}[theorem]{Corollary}
\newtheorem{lemma}[theorem]{Lemma}

\theoremstyle{remark}
\newtheorem{remark}[theorem]{Remark}

\theoremstyle{definition}
\newtheorem{definition}[theorem]{Definition}

\makeatletter
\renewcommand{\tocsection}[3]{\indentlabel{\@ifnotempty{#2}{\ignorespaces #1 #2.\,}}#3}

\def\l@subsection{\@tocline{2}{0pt}{2em}{}{}}
\def\l@subsection{\@tocline{2}{0pt}{2em}{}{}}
\makeatother

\numberwithin{equation}{section}

\newcommand{\RR}{{\mathbb R}}

\newcommand{\eps}{{\varepsilon}}
\def\1{\raisebox{2pt}{\rm{$\chi$}}}

\usepackage{enumerate}

\def\undertilde#1{\mathord{\vtop{\ialign{##\crcr
$\hfil\displaystyle{#1}\hfil$\crcr\noalign{\kern1.5pt\nointerlineskip}
$\hfil\tilde{}\hfil$\crcr\noalign{\kern1.5pt}}}}}

\DeclareMathAlphabet{\mathpzc}{OT1}{pzc}{m}{it}

\title{
Finding the convex hull of a set using the flow by minimal curvature with an obstacle. A game theoretical 
approach
}

	\author{Irene Gonz\' alvez, Alfredo Miranda, Julio D. Rossi, Jorge Ruiz-Cases}
	
	 \address{ 
		 I. Gonz\'alvez and J. Ruiz-Cases. Departamento de Matem\'{a}ticas, Universidad Aut\'{o}noma de Madrid,
		Campus de Cantoblanco, 28049 Madrid, Spain.
		\newline
		\texttt{~irene.gonzalvez@uam.es,~jorge.ruizc@uam.es}; 
		\bigskip
		\newline
		 \indent A. Miranda and J. D. Rossi. Departamento de Matem\'{a}ticas, FCEyN, Universidad  de Buenos Aires, Pabell\'{o}n I, Ciudad Universitaria (1428), Buenos Aires, Argentina.  \newline 
		 \texttt{~amiranda@dm.uba.ar,~jrossi@dm.uba.ar }		
}
	

\begin{document}

\begin{abstract} In this paper we look for the convex hull of a set 
using the geometric evolution by minimal curvature of a hypersurface that surrounds the set. To find the convex hull,
we study the large time behavior of solutions to an obstacle problem for the level set formulation of the geometric flow driven by the minimum of the principal curvatures (that coincides with the mean curvature flow only in two dimensions). We prove that the superlevel set where the solution to this obstacle problem is positive converges as time goes to infinity to the convex hull of the obstacle. Our approach is based on a game-theoretic approximation
for this geometric flow that is inspired by previous results for the mean curvature flow. 
\end{abstract}

\maketitle

\begin{center}\begin{minipage}{12cm}{\tableofcontents}\end{minipage}\end{center}

\section{Introduction}

\subsection{Description of the main goal}

The aim of this paper is to study the asymptotic behaviour of the solution to the obstacle problem for a degenerate fully nonlinear parabolic equation that models the motion by minimal curvature of a hypersurface.  
We prove (under certain conditions on the initial condition and the obstacle) that the positivity set of this solution converges to the convex hull of the set where the obstacle is positive as $t$ goes to infinity. 

To obtain our results we use a deterministic two-person zero-sum game (inspired by previous results by \cite{KS} and \cite{Misu}) and the theory of viscosity solutions (we will rely on previous theory developed for viscosity solutions to geometric flows, see \cite{CIL} and \cite{GGIS}).	

Now, let us describe the main ingredients that we need to state and prove our results. First we introduce a parabolic nonlinear partial differential equation (PDE) related to the geometric motion by minimal curvature with an obstacle and next we describe an associated game whose value functions
approximate solutions to the PDE. 

\subsubsection{The movement by minimal curvature with an obstacle
and its associated parabolic equation} 
Our aim is to describe how a hypersurface that is the boundary of a connected domain, $S=\partial \Omega_{0}\subset \mathbb{R}^N$,  $N\geq 2$, evolves in time 
according to the minimal curvature flow. We will use a level set
approach to describe this geometric evolution. Assume that there is a 
real valued function, $u(x,t)$, defined for $(x,t) \in \mathbb{R}^N \times [0,\infty)$, and consider
the zero superlevel sets of $u(x,t)$, 
$$
\Omega_t = \{ x : u (x,t) >0 \}.
$$
In what follows we denote by $\nabla u (x,t)$
and by $D^2 u (x,t)$ the gradient and the Hessian of $u$ 
with respect to the spatial variable, $x$. 
Assume that $\partial \Omega_t$ is smooth. 
At a point $(x,t)\in \partial \Omega_t$ we have 
$\nabla u (x,t) \perp \partial \Omega_t$ (also assume that $(x,t)$ is a regular point and that we have  
$|\nabla u |(x,t)=1$) and 
for a unitary vector $v \perp \nabla u (x,t)$ (notice that $v$ is
tangential to the hypersurface $\partial \Omega_t$)
the quantity $-\langle D^2 u (x,t) v , v \rangle$ gives the  
curvature of $\partial \Omega_t$ in the direction of $v$.  
Therefore, under these conditions,
the minimum of the curvatures of $\partial \Omega_t$ at a regular point is given by
$$
\kappa_{min} =\min_{\substack{|v|=1 \\ v \mbox{ tangent to } \partial \Omega_t }} \Big\{ \kappa (v) \Big\}
= \inf_{\substack{|v|=1 \\ v\perp \nabla  u (x,t) }} \Big(-\langle   D^2  u (x,t) v, v 
\rangle \Big).
$$
Then, we consider the geometric evolution of the hypersurface
$\partial \Omega_t$ moving its points in the direction of the normal vector (pointing
inside the set $\Omega_t$) with speed given by the minimal curvature,  
$$
V = -\kappa_{min} \qquad \mbox{ on } \partial \Omega_t,
$$
and we obtain, as the associated level set formulation for this geometric evolution, the parabolic equation 
\begin{equation}
\label{PE}	\partial_t u(x,t)  + \mathcal{L}(u(x,t)) = 0,
\end{equation}
associated with the elliptic and degenerate operator $\mathcal{L}$ given by
\begin{equation}\label{1} 
    \mathcal{L}(f(x)) :=  \inf_{\substack{|v|=1 \\ v\perp \nabla f(x)}} \Big(-\langle   D^2f(x) v, v \rangle \Big).
\end{equation}

In this paper we add to this geometric evolution the presence of an open and bounded obstacle $K$ (the 
hypersurface is stopped when it touches $K$). 
We assume that 
our geometric evolution starts with the prescribed set inside, that is, we have $K\subset\subset \Omega_{0}$. 
To include the obstacle in the level set formulation, we consider
a continuous function $\psi$ such that
$$
K = \{ x : \psi (x) > 0\}
$$ 
and we assume that the initial condition for our evolution satisfies $$u_0 (x)=u(x,0) \geq \psi (x)$$
in the whole $\mathbb{R}^N$. Notice that under this assumption
we have $$K = \{ x : \psi (x) > 0\} \subset \{ x : u_0 (x) >0 \}
=
\Omega_0,$$ that is, the initial hypersurface $ \partial
\Omega_0$ surrounds the obstacle $K$. 
Therefore, we get that
the parabolic problem that corresponds to the level set formulation for the movement by minimum curvature with an obstacle is the following:
	\begin{equation}
		\label{ParabolicPb}\tag{$P$}
\begin{cases}
	 u(x,t) \geq   \psi (x),& \;\;(x,t)\in
	 \mathbb{R}^N \times\{ t>0\},   \\
	 \partial_t u(x,t)  + \mathcal{L}(u(x,t)) \geq 0,& \;\;(x,t)\in
	  \mathbb{R}^N \times\{ t>0\}, \\
	\partial_t u(x,t)  + \mathcal{L}(u(x,t)) = 0,& \;\;(x,t)\in\{(x,t) : u (x,t) > \psi (x)\}, \\
	u(x,0) = u_0(x),&\;\;x\in\mathbb{R}^{N}.
	\end{cases}
   \end{equation}
This problem can also be written as 
\begin{equation}
\begin{cases}
    \max \Big\{ - \partial_t u(x,t)  - \mathcal{L}(u(x,t)) , \psi (x) -u (x,t) \Big\} = 0, 
    & \;\;(x,t)\in  \mathbb{R}^N \times\{ t>0\}, \\
    u(x,0) = u_0(x), &\;\;x\in\mathbb{R}^{N}.
    \end{cases}
\end{equation}
For this obstacle problem we will prove that there exists a unique viscosity solution. To this end, we first show that a comparison principle holds for \eqref{ParabolicPb}, let
$\overline{u}$ be a supersolution
and $\underline{u}$ be a subsolution of \eqref{ParabolicPb}, then
$
\overline{u} (x,t) \geq \underline{u} (x,t)$
 for every $(x,t)\in\mathbb{R}^N \times\{ t>0\}$.
Next, we will obtain the existence and uniqueness of the solution via Perron's method, that is, we obtain the unique solution of the obstacle problem as the infimum of supersolutions to $u_{t}+\mathcal{L}u=0$ in $\mathbb{R}^{n}\times (0,\infty)$ that are greater or equal than the obstacle $\psi$ and are also greater or equal than $u_0$ at time $t=0$.

Once we have existence and uniqueness of solutions we turn our attention
to their asymptotic behavior.
Now, we will assume that 
$\overline{K} \subset \Omega_0$ (for instance, this happens 
when $\psi (x) < u_0(x)$).
 Let us consider again the positivity sets of $u(x,t)$, 
$\Omega_t = \{ x : u(x,t) >0 \}$.
According to our previous discussion, these sets encode the evolution
by minimal curvature of the hypersurface $\partial \Omega_0 = \partial \{x:u_0(x)>0\}$ with the obstacle 
$K= \{ x : \psi (x) > 0\}$. 

Our main goal is to prove that the sets $\Omega_t = \{ x : u(x,t) >0 \}$ approximate the convex hull of $K$ as $t \to \infty$. To get some insight on why this result holds let us look at a simple example. 
In $\mathbb{R}^3$ let $K$ be the union of two disjoint balls, 
$$
K = B_1 (2,0,0) \cup B_1 (-2,0,0). 
$$
For this set $K$ its convex hull, that we denote by $\textrm{co}(K)$, is given by a tube that joins the two balls,
see Figure \ref{evolucionacilindro}.
Now, take a set $\Omega_0$ that is convex and contains $\textrm{co}(K)$. When the set evolves by minimal curvature it 
shrinks until it touches the obstacle and then converges to the convex hull. In fact, remark that 
at points on the boundary of the convex hull that are not in $K$ we have that
one the principal curvatures is positive but the one that corresponds to the 
horizontal direction is zero. Therefore, the flow  by minimal curvature with 
$K$ as obstacle has this
convex hull as a stationary solution, see Figure \ref{evolucionacilindro} below. 

\begin{figure}[H] 
	\centering
	\includegraphics[width=14cm]{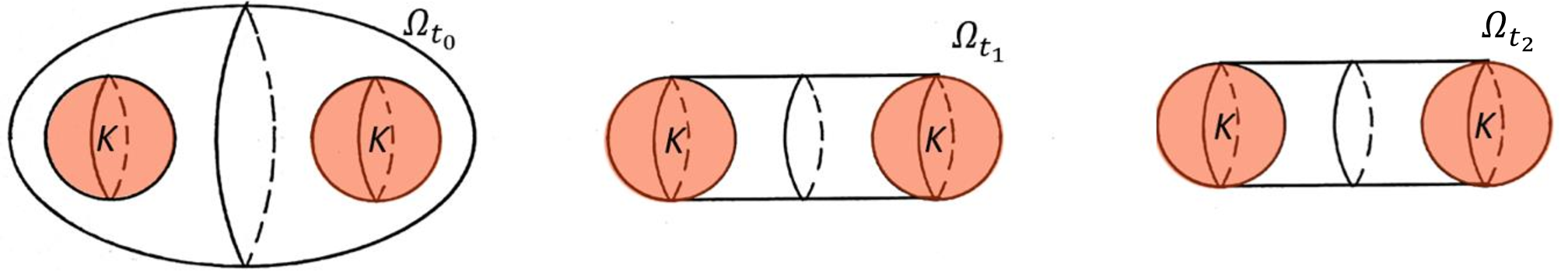}
	\caption{Motion by minimal curvature flow.}
	\label{evolucionacilindro}
\end{figure}

Notice that this does not happens for the usual mean curvature flow.
In fact, the mean curvature of a point 
on the boundary of the convex hull that is not in $K$ is strictly positive
(one of the main curvatures is zero but the other one is strictly positive), and
therefore we have that the evolution by mean curvature flow with obstacle $K$ continues shrinking
the set when it reaches the convex hull of $K$, see Figure \ref{evolucionachurro} below.
\begin{figure}[H]
	\centering
	\includegraphics[width=14cm]{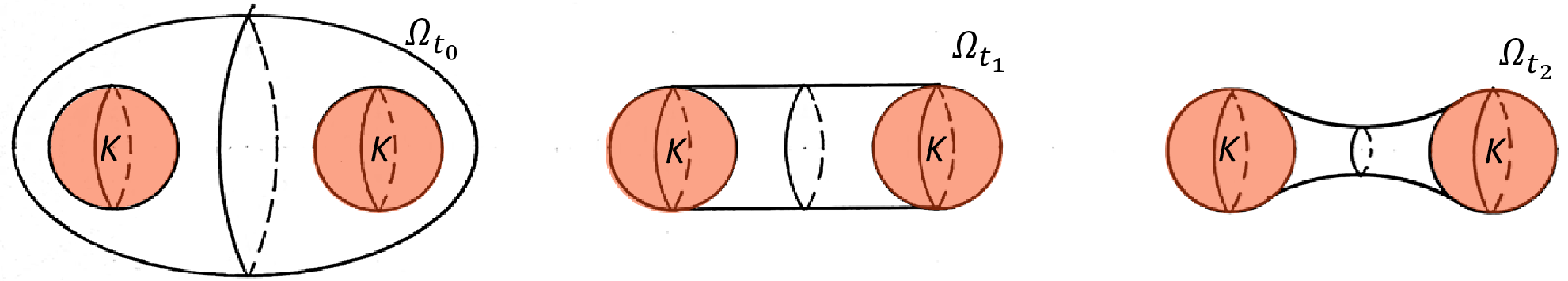}
	\caption{Motion by mean curvature flow.}
	\label{evolucionachurro}
\end{figure}

The motion by mean curvature is nowadays a classical subject, we refer to~\cite{2,ES,CGG,Giga,8,15,16} and references therein. Since there is only one curvature for a curve in dimension two our evolution problem coincides with the mean
curvature flow with an obstacle studied in~\cite{9,Mercier,19,Misu}. However,
for higher dimensions, $N\geq 3$, the minimal curvature flow studied here is
quite different from the mean curvature flow.

In the special case that the initial domain is a ball $\Omega_0=B_R (0) \subset \mathbb{R}^N$, since all the curvatures are equal, the evolution 
according to minimal curvature coincides with the evolution by
mean curvature (and is given by $B_{R(t)} (0)$
with $R'(t) = 1/R(t)$) until the first time when the surface touches
the obstacle, from that time on, the two evolutions are different. 
Then, when there is no obstacle, $K =\emptyset$, and the initial domain is a ball $B_R (0) \subset \mathbb{R}^N$ we get that the surface shrinks to a point in finite time. Since there is a comparison principle for the motion by minimal curvature we have that when $K =\emptyset$ the 
hypersurface 
shrinks and disappears in finite time.

\subsubsection{A game approximation for the parabolic problem}
Next, let us describe a game whose value function approximates 
the solution to \eqref{ParabolicPb}.

The game is a deterministic two-person zero-sum game. 
As in \cite{KS} we use Paul for the name of the first player and Carol for the second player. Take $\varepsilon>0$ (a parameter that controls the size of the possible movements in the game), $\Omega_0 \subset \mathbb{R}^N$ an open, bounded set and $K_\varepsilon \subset \Omega_0$ an open set
(it will be convenient that the obstacle for the game also depends on the parameter $\varepsilon$). We also have two functions, $u_0 :\mathbb{R}^N \mapsto \mathbb{R}$ and $\psi_\varepsilon :\mathbb{R}^N \mapsto \mathbb{R}$
that will give us the final payoff of the game. We will assume that 
the obstacle satisfies
$$
K_\varepsilon = \{ x : \psi_\varepsilon (x) >0\} .
$$
On the other hand, the function $u_0(x)$ will give the initial condition for the limit equation and then we assume that $u_0(x)
\geq \psi_\varepsilon (x)$. The initial condition $u_0$ may also
depend on $\varepsilon$ but, to simplify, we drop this dependence here. 

Let $x=x_0$ be the initial position of the game and $t = t_0>0$ the initial time. The game is played as follows: at the $i-$th round, Paul chooses a unitary vector $v_i$ and Carol chooses a sign $b_i = \pm 1$ after Paul's choice. Then, the next position of the game is given by 
\[x_i = x_{i-1} + b_i v_i \varepsilon,\]
while time decreases as 
\[t_i = t_{i-1}- \frac12 \varepsilon^2.\]
As a remark we point out that one can also choose to play with 
$x_i = x_{i-1} + \sqrt{2} b_i v_i \varepsilon$ and 
$t_i = t_{i-1}-  \varepsilon^2$ as in \cite{KS,Misu}. We prefer to keep
the spatial steps of size $\eps$ in order to simplify some
computations when analyzing the game. 

After $n = \displaystyle \left\lceil \frac{2 t_0}{\varepsilon^2} \right\rceil$ rounds, the game ends and Carol pays the terminal cost $u_0(x_n)$ to Paul. 
We also add an extra stopping rule. At any round Paul has the right to quit the game and Carol pays him the value given by the obstacle, $\psi_\varepsilon(x_i)$. We define $u^\varepsilon (x,t)$ as the value of this game starting at $x$ with time $t$ (the value of the game is just the final cost optimized by both players, Paul wants to maximize the expected outcome, while Carol aims to minimize it). The value of the game is given by the following formula, 
\[u^\varepsilon (x_0,t_0) \! = \! \max \! \Bigg\{ \psi_{\varepsilon}(x_0), \sup_{|v_1|=1} \min_{b_1=\pm 1} \bigg\{\cdots \max \Big\{\psi(x_{n-1}),\sup_{|v_n|=1} \min_{b_n=\pm 1} u_0 (x_{n-1} + b_n v_n \varepsilon )  \Big\}  \bigg\}\Bigg\}\]
for $n=\displaystyle \left\lceil \frac{2 t_0}{\varepsilon^2} \right\rceil$.
Notice that the rule that allows Paul to stop the game at any time and obtain $\psi_\varepsilon$ ensures that the value of the game satisfies $$u^\varepsilon (x,t) \geq \psi_\varepsilon(x).$$ 
After only one round of the game starting from $(x,t)$ we have that the outcome is given by 
\begin{equation}
	\label{DPP} \tag{$P_{\varepsilon}$}
	\left\{
	\begin{array}{ll}
	\displaystyle
	u^\varepsilon(x,t) = \max \left\{\psi_\varepsilon(x), \sup_{|v|=1} \min_{b = \pm 1} u^\varepsilon \big(x + b v\varepsilon, t- \frac12\varepsilon^2 \big) \right\}, \qquad & x\in \mathbb{R}^N, \ t>0,\\[8pt]
	u^\varepsilon(x,t) = u_0 (x), \qquad & x\in \mathbb{R}^N, \ t\leq 0.
	\end{array} \right.
\end{equation}

The equation that appears in \eqref{DPP} is known as the Dynamic Programming Principle (DPP) in the game literature,
see \cite{MS2}, and reflects the rules of the game. In fact, Paul wants to
maximize the outcome and he can choose between getting $\psi_\varepsilon (x_0)$
or the value after making one move (that is given by the supremum among directions (Paul's choice) of the minimum 
among a sign (Carol's choice) of the value at the new position at time $t_0-\frac{1}{2}\varepsilon^2$). 
The successive positions of the game are illustrated in Figure \ref{cilindroconamebas} below.  

\begin{figure}[H]
	\centering
	\includegraphics[width=13cm]{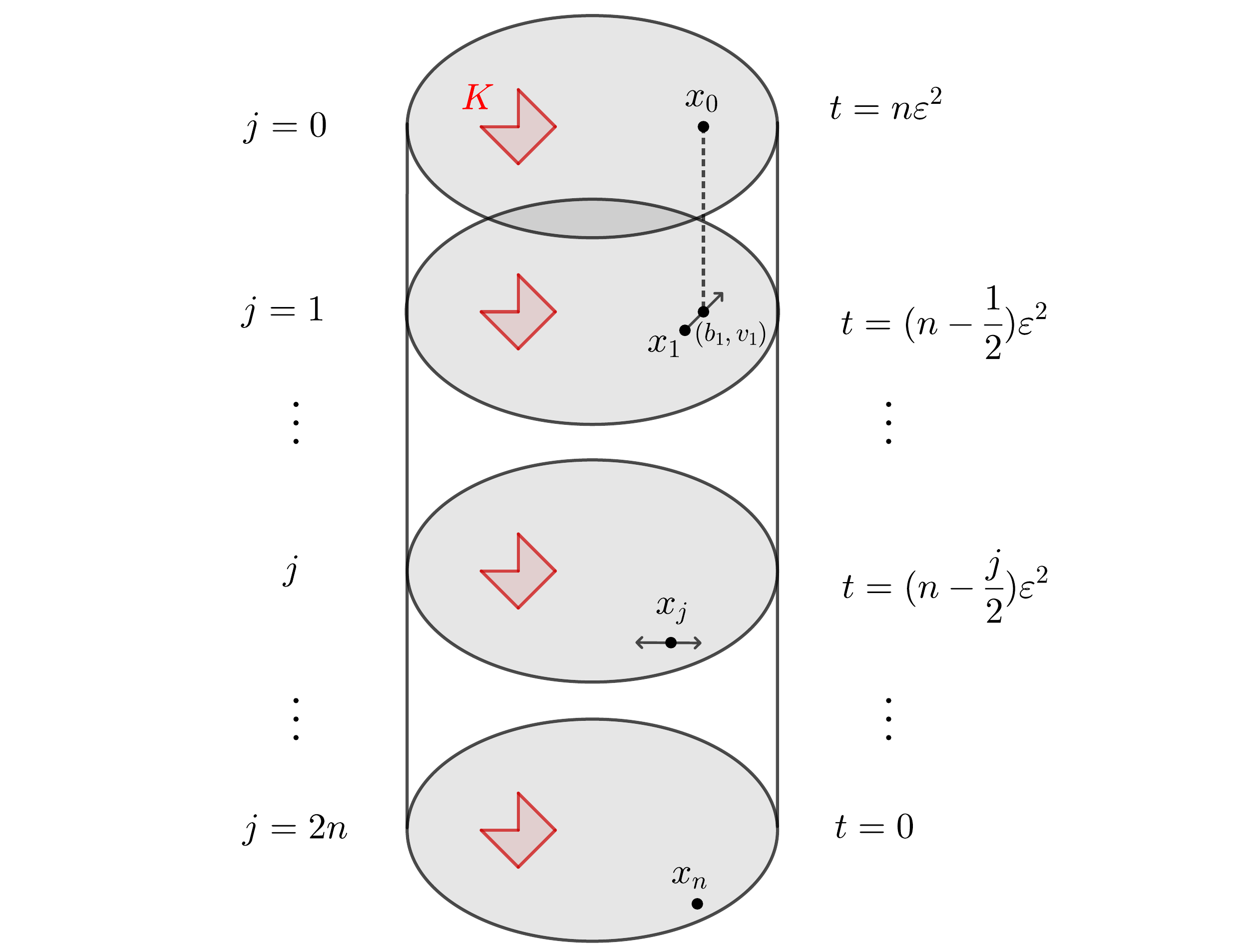}
	\caption{Positions playing the game.}
	\label{cilindroconamebas}
\end{figure}

We observe that Paul plays aiming to stay inside $\Omega_0$
and also trying to reach the obstacle $K_\eps$. Notice that the  continuous and bounded terminal payoff, $u_0$, is such that $u_0 (x) > 0$ for $x \in \Omega_0$ and $u_0(x) < 0$ for $x \in \mathbb{R}^N \setminus \overline{\Omega_0}$. Hence, Paul has a strong motivation to remain inside $\Omega_0$
when the game ends. On the other hand,
under this setting, Carol will try to stay outside $\Omega_0$
and plays her best in order to
evade the set $K_\eps$ at every move.

Now, to look for a partial differential equation related to this game, we argue formally (one of the main goals of this paper is to make rigorous what follows). 
Let us find the asymptotic behaviour as $\varepsilon \approx 0$ of the 
DPP, \eqref{DPP}, evaluated at a smooth function with an obstacle $\psi$ that is independent of 
$\varepsilon$ (we will have that $\psi_\eps \to \psi$ uniformly as $\eps\to 0$, hence we can assume here that $\psi$ does not depend on $\eps$). Assume that for a $C^{2,1}$ function $\phi$ we have
\begin{equation}
	\label{DPP-22} \phi (x,t) \approx \max \left\{\psi(x), \sup_{|v|=1} \min_{b = \pm 1} \phi (x + bv\varepsilon, t- \frac12 \varepsilon^2 ) \right\} ,
\end{equation}
that is, 
\begin{equation}
	\label{DPP-33} 0 \! \approx \! \max \left\{\psi(x) - \phi (x , t), \sup_{|v|=1} \min_{b = \pm 1} 
	\frac{\phi (x + bv\varepsilon, t- \frac12 \varepsilon^2 ) - \phi (x , t- \frac12\varepsilon^2 ) }{\varepsilon^2} - \frac{\phi (x,t) - \phi (x , t- \frac12 \varepsilon^2 )}{\varepsilon^2}\right\}.
\end{equation}
Using that $\phi \in C^{2,1}$ and neglecting higher order terms, after multiplying by 2 the second term, we arrive to 
\begin{equation}
	\label{DPP-44} 0 \approx \max \left\{\psi(x) - \phi (x , t), \sup_{|v|=1} \min_{b = \pm 1} 
	\frac{2}{\varepsilon} b \langle \nabla \phi (x,t), v \rangle 
	+ \langle D^2 \phi (x,t) v , v \rangle - \frac{\partial \phi}{\partial t} (x,t)
\right\} .
\end{equation}
Now, assume that $\phi(x,t)>\psi(x)$. In this case, we notice that when computing the supremum we want to use a vector $v$ orthogonal
to $\nabla \phi (x,t)$ (otherwise, we can select the right sign when we compute the minimum in order to obtain 
$-  2\varepsilon^{-1} |\langle \nabla \phi (x,t), v \rangle |$, that is negative and large for 
$\varepsilon \approx 0$). Then, assuming that the vector $v$ is orthogonal
to $\nabla \phi (x,t)$ we obtain 
\begin{equation}
	\label{DPP-55} 0
	 \approx \max \Big\{\psi(x)-\phi(x,t), \sup_{\substack{|v|=1, \\ v\perp \nabla  \phi (x,t)}}  
	  \langle D^2 \phi (x,t) v , v \rangle - \frac{\partial \phi}{\partial t} (x,t)
 \Big\} ,
\end{equation}
that is, we find the obstacle problem for the equation 
\begin{equation} \label{jj}
\frac{\partial \phi}{\partial t} (x,t) = \sup_{\substack{|v|=1, \\ v\perp  \nabla \phi (x,t)} } 
	 \langle D^2 \phi (x,t) v , v \rangle 
\end{equation}
with $\psi$ as obstacle. 
Notice that the right hand side of our equation \eqref{jj} is given by
the minimum of the curvatures of $\partial \Omega_t$,
$$
 \sup_{\substack{|v|=1, \\ v\perp \nabla \phi (x,t)}} \langle   D^2  \phi (x,t) v, v \rangle 
= -  \inf_{\substack{|v|=1, \\ v\perp \nabla  \phi (x,t) }} \Big(-\langle   D^2  \phi (x,t) v, v \rangle \Big)
= - \min_{\substack{|v|=1, \\ v \mbox{ tangent to } \partial \Omega_t }} \{ \kappa (v)\}.
$$
Then, we have that \eqref{jj} is just the level set formulation of the geometric evolution 
of a hypersurface by the minimum curvature with an obstacle.
Hence, we conclude that the DPP associated to the game, \eqref{DPP}, is related to the obstacle problem (with obstacle $\psi$) for the geometric flow of a hypersurface driven by the minimum curvature. That is, the limit as $\varepsilon \to 0$ for the value of the game, gets us back to
the parabolic problem \eqref{ParabolicPb}.
In fact, we prove that the value function of our game converges
 locally uniformly to 
the viscosity solution to \eqref{ParabolicPb} as $\varepsilon \to 0$. This is an alternative proof of existence of viscosity solutions for problem \eqref{ParabolicPb} that avoids the use of Perron's method.

For similar results for the classical mean curvature flow we refer to 
\cite{KS,Misu}. Again, we point out that, in 
dimension two, our game coincides with the one studied in \cite{KS}
(without the obstacle)
and in \cite{Misu} (with an obstacle). However, in dimensions
bigger than three the game presented here and 
the game for the mean curvature flow are quite different.

Since Paul aims to maximize the payoff and 
Carol wants to minimize it, we introduce the sets
$$
\Omega_t^{\varepsilon} = \{ x: u^\varepsilon (x,t) >0\},
$$
that is, the set of positions in $\mathbb{R}^N$ where starting at time $t$ Paul expects to win (and Carol to loose) a positive payoff 
after $\displaystyle \left\lceil \frac{2t}{\varepsilon^2} \right\rceil$ rounds. Notice that $
K_\varepsilon = \{ x : \psi_\varepsilon (x) >0\} \subset \Omega^\varepsilon_t$ for every $t$ (Paul can end the game at any round and obtain a positive
payoff inside the set $\{ x : \psi_\varepsilon (x) >0\}$).

As a consequence of the previous locally uniform convergence
of $u^\eps$ to $u$, we have that  for each $t>0$, the set where the value function of the $\varepsilon$-game is positive at time $t$, $\Omega_t^\eps = \{x : u^\eps (x,t)>0\}$ approximates as $\varepsilon \to 0$ the set where the solution of the obstacle problem is positive, $\Omega_t = \{x : u (x,t)>0\}$.

Finally, for $\eps$ fixed, we also study the shape of the sets $\Omega^\varepsilon_t$ for $t$ large. We find that the limit is related to the convex hull of the set $K_\eps$ with respect to a notion of
convexity that takes into account that we have a fixed particular length,
$\eps$, for the movements of the game. Taking the limit $\varepsilon \to 0$ yields the results for the asymptotic limit of $\Omega_t$ when $t \to \infty$ and in this way we show our main result, that $\Omega_t$ approaches the convex hull of $K$ as $t \to \infty$.

\subsection{Statements of the main results}

Let us state rigorously the main theorems that are included in this
paper. 

In Section \ref{sect.PE} we analyze 
the parabolic problem  \eqref{ParabolicPb} for which we will prove the following results:
First, a comparison principle holds for viscosity sub and supersolutions to 
 \eqref{ParabolicPb} (for the precise definition of being 
 a sub or a supersolution we refer to Section \ref{sect.PE}).

\begin{theorem}[Comparison Principle]\label{comp.pp}
	Assume that the obstacle $\psi$ is bounded and has a uniform modulus of continuity $\omega$, that is, $$|\psi(x)-\psi(y)|\leq \omega(|x-y|), \qquad \mbox{ for all } x, y \in \mathbb{R}^N$$
	for an increasing continuous function $w : [0,\infty) \to [0,\infty)$ such that $w(0)=0$. Also assume that the initial datum $u_{0}\in C(\mathbb{R}^{N})$ is such that  $$\lim_{|x|\to\infty}u_{0}(x)=\mu\in\mathbb{R}_{<0}$$ and $$u_0 (x)
	 \geq \psi(x)$$ for $x\in \mathbb{R}^N$. If $\underline{u}$ is a  viscosity subsolution and $\overline{u}$ a supersolution to problem
	 \eqref{ParabolicPb}, then $$\underline{u} (x,t)\leq \overline{u} (x,t) \qquad \mbox{ in }\mathbb{R}^{N}\times [0,\infty).$$
\end{theorem}

From now on and along the whole paper, we will assume that $u_{0}$ and $\psi$ 
satisfy the hypotheses in Theorem \ref{comp.pp}. We will omit the explicit mention to these hypotheses in the following statements.

With this comparison result we can show existence and uniqueness
of solutions to \eqref{ParabolicPb}.

\begin{theorem}[Existence and uniqueness for the PDE]\label{etheoremExistencePerron} There exists a unique solution $u$ to the Obstacle Problem \eqref{ParabolicPb} with obstacle $\psi$ and initial datum $u_{0}$.  Moreover, $u$ is a continuous function in $\mathbb{R}^{N}\times [0,\infty)$ that  is given by
	\begin{equation}
		u(x,t)=\inf\big\{ w(x,t)\,:\;\;w\in \mathcal{A}\big\}
	\end{equation}
	where 
	\begin{align}
		\mathcal{A}:=  \Big\{ & w(x,t)\;:\;\;w\;\;\textrm{is a viscosity supersolution to \eqref{ParabolicPb}, that is, }  \\ & w_{t}+\mathcal{L}w\geq 0\;\;\textrm{in}\;\;\mathbb{R}^{N}\times(0,\infty), \ w(\cdot, t)\geq \psi\;\;\textrm{for all }\;\;t>0\;\;\textrm{and}\;\;\;w(\cdot,0)\geq u_{0}\Big\}.
	\end{align}
\end{theorem} 

Now, let us describe our results for the game.
To this end, we first need to be precise in the way that we
enlarge the obstacle. Let us consider 
$K_\varepsilon = K + B_{N \varepsilon}(0)$. Notice that here we are using the dimension $N$ to enlarge the obstacle by adding a ball or radius $N \eps$. That the added ball has exactly radius $N \eps$ has a purpose later on, but let us mention that adding any ball $B_{r(\varepsilon)}(0)$ with
$r(\varepsilon)\geq N \eps$ and $r(\varepsilon) \to 0$ as $\eps \to 0$
is enough for our arguments. 
	We assume that $$K_\varepsilon = K + B_{N \varepsilon}(0) \subset \Omega_0.$$ This holds for $\eps $ small when $\overline{K} \subset \Omega_0$.
	Recall that $\psi$ is a continuous function such that 
	\[K= \{x\in \RR^N : \psi(x) > 0\}.\]
	We need to characterize the set $K_\varepsilon = K + B_{N \varepsilon} (0)$ in a similar way, with a function $\psi_\varepsilon$.
We want to choose a function $\psi_\varepsilon$ that is positive when $x$ is an interior point of $K_\varepsilon$, that is, 
	\[K_\varepsilon = \{x\in \RR^N : \psi_\varepsilon(x) > 0\},\]
	and such that, in addition,
	$\psi_\varepsilon \rightarrow \psi$ uniformly when $\varepsilon \rightarrow 0$. We use this function $\psi_\varepsilon$ as
	the obstacle for our game. To see that there is a function
	$\psi_\varepsilon$ that fulfills our requirements we argue as follows:
	We assumed that
	$\psi$ is bounded and has a uniform modulus of continuity $\omega$, that is, $$|\psi(x)-\psi(y)|\leq \omega(|x-y|), \qquad \mbox{ for all } x, y \in \mathbb{R}^N.$$
	Now, assuming that $\omega$ is continuous, let $h_\varepsilon$
	be given by 
	\[h_\varepsilon(x) = \left\{
	\begin{aligned}
		& \min_{y \in \partial K_\varepsilon} ( \psi (y) + 2 \omega (|x-y|) ) \qquad \text{ if } x\in K_\varepsilon, \\
		& \max_{y \in \partial K_\varepsilon} ( \psi (y) - 2\omega (|x-y|) ) \qquad \text{ if } x\in \RR^N \setminus K_\varepsilon .
	\end{aligned}
	\right.\]
	Notice that $h_\varepsilon$ is uniformly continuous, $h_\varepsilon(x) = \psi(x)$ for $x\in \partial K_{\varepsilon}$,
	$ \psi(x) \leq h_\varepsilon(x)$ for  
	$x\in K_{\varepsilon}$ and $h_\varepsilon (x) \leq \psi (x)$ for  
	$x\in \mathbb{R}^N \setminus K_{\varepsilon}$. 
	Also notice that there is a constant $C_\eps > 0$ that goes to zero as $\eps \to 0$ such that 
	$0<h_\varepsilon(x) - \psi (x)< C_\eps$ for $x\in K_\varepsilon
	\setminus K$ and 
	$0>h_\varepsilon(x) - \psi (x)> - C_\eps$ for $x\in K_{2\varepsilon}
	\setminus K_\eps$.
	Now, we choose 
	\begin{equation} \label{psi-eps}
	\psi_\varepsilon (x) :=
	 \left\{
	\begin{array}{ll}
	\displaystyle	\max \Big\{  \psi (x) , \min \{h_\varepsilon (x) - \psi (x), C_\eps \} \Big\}  \qquad & \text{ if } x\in K_\varepsilon, \\[7pt]
	\displaystyle	 \min \Big\{  \psi (x) , \max \{h_\varepsilon (x) - \psi (x), - C_\eps \} \Big\}  \qquad & \text{ if } x\in \RR^N \setminus K_\varepsilon .
	\end{array} \right.
	\end{equation}

 In Section \ref{sect.GameAndEq} we include the following results:

\begin{theorem}\label{maintheorem.juego.conv}
	Let $\psi_\varepsilon$
	be as above, and $u_{0} $ such that $u_0 \geq \psi_\varepsilon$ for every $\varepsilon$
	small. Let $\{u^{\varepsilon}\}_{\varepsilon>0}$  be a family of value functions of the game with obstacle $\psi_\varepsilon$
	and final payoff $u_0$,
	that is, for each $\varepsilon>0$,  $u^{\varepsilon}$ is a solution to \eqref{DPP}. 
	
	Then, $u^{\varepsilon}$ converges locally
	uniformly as $\varepsilon \to 0$ to $u$, the unique solution to \eqref{ParabolicPb}.
\end{theorem}

As a consequence, we can obtain the behaviour of
the positivity sets of $u^\varepsilon$ as $\varepsilon \to 0$.

\begin{corollary}\label{cor1} 
Consider  
	\begin{equation}
	\label{OmegaSets}	\Omega^{\varepsilon}_{t}:=\{x\,:\;\;u^{\varepsilon}(x,t)>0\}\;\;\textrm{and}\;\;\Omega_{t}:=\{x\,:\;\;u(x,t)>0\}.
	\end{equation}Then, for each $t>0$, we have that 
	\begin{equation*}
		\Omega_t\subset \liminf_{\varepsilon \to 0} \Omega^\varepsilon_t \subset \limsup_{\varepsilon \to 0} \Omega^\varepsilon_t\subset \overline{\Omega}_t.
	\end{equation*}
\end{corollary}

Next, we drive our attention to the behaviour of the positivity
sets of $u^\varepsilon$ as $t \to \infty$. 
We need to introduce a graph $G=(\mathcal{V},\mathcal{E})$
associated with the geometric configuration 
of $K$ and $\Omega_0$. The vertices of the graph $G$
are the connected components of $K$ and two vertices are
connected by an edge if there is a segment in $\mathbb{R}^N$ with endpoints
in two connected components of $K$ that correspond to the two vertices, that is contained in $\Omega_0$.  
The main hypothesis in the next results is going to be the following:
\begin{equation} \tag{G} \label{(G)}
\mbox{The graph $G$ is connected.}
\end{equation} 
Notice that \eqref{(G)} holds when we have the stronger
condition $\textrm{co}(K) \subset \Omega_0$.

In the next statement recall that we denoted by $\textrm{co}(A)$ the convex hull of a set $A$.
In Section \ref{sect.LongTimeBehaviour} we prove:

 \begin{theorem}\label{Th.Cadenas}
 Assume that $K\subset \mathbb{R}^{N}$ and $\Omega_0\subset \mathbb{R}^{N}$ satisfy condition \eqref{(G)}. If $\text{co}(K) \subset \Omega_0$ or $N=2$,  then there exists $\varepsilon_{0}>0$ such that 
 	\begin{equation}
 		co(K)\subset \liminf\limits_{t\to\infty}\Omega_{t}^{\varepsilon}\subset\limsup\limits_{t\to\infty}\Omega_{t}^{\varepsilon}\subset co(K_{\varepsilon})\subset co(K)+B_{N\varepsilon} (0)
 	\end{equation}
	for every $0<\varepsilon\leq \varepsilon_{0}$.
 \end{theorem}
 
 In the proof of the previous result is where we use that we enlarged
 the obstacle and considered the set $K_\varepsilon = K + B_{N \varepsilon}(0)$
and its associated function $\psi_\varepsilon$ as the obstacle 
for our game. Notice that we have $\mbox{co}(K)\subset \liminf\limits_{t\to\infty}\Omega_{t}^{\varepsilon}$ and
$\limsup\limits_{t\to\infty}\Omega_{t}^{\varepsilon}\subset \mbox{co}(K_{\varepsilon})$ when we play with $K+B_{N\varepsilon}(0)$
and its associated function $\psi_\varepsilon$ as obstacle, but, in general, these inclusions are not true if we play with $K$ and $\psi$ as the obstacle. We can give a slightly better result when we assume that $\textrm{co}(K) \subset \Omega_0$, see Section \ref{sect.LongTimeBehaviour}.

Finally, we use our previous results and some extra arguments to recover exactly  the 
convex hull of $K$ when we take the limit as $t\to \infty$ of the geometric evolution by minimal curvature. The following result is valid for any dimension.

 \begin{theorem} \label{cor2} 
 Assume the $K$ and $\Omega_0$ satisfy condition \eqref{(G)}. The family of positivity sets of the unique viscosity solution $u(x,t)$ to the obstacle problem~\eqref{ParabolicPb}, $(\Omega_{t})_{t>0}$ defined in~\eqref{OmegaSets},
 verifies
 \begin{equation}
	\text{co}(K) = \lim_{t\to \infty} \Omega_{t}.
 \end{equation}
\end{theorem}

\subsection{An alternative game}

Let us describe an alternative game. As before, we have two players,
Paul (who wants to maximize the final payoff) and Carol
(who wants to minimize it). Again as before, the final payoff is given by 
$u_0$ and an obstacle $\psi_ \eps$.

Now, the rules of the game change. 
At each turn, Carol chooses a subspace $S_i$ of dimension $N-1$, then Paul chooses a unit vector in that subspace, $v_i \in S$, $|v_i|=1$,
and the new position of the game is $$x_{i} = x_{i-1} + \eps v_i$$
and 
time decreases as before,
\[t_i = t_{i-1}- \frac12 \varepsilon^2.\]
Again, we have the stopping rule and the terminal rule. Paul has the right to end the game at any turn and he gets the value given by the obstacle, $\psi_\varepsilon(x_i)$ and if time is consumed (after
$n = \displaystyle \left\lceil \frac{2 t_0}{\varepsilon^2} \right\rceil$ rounds) Paul gets the final payoff $u_0(x_n)$.

Let $u^\varepsilon (x,t)$ be the value of this game starting at $x$ with time $t$. For this game the corresponding  
Dynamic Programming Principle (DPP) is given by
\begin{equation}
	\label{DPP-alternativo} \tag{$P_{\varepsilon}$}
	\left\{
	\begin{array}{ll}
	\displaystyle
	u^\varepsilon(x,t) = \max \Big\{\psi_\varepsilon(x), 
	\inf_{\substack{S\subset \mathbb{R}^N, \\ dim(S) = N-1}}
	\sup_{\substack{v \in S, \\ |v|=1}}  u^\varepsilon \big(x + v\varepsilon, t- \frac12\varepsilon^2 \big) \Big\}, \qquad & x\in \mathbb{R}^N, \ t>0,\\[8pt]
	u^\varepsilon(x,t) = u_0 (x), \qquad & x\in \mathbb{R}^N, \ t\leq 0.
	\end{array} \right.
\end{equation}

The partial differential equation related to this game is exactly the
level set formulation of the geometric evolution 
of a hypersurface by the minimum curvature with an obstacle.
To see this, formally, we find the asymptotic behaviour as $\varepsilon \approx 0$ of the 
DPP, \eqref{DPP-alternativo}, evaluated at a smooth function 
$\phi$ assuming again that the obstacle $\psi$ is independent of 
$\varepsilon$. We have
\begin{equation}
	\label{DPP-22-alter}
	\phi (x,t) \approx
	\max \Big\{\psi_\varepsilon(x), 
	\inf_{\substack{S\subset \mathbb{R}^N, \\ dim(S) = N-1}}
	\sup_{\substack{v \in S, \\ |v|=1}}  \phi \big(x + v\varepsilon, t- \frac12\varepsilon^2 \big) \Big\}.
\end{equation}
Using similar Taylor expansions as we did before 
we get
\begin{equation}
	\label{DPP-44-alter} 0 \approx \max \Big\{\psi(x) - \phi (x , t), 
	\inf_{\substack{S\subset \mathbb{R}^N, \\ dim(S) = N-1}}
	\sup_{\substack{v \in S, \\ |v|=1}}
	\frac{2}{\varepsilon} \langle \nabla \phi (x,t), v \rangle 
	+ \langle D^2 \phi (x,t) v , v \rangle - \frac{\partial \phi}{\partial t} (x,t)
\Big\} .
\end{equation}
Now, in the second term, we notice that when computing the infimum we  use a subspace $S$ orthogonal
to $\nabla \phi (x,t)$ (otherwise, we can select a unit vector $v \in S$ such that 
$2\varepsilon^{-1} \langle \nabla \phi (x,t), v \rangle$ is positive and large for 
$\varepsilon \approx 0$). Then, assuming that the subspace $S$ is orthogonal
to $\nabla \phi (x,t)$ we arrive to 
\begin{equation}
	\label{DPP-55-alter} 
	0 \approx \max \Big\{\psi(x)-\phi(x,t), \sup_{\substack{|v|=1, \\ v\perp \nabla  \phi (x,t)}}  
	  \langle D^2 \phi (x,t) v , v \rangle - \frac{\partial \phi}{\partial t} (x,t)
 \Big\} ,
\end{equation}
that is, we find the obstacle problem for the level set formulation of the geometric evolution 
by the minimum curvature with an obstacle.

For this game, one can show
(using similar computations to the ones gathered in
Section \ref{sect.GameAndEq}) that the limit as $\varepsilon \to 0$ for the value of the game, gives the solution to 
the parabolic problem \eqref{ParabolicPb}.

We have chosen to include 
the details of the proofs of our results concerning
the asymptotic behaviour of the positivity set for the value function 
for our 
previous game and leave
to the reader the analysis of this alternative game.

\subsection{Comments on the ideas involved in the proofs and on our hypotheses}

The idea of considering the flow  by minimal curvature with an 
obstacle in order to recover the convex hull of the obstacle is natural
since the boundary of the convex hull of a set can be touched by
hyperplanes and when these hyperplanes touch the convex hull at points that are not in the original set we obtain that one of the main curvatures
of the boundary has to be zero at those points, while all the other curvatures need to be nonnegative in order for the convex hull to be a convex set. 
Therefore, we have that the convex hull of a set is a solution to 
the minimal curvature equal to zero at boundary points that are not in the closure of the original set (where the surface does not touch the obstacle). Since our 
target (the convex hull of a set) is a stationary solution to our evolution
equation it seems natural to expect that we find the convex hull in the
long time behaviour of the solutions to our parabolic problem. 

We use viscosity theory to deal with the parabolic problem \eqref{ParabolicPb}
using ideas from \cite{Mercier,19}. Here we assume that 
the obstacle $\psi$ has a uniform modulus of continuity
and that the initial condition is continuous and bounded with
a limit as $|x|\to \infty$. These assumptions allow us to use
a doubling of variables technique in order to show the comparison
principle for our problem.  Once we have a comparison principle, 
existence and uniqueness follow from Perron's method. 

The analysis of the game relies on ideas from \cite{KS,Misu}. 
However, here we enlarge the  
obstacle set and we introduce the idea of $\eps$-convexity in order to study the long time behaviour of the solutions. These 
modifications of the arguments in \cite{KS,Misu} prove to be quite useful
for the analysis of the asymptotic behaviour of the positivity sets, but introduce new difficulties in the analysis. 

The proof of the convergence of the values of the game $u^\eps$ to the unique viscosity solution to \eqref{ParabolicPb} is based on the half-relaxed limits. Some care needs to be taken here since taking these limits we want to obtain
sub and supersolutions to the degenerate parabolic equation \eqref{PE}.

Concerning the study of the asymptotic behaviour of the solutions 
to \eqref{ParabolicPb} and the values of the game we borrow some 
geometric ideas from \cite{Misu} and use that the positivity set 
of the initial condition $u_0$ contains some segments that connect
the connected components of the obstacle $K$ (condition \eqref{(G)}). Without this assumption the general asymptotic result may be false. For example, in the case where there are different connected components of $K$ contained in different connected components of $\Omega_0$ the positivity set of a solution to the evolution problem may remain disconnected and  we can only recover in the limit the union of the convex hulls of the components that are connected. The study of the long time behaviour of the 
positivity set of the value function of the game relies on the construction
of clever strategies for Paul and Carol that allow us to obtain lower an upper bounds on the sets where Paul wins (that is, for the set where
the value of the game is strictly positive). 

We want to highlight that our hypothesis that the obstacle set $K$ is open can be relaxed at some points, in particular when we have that $\textrm{co}(K) \subset \Omega_0$. We will make further comments on this along the text when appropriate.

To see the utility of enlarging the obstacle considering 
$K_\eps = K + B_{N \eps}(0)$ and that we can deal with sets $K$ that are not necessarily open, let us describe a concrete example.
Assume that we want to recover the convex hull of only two
points, that is, we take $K=\{x,y\}$. In this simple case the
convex hull of $K$ is just the segment that connects $x$ and $y$.
Assume that we are in the most favorable case, the positivity set of
the initial condition contains the segment and also assume that the initial position
$x_0$
is inside the segment.
In this situation, since Paul wins when the position of the game 
arrives to the obstacle or remains in the segment, starting at a point inside the segment 
Paul will choose that direction to play. However, if $\eps$ is such that
the starting point $x_0$ is inside the segment but is such that $|x-x_0|/\eps$ or $|y-x_0|/\eps$
is not a natural number, then Carol may choose the signs in such a way
that the position of the game for large times exits any large ball without touching the
obstacle (and then Paul loses). See Figure \ref{dospuntosFig}.
Therefore, we get that $\liminf\limits_{t\to\infty}\Omega_{t}^{\varepsilon}=
K=\{x,y\}$ and in this limit we do not recover the convex hull of $K$ (the segment).

Nevertheless, when we enlarge the set $K$, taking $K_\eps = B_{N\eps} (x) \cup B_{N\eps} (y)$, it is simple to observe that 
when $x_0$ is inside the segment that connects $x$ and $y$ and Paul chooses that direction, then no matter what Carol does the position will
reach the obstacle or remain in the segment and hence Paul wins. 
See Figure \ref{dospuntosFig}.
In this case we have that the segment $\mbox{co} (K)$ is included in $\liminf\limits_{t\to\infty}\Omega_{t}^{\varepsilon}$ as we wanted.
 
 \begin{figure}[H]
	\centering
	\includegraphics[width=8cm]{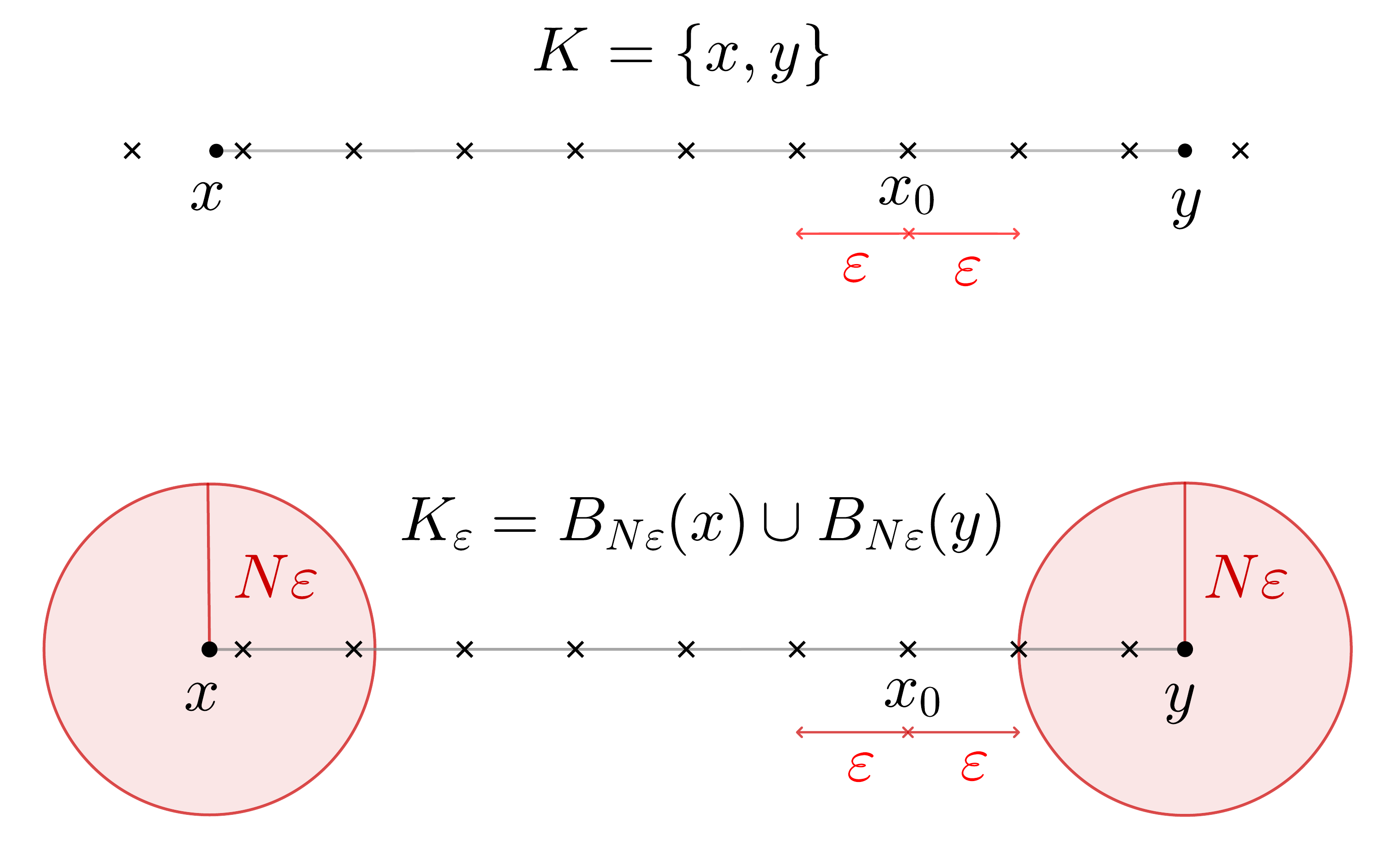}
	\caption{Positions of the game with or without enlarging the obstacle.}
	\label{dospuntosFig}
\end{figure}

\subsection{Precedents} 
Let us end this introduction with a brief description of previous results in the literature. 

Our immediate precedent and source of inspiration is \cite{Misu} where the
author studies the mean curvature flow with an obstacle. In two space dimensions the mean curvature flow is the well known curve shortening problem and the curve evolution was analyzed in 
\cite{5,10} and references therein. Without obstacles, it is known that the curve becomes convex at some time and then it converges to a single point in finite time, provided
the initial curve is closed and smooth (\cite{5,10}).
Since there is only one curvature for a curve in 
dimension two our evolution problem coincides with the one studied in \cite{Misu}. However, for dimensions bigger than two
the two problems are different in nature and, as we explained before, the evolution by minimal curvature studied here recovers the convex hull of the obstacle in the limit as $t \to \infty$; while the mean curvature flow does not.

The level set method for surface evolution equations was first rigorously analyzed in 
\cite{CGG,ES}. The level set method is applied to various surface evolution equations including the mean curvature flow equation
(with or without an obstacle), see \cite{Almeida,CGG,ES}. For general references on 
level sets formulation of geometric flows we refer to \cite{2,6,8,9,15,16,19} and the book \cite{Giga}.

On the other hand, the relation between game theory and nonlinear PDEs 
is quite rich and has attracted considerable attention in recent years. We quote \cite{BR,KS,KS2,Luiro,MPR,MPR2,ML,MiRo,Misu,PSSW,PS,R} and the books \cite{BRLibro,Lewicka}.
As we have mentioned, in \cite{KS} a game for the motion by mean curvature was studied and in \cite{Misu} the same game with an obstacle was tackled. We use some ideas for the game in dimension two that were obtained in \cite{Misu}. In particular, we extend the geometric condition on the geometric configuration of the obstacle $K$ and the initial set $\Omega_0$ from \cite{Misu} to several space dimensions.
We refer to \cite{17,ML} for games for obstacle problems and to \cite{MiRo}
for a game for the two membranes problem.

\section{The level set formulation for the movement of a hypersurface by minimal curvature with an obstacle} \label{sect.PE}

In this section we gather the results concerning the evolution problem
\eqref{ParabolicPb}. First, we need to be precise in the 
definition of what we understand by a viscosity sub and supersolution. We recall the definition of semicontinuous functions, since they are useful when working in the viscosity framework.   

\begin{definition}\label{defEnvelope}
	We say that $f:\mathbb{R}^{N}\times (0,\infty)\mapsto \RR$ is a \emph{lower semicontinuous function}, l.s.c.,  at $(x,t)\in\mathbb{R}^{N}\times (0,\infty) $ if for each $\varepsilon>0$ there exists $\delta>0$ such that 
	\begin{equation*}
		f(x,t)\leq f(y,s)+\varepsilon\;\;\textrm{for all}\;\;(y,s)\in B_{\delta}(x,t)\setminus \{(x,t)\}, 
	\end{equation*}
	or, equivalently, 
	\begin{equation}
		f(x,t) \leq \liminf_{(y,s) \to (x,t) } f(y,s).
	\end{equation}
	If a function $f$ is l.s.c. at every $(x,t)\in\mathbb{R}^{N}\times (0,\infty)$, we say that $f$ is l.s.c. in $\mathbb{R}^{N}\times (0,\infty)$. 
	
	On the other hand, we say that
	$g:\mathbb{R}^{N}\times (0,\infty)\mapsto \RR$ is an \emph{upper semicontinuous function}, u.s.c,  at $(x,t)\in\mathbb{R}^{N}\times (0,\infty)$ if for each $\varepsilon>0$ there exists $\delta>0$ such that 
	\begin{equation*}
		g(y,s)\leq g(x,t)+\varepsilon\;\;\textrm{for all}\;\;(y,s)\in B_{\delta}(x,t)\setminus \{(x,t)\},
	\end{equation*}
	or, equivalently, 
	\begin{equation}
		\limsup_{(y,s) \to (x,t) } f(y,s) \leq f(x,t).
	\end{equation}
	If a functions $g$ is u.s.c. at every $(x,t)\in\mathbb{R}^{N}\times (0,\infty)$, we say that it is u.s.c. in $\mathbb{R}^{N}\times (0,\infty).$
\end{definition}
	
We will also require the concepts of lower and upper semicontinuous envelopes of a function. 

\begin{definition}
		The \emph{lower semicontinuous envelope} of $f$ is defined as
	\begin{equation}
		f_{*}(x,t)=\liminf\limits_{(y,s)\to(x,t)}f(y,s).
	\end{equation}
	It verifies that is the largest lower semicontinuous function smaller or equal than $f$, i.e, 
	\begin{equation*}
		f_{*}=\sup \Big\{h:\mathbb{R}^{n}\times (0,\infty)\mapsto \RR\;\;\textrm{l.s.c.}\,:\;\;h(x,t)\leq f(x,t)\;\;\textrm{for all}\;\;(x,t)\in\mathbb{R}^{N}\times (0,\infty)\Big\}.
	\end{equation*}
	
	The \emph{upper semicontinuous envelope} of $f$ is
	\begin{equation}
		f^{*}(x,t)=\limsup\limits_{(y,s)\to(x,t)}f(y,s).
	\end{equation}
	It verifies that is the smallest upper semicontinuous function larger or equal than $f$, i.e, 
	\begin{equation*}
		g^{*}=\inf \Big\{h:\mathbb{R}^{N}\times (0,\infty)\mapsto \RR\;\;\textrm{u.s.c.}\,:\;\;h(x,t)\geq f(x,t)\;\;\textrm{for all}\;\;(x,t)\in\mathbb{R}^{N}\times (0,\infty)\Big\}.
	\end{equation*}
\end{definition}
	
	Now, to use viscosity theory for our operator
	\begin{equation}\label{1.77} 
    \mathcal{L}(f(x)) :=  \inf_{\substack{|v|=1 \\ v\perp \nabla f(x)}} \Big(-\langle   D^2f(x) v, v \rangle \Big).
\end{equation}
we look at it using the function $ \mathcal{L} : \mathbb{R}^N \setminus\{0\}\times 
\mathbb{S}^{N\times N} \mapsto \mathbb{R}$ 
(here $\mathbb{S}^{N\times N}$ denotes the set of real symmetric matrices) 
given by
\begin{equation}\label{1.88} 
    \mathcal{L}(p, X) :=  \inf_{\substack{|v|=1 \\ v\perp p}} \Big(-\langle  X v, v \rangle \Big).
\end{equation} 
Observe that, with an abuse of notation, we use $\mathcal{L}$ to denote both the operator and the function associated with it. 
When $p=0$, we take the infimum over all the unitary vectors. The function $\mathcal{L}$ is discontinuous at $p=0$. Therefore, when we define sub and supersolutions we will need to use $\mathcal{L}_{*}$, the lower semicontinuous envelope and $\mathcal{L}^{*}$, the upper semicontinuous envelope of~$\mathcal{L}$.

Now, let us state the definition of being a viscosity subsolution and
a viscosity supersolution (we refer to \cite{CIL} for more details on
viscosity theory).  

\begin{definition}\label{edefvisc1}
	Let $\mathcal{D}\subset\mathbb{R}^{N}\times (0,\infty)$ be an open set.  A function $u:\mathbb{R}^{N}\times (0,\infty) \rightarrow \mathbb{R}$ is called a \emph{viscosity subsolution} of the parabolic equation \eqref{PE} in $\mathcal{D}$ if its upper semicontinuous envelope, $u^{*}$, satisfies that  for every $\phi\in C^{2,1}(\mathbb{R}^{N}\times (0,\infty))$ such that $\phi$ touches $u^{*}$ at $(x,t)\in \mathcal{D}$ strictly from above, that is, $\phi-u^{*}$ has a strict minimum at $(x,t)$, we have,
	\begin{equation}
		\partial_t \phi(x,t)  + \mathcal{L}_* (\nabla \phi(x,t),
		D^2 \phi (x,t)) \leq 0.
	\end{equation}
	Conversely, $u:\mathbb{R}^{N}\times (0,\infty) \rightarrow \mathbb{R}$ is called a \emph{viscosity supersolution} of the parabolic equation \eqref{PE} in $\mathcal{D}$ 
	if its lower semicontinuous envelope function, $u_{*}$, satisfies that
	for every $\phi\in C^{2,1}(\mathbb{R}^{N}\times (0,\infty))$ such that $\phi$ touches $u_{*}$ at $(x,t)\in \mathcal{D}$ strictly from below, that is, $u_{*}-\phi$ has a strict minimum at $(x,t)$, we have 
	\begin{equation}
		\partial_t \phi(x,t)  +\mathcal{L}^* (\nabla \phi(x,t),
		D^2 \phi (x,t)) \geq 0.
	\end{equation}
	If $u$ is both a sub and supersolution of \eqref{PE}, we say that $u$ is a viscosity solution to \eqref{PE}.
\end{definition}

To motivate the following equivalent notion of viscosity solution, suppose that $u$ is a viscosity subsolution to the equation \eqref{PE}. Assume that $\phi\in C^{2,1}(\mathbb{R}^{N}\times (0,\infty))$ is such that $\phi-u^{*}$ has a strict minimum at $(x,t)$. Then,
\begin{equation}
	\phi(x,t)-u^{*}(x,t)<\phi(y,s)-u^{*}(y,s)\;\;\textrm{for all}\;\;(y,s)\in\mathbb{R}^{N}\times (0,\infty).
\end{equation}
With this in mind, using the Taylor's expansion of $\phi(y,s)$ on $(x,t)$ we get
\begin{equation}
	\label{e6}	u^{*}(y,s)\leq u^{*}(x,t)+a(s-t)+\langle p,y-x\rangle+\frac{1}{2}\langle X(y-x),y-x\rangle+ o(|t-s|+|y-x|^{2})\;\;\textrm{as}\;(y,s)\to(x,s)
\end{equation}
where 
\begin{equation*}
	a=\partial_t \phi(x,t),\;\;p=\nabla \phi (x,t)\;\;\textrm{and}\;\;X=D^{2}\phi(x,t).
\end{equation*} 
Moreover, if $ u\in C^{2,1}(\mathbb{R}^{N}\times (0,\infty))$ and \eqref{e6} holds for some 
$(p,X)\in\mathbb{R}^{N}\times \mathbb{S}^{N\times N}$, then 
\begin{equation*}
	a=\partial_t u(x,t),\;\;p=\nabla u(x,t)\;\;\textrm{and}\;\;X=D^{2}u(x,t).
\end{equation*}

\begin{definition} \label{defsuperjet}
	Consider a function $v:\mathbb{R}^{N}\times (0,\infty)\mapsto \mathbb{R}$. The parabolic super 2-jet of $v$ at $(x,t)$, called $ \mathcal{P}^{2,+}v(x,t)$, is the set of all $(a,p,X)\in \mathbb{R}\times \mathbb{R}^{N}\times \mathbb{S}^{N\times N}$ such that 
	\begin{equation*}
		v(y,s)\leq v(x,t)+a(s-t)+\langle p, y-z\rangle+\frac{1}{2}\langle X(y-x),y-x\rangle+\textrm{o}(|s-t|+|y-x|^{2})\;\;\textrm{as}\;\;(y,s)\to (x,t).
	\end{equation*}
	Similarly, the parabolic lower super 2-jet of $v$ at $(x,t)$, that we call $ \mathcal{P}^{2,-}v(x,t)$, is given by 
	\begin{equation*}
		\mathcal{P}^{2,-}v(x,t)= -\mathcal{P}^{2,+}(-v)(x,t).
	\end{equation*}
\end{definition}

Furthermore, as explained in \cite{CIL}, if $w:\mathbb{R}^{N}\times (0,\infty)\mapsto \mathbb{R}$ is an upper semicontinuous function, there is a correspondence between the tuples $(a,p,X)\in	\mathcal{P}^{2,-}w(x,t)$ and the test functions that touch $w$ at $(x,t)$ strictly from above: $(a,p,X)\in \mathcal{P}^{2,-}w(x,t)$ if and only if there exists a $\phi\in C^{2,1}(\mathbb{R}^{N}\times(0,\infty)$ such that $(a,p,X)=(\partial_t \phi (x,t), \nabla\phi(x,t),D^{2}\phi (x,t))$ and $\phi$ touches $w$ at $(x,t)$ strictly form above. That is, the parabolic upper super 2-jet of the upper semicontinuous function $w$ at $(x,t)$ can be expressed as 
\begin{equation*}
	\mathcal{P}^{2,+}w(x,t)=\Big\{ (\partial_t \phi (x,t), \nabla\phi(x,t),D^{2}\phi (x,t))\,:\;\;
	\begin{array}{l}
	\displaystyle 
	\phi\in C^{2}(\mathbb{R}^{n}\times(0,\infty))\;\;\textrm{and}\\ 
	\displaystyle \phi-w\;\;\textrm{has an strict minimun at}\;\;(x,t)
	\end{array}
	\Big \}.
\end{equation*}

With the aim of applying  \cite[Theorem 8.2 ]{CIL} in the proof of the comparison principle, we reformulate the definition of viscosity super and subsolutions in terms of inequalities that involves elements of parabolic super and lower 2-jets of a function.

\begin{definition}
	\label{edefvisc2}Let $\mathcal{D}$ be  non-empty subset of $\mathbb{R}^{N}\times (0,\infty)$. Let $\mathcal{L}$ be defined as before. Consider the equation
	\begin{equation}
		\label{e5}	w_{t}+\mathcal{L}(\nabla w, D^{2}w)=0\;\;\text{in}\;\;\mathcal{D} .
	\end{equation}
	The function $u:\mathbb{R}^{N}\times (0,\infty)\mapsto \mathbb{R}$ is called a viscosity subsolution of the above equation \eqref{e5} in $\mathcal{D}$
	if its upper semicontinuous envelope, $u^{*}$, satisfies that 
	\begin{equation}
		a+\mathcal{L}_{*}(p,X)\leq 0\;\;\textrm{for all}\;\;(a,p,X)\in \mathcal{P}^{2,+}u^{*}(x,t).
	\end{equation}
	Reciprocally, we say that $u$ is a  viscosity supersolution of \eqref{e5} in $\mathcal{D}$ if its lower semicontinuous envelope, $u_{*}$, satisfies that	
	\begin{equation}
		a+\mathcal{L}^{*}(p,X)\geq 0\;\;\textrm{for all}\;\;(a,p,X)\in \mathcal{P}^{2,-}u_{*}(x,t).
	\end{equation}
	  Moreover, if $u$ is both a sub and supersolution to \eqref{e5}, we say that $u$ is a viscosity solution to \eqref{e5} in $\mathcal{D}$.
\end{definition}

Next we will prove some important properties of $\mathcal{L}$. Subsequently, as a consequence, we will establish the equivalence 
between  Definition \ref{edefvisc1} and Definition \ref{edefvisc2}.

\begin{proposition}\label{eprop1}
	The function $\mathcal{L}$ satisifies the following properties:
	\begin{itemize}
		\item[\emph{i)}] $\mathcal{L}$ is continuous in $\big(\mathbb{R}^{N}\setminus\{0\}\big)\times \mathbb{S}^{N\times N}$,
		
		\medskip
		
		\item[\emph{ii)}] $\lambda(-X)\leq \mathcal{L}_{*}(p,X)\leq \mathcal{L}^{*}(p,X)\leq \Lambda(-X)$ for all $(p,X)\in\mathbb{R}^{N}\times 
		\mathbb{S}^{N\times N}$
		where \begin{equation}
			\label{lambda}	\lambda(Y)=\min\Big\{\mu\in\mathbb{R}\,;\;\;\mu\;\;\textrm{eigenvalue of }\;\;Y\Big\}=\min\limits_{|v|=1}\,\langle Yv,v\rangle
		\end{equation}
		and 
		\begin{equation}
			\label{Lambda}	\Lambda(Y)=\max\Big\{\mu\in\mathbb{R}\,;\;\;\mu\;\;\textrm{eigenvalue of }\;\;Y\Big\}=\max\limits_{|v|=1}\,\langle Yv,v\rangle
		\end{equation}
		for $Y\in \mathbb{S}^{N\times N}$.
		
			\medskip
			
		\item[\emph{iii)}] $\mathcal{L}_{*}(0,O)=\mathcal{L}^{*}(0,O)=0$ where $O$ is the null matrix in $\mathbb{S}^{N\times N}$.
		
			\medskip
			
		\item[\emph{iv)}] $\mathcal{L}$ is degenerate elliptic, i.e., $$\mathcal{L}(p, Y)\leq \mathcal{L}(p,X)$$ whenever $X\leq Y$ and $p\not =0$.
		
			\medskip
			
		\item[\emph{v)}] $\mathcal{L}$ is geometric, i.e, for all $\alpha>0$ and $\sigma\in\mathbb{R}^{N}$, $$\mathcal{L}(\alpha p,\alpha X+\sigma p\otimes p )=\alpha \mathcal{L}(p,X)$$ for all $(p, X)\in \big(\mathbb{R}^{N}\setminus\{0\}\big)\times \mathbb{S}^{N\times N}$.
	\end{itemize}
\end{proposition}

\begin{proof}
	We begin with the proof of \emph{i)}. Let $p_{1},p_{2}\in \mathbb{R}^{N}\setminus\{0\}$ and $X_{1},X_{2}\in \mathbb{S}^{N\times N}$. Notice that
	\begin{equation}
		|\mathcal{L}(p_{1},X_{1})-\mathcal{L}(p_{2},X_{2})|\leq |\mathcal{L}(p_{1},X_{1})-\mathcal{L}(p_{1},X_{2})|+|\mathcal{L}(p_{1},X_{2})-\mathcal{L}(p_{2},X_{2})|
	\end{equation}
	and 
	\begin{equation}
		|\mathcal{L}(p_{1},X_{1})-\mathcal{L}(p_{1},X_{2})|\leq \sup_{\substack{|v|=1 \\ v \perp p_{1}}} |\langle (X_{1}-X_{2})v,v\rangle|\leq \|X_{1}-X_{2}\|.
	\end{equation}
	Then, to prove that $\mathcal{L}$ is continuous in $\big(\mathbb{R}^{N}\setminus\{0\}\big)\times \mathbb{S}^{N\times N}$, it remains to prove that $\mathcal{L}(\cdot, X)$ is continuous in $\mathbb{R}^{N}\setminus\{0\}$ for every $X\in \mathbb{S}^{N\times N}$.
	We fix $(p,X)\in\big(\mathbb{R}^{N}\setminus\{0\}\big)\times 
	\mathbb{S}^{N\times N}$. We have to show that
	\begin{equation}
		\mathcal{L} (p,X)=\lim\limits_{q\to p} \mathcal{L} (q,X).
	\end{equation}
	Let us prove first that $\mathcal{L}$ is lower semicontinuous in $(p,X)\in\big(\mathbb{R}^{N}\setminus\{0\}\big)\times 
	\mathbb{S}^{N\times N}$ on the first variable, that is,
\begin{equation}  \label{lislower}
		\liminf_{q\to p}\mathcal{L}(q,X)\geq \mathcal{L}(p,X).
\end{equation}
	We argue by contradiction. Suppose that there is a $\delta>0$ such that 
	\begin{equation}
		\liminf_{q\to p}\mathcal{L}(q,X)+\delta \leq \mathcal{L}(p,X).
	\end{equation}
	Note that, due to the definition of $\mathcal{L}$, $\mathcal{L}(q,X)\geq \lambda(-X)$ for all $q\not =0$ with $\lambda$ given by \eqref{lambda}. Therefore, there is a sequence of vectors $\{q_{n}\}_{n\in\mathbb{N}}$ with $|q_{n}|=1$ for every $n\in\mathbb{N}$ such that $q_{n}\to p$ and
	\begin{equation}
		\mathcal{L} (q_{n},X)\to \liminf_{q\to p} \mathcal{L}(q,X)
	\end{equation} 
	as $n\to \infty$. Then, there exists $n_{0}\in\mathbb{N}$ such that 
	\begin{equation}
		\mathcal{L}(q_{n},X)-\frac{\delta}{2}\leq \liminf_{q\to p}\mathcal{L}(q,X)\qquad \;\textrm{for all}\;\;n\geq n_{0}.
	\end{equation}
	Putting all together, and using the definition of $\mathcal{L}$, we have
	\begin{equation}
		\label{e2}	\sup_{\substack{|v|=1 \\ v \perp p}} \langle Xv,v\rangle +\frac{\delta}{2}\leq \sup_{\substack{|v|=1 \\ v \perp q_{n}}} \langle Xv,v\rangle \qquad \;\textrm{for all}\;\;n\geq n_{0}.
	\end{equation}
	Moreover, since $$\sup_{\{|v|=1 \;\;v \perp q_{n}\}} \langle Xv,v\rangle \leq \Lambda(X)$$ and the set $\{v\in \mathbb{S}^{N}\,:\;\;v\perp q_{n}\}$ is a compact set for all $n\in\mathbb{N}$, there exists a sequence of vectors in $\{w_{n}\}\subset \mathbb{S}^{N\times N}$  with $w_{n}\perp q_{n}$ such that the maximums are attained,
	\begin{equation}
		\label{e3} \langle Xw_{n},w_{n}\rangle =\sup_{\substack{|v|=1 \\ v \perp q_{n}}} \langle Xv,v\rangle \qquad \;\textrm{for all}\;\;n\in\mathbb{N}.
	\end{equation}
	In addition, there is a subsequence $\{w_{n_{k}}\}\subset\{w_{n}\}$ so that $w_{n_{k}}\to w_{0}$ for some $w_{0}\in \mathbb{S}^{N\times N}$ as $k\to\infty.$
	Observe, since  $w_{n_{k}}\perp q_{n_{k}}$ and $|w_{n_{k}}|=1$, we get 
	\begin{equation}
		|\langle w_{n_{k}},p\rangle|=|\langle w_{n_{k}},p-q_{n_{k}}\rangle|+|\langle w_{n_{k}}, q_{n_{k}}\rangle|\leq |p-q_{n_{k}}|
	\end{equation}
	and this implies that the limit must satisfy
	\begin{equation}
		w_{0}\perp p
	\end{equation}
	due to  the fact that $q_{n_{k}}\to p$ as $k$ goes to infinity. As a consequence, using \eqref{e2} and \eqref{e3},
	\begin{equation*}
		\langle X w_{0},w_{0}\rangle+\frac{\delta}{2}\leq \sup_{\substack{|v|=1 \\ v \perp p}} \langle Xv,v\rangle +\frac{\delta}{2}\leq \langle Xw_{k_{n}},w_{k_{n}}\rangle 
		\qquad \;\textrm{for all}\;\;n_{k}\geq n_{0}
	\end{equation*}
	and this is a contradiction, because $w_{n_{k}}\to w_{0}$ as $k$ goes to infinity.
	
	We now need to prove that $\mathcal{L}$ is upper semicontinuous in the first variable for $p\neq0$, that is
\begin{equation} \label{lisupper}
		\limsup \limits_{q\to p}\mathcal{L}(q,X)\leq \mathcal{L}(p,X).
\end{equation}
	To see this, we proceed by contradiction, arguing in a similar way to the one used before. Suppose that there is a $\delta>0$ such that 
	\begin{equation*}
		\mathcal{L}(p,X)+\delta\leq \limsup \limits_{q\to p}\mathcal{L}(q,X).
	\end{equation*}
	Observe that $\mathcal{L}(q,X)\leq \Lambda(-X)$ for all $q\not =0$ with $\Lambda$ given by \eqref{Lambda}. For that reason, there is a family of vectors $\{q_{n}\}_{n\in\mathbb{N}}$ such that $q_{n}\to p$, 
	\begin{equation}
		\mathcal{L}(q_{n},X)\longrightarrow\limsup_{q\to p}\mathcal{L}(q,X)
	\end{equation}  as $n$ goes to infinity and $n_{0}\in\mathbb{N}$ for which
	\begin{equation*}
		\sup_{\substack{|v|=1 \\ v \perp q_{n}}}\langle Xv,v\rangle+\frac{\delta}{2}\leq \sup_{\substack{|v|=1 \\ v \perp p}}\langle Xv,v\rangle\qquad\;\textrm{for all}\;\;n\geq n_{0}.
	\end{equation*}
	Moreover, there exists a vector $w_{0}$ with $w_{0}\perp p$ and $|w_{0}|=1$ such that 
	\begin{equation*}
		\langle Xw_{0},w_{0}\rangle=\sup_{\substack{|v|=1 \\ v \perp p}}\langle Xv,v\rangle.
	\end{equation*}
	Therefore, 
	\begin{equation}
		\label{e4}	\langle Xw_{n},w_{n}\rangle+\frac{\delta}{2}\leq \sup_{\substack{|v|=1 \\ v \perp q_{n}}}\langle Xv,v\rangle+\frac{\delta}{2}\leq \langle Xw_{0},w_{0}\rangle\;\;\textrm{for all}\;\;n\geq n_{0}
	\end{equation}
	with 
	\begin{equation}
		w_{n}=\frac{w_{0}-\frac{\langle w_{0},q_{n}\rangle}{|q_{n}|^{2}}q_{n}}{|w_{0}-\frac{\langle w_{0},q_{n}\rangle}{|q_{n}|^{2}}q_{n}|},
	\end{equation}
	due to the fact that $w_{n}\perp q_{n}$ and $|w_{n}|=1$. Notice that , since $q_{n}\to p$ as $n\to \infty$ and $w_{0}\perp p$, we get that $w_{n}\to w_{0}$ as $n$ goes to infinity. Therefore, \eqref{e4} gives a contradiction. Thus, from \eqref{lislower} and \eqref{lisupper} we have that
	\begin{equation*}
		\mathcal{L}(p,X)=\lim\limits_{q\to p}\mathcal{L}(q,X)\;\;\textrm{for all}\;\;(p,X)\in\big(\mathbb{R}^{N}\setminus\{0\}\big)\times 
		\mathbb{S}^{N\times N},
	\end{equation*}
	and, as a consequence, we conclude that
	$\mathcal{L}$ is continuous in $\big(\mathbb{R}^{N}\setminus\{0\}\big)\times \mathbb{S}^{N\times N}$.
	
	Next, we prove \textit{ii)}. Since, $p\not =0$,
	\begin{align*}
		\mathcal{L}_{*}(p,X)&=\lim\limits_{\varepsilon\to 0^{+}}\inf\,
		\Big\{\mathcal{L}(q,Y)\,:\;\;0<|p-q|\leq \varepsilon,\;\;0<|X-Y|\leq \varepsilon\Big\},\\
		\mathcal{L}^{*}(p,X)&=\lim\limits_{\varepsilon\to 0^{+}}\sup\,
		\Big\{\mathcal{L}(q,Y)\,:\;\;0<|p-q|\leq \varepsilon,\;\;0<|X-Y|\leq \varepsilon\Big\}
	\end{align*} 
	and 
	\begin{equation}
		\lambda(-Y)\leq \mathcal{L}(q,Y)\leq \Lambda(-Y)\;\;\textrm{for all}\;\;(q,Y)\in \big(\mathbb{R}^{N}\setminus\{0\}\big)\times \mathbb{S}^{N\times N},
	\end{equation}
	we have that, 
	\begin{equation*}
	\begin{array}{l}
	\displaystyle 
		\lim\limits_{\varepsilon\to 0^{+}}\inf\,\Big\{\lambda(-Y)\,:\;\;0<|X-Y|\leq \varepsilon\Big\}\leq \mathcal{L}_{*}(p,X) \\[8pt]
		\qquad \qquad \displaystyle \leq \mathcal{L}^{*}(p,X)\leq 	\lim\limits_{\varepsilon\to 0^{+}}\sup\,\Big\{\Lambda(-Y)\,:\;\;0<|X-Y|\leq \varepsilon\Big\}
		\end{array}
	\end{equation*}
	for all $(p,X)\in\mathbb{R}^{N}\times \mathbb{S}^{N\times N}$. Therefore, we have
	\begin{equation*}
		\lambda(-X)\leq \mathcal{L}_{*}(p,X)\leq \mathcal{L}^{*}(p,X)\leq \Lambda (-X)\qquad\;\textrm{for all}\;\;(p,X)\in\mathbb{R}^{N}\times \mathbb{S}^{N\times N}.
	\end{equation*}
	Moreover, the previous arguments imply that $\mathcal{L}_{*}(0,O)=\mathcal{L}^{*}(0,O)$. Thus, we have also proved \textit{iii)}.
	
	Now, we continue the proof with \textit{iv)}. If $X, Y\in \mathbb{S}^{N\times N}$ are such that $0\leq Y-X$, then $0\leq\langle (Y-X)v,v\rangle$ for all $v\in\mathbb{R}^{N}$. Therefore,
	\begin{equation*}
		\mathcal{L}(p,Y)=\inf_{\substack{|v|=1 \\ v\perp p}}\Big\{- \langle Y v,v\rangle\Big\}\leq \inf_{\substack{|v|=1 \\ v\perp p}}\Big\{- \langle X v,v\rangle\Big\}=\mathcal{L}(p,X)
	\end{equation*}
	for all $p\not =0$. So, $\mathcal{L}$ is degenerate elliptic.

	Finally, we prove \textit{v)}. Consider 
	$p\in\mathbb{R}^N\setminus\{0\}$. For $v\perp p$ it holds that 
	\begin{equation}
		\label{e1} (p \otimes p )v=0,
	\end{equation}
	and hence we have that 
	\begin{equation*}
		\langle Xv,v\rangle=\langle Xv,\big(I-\frac{p \otimes p }{|p|^{2}}\big)v\rangle =\langle \big(I-\frac{p \otimes p }{|p|^{2}}\big) Xv, v\rangle\;\;
	\end{equation*}
	for every $v\perp p$ and every $X\in \mathbb{S}^{N\times N}$, due to the fact that $\big(I-\frac{p \otimes p }{|p|^{2}}\big)$ is a symmetric matrix.  Then, $\mathcal{L}$ can be also to be  expressed as 
	\begin{equation}
		\mathcal{L}(p,X)= \inf_{\substack{|v|=1 \\ v\perp p}}\Big\{- \langle \big(I-\frac{p \otimes p }{|p|^{2}}\big)X v,v\rangle\Big\}\qquad\;(p,X)\in \big(\mathbb{R}^{N}\setminus\{0\}\big)\times \mathbb{S}^{ N \times N}.
	\end{equation}
	Now, since $\alpha p\otimes \alpha p=\alpha^{2}( p\otimes p)$ and \eqref{e1} holds, we have that 
	\begin{equation}
		\langle \big(I-\frac{\alpha p\otimes \alpha p}{|\alpha p|^{2}}\big)(\alpha X+\sigma p\otimes p)v,v\rangle 
		=\alpha \langle \big(I-\frac{p \otimes p }{|p|^{2}}\big)Xv,v\rangle
	\end{equation}
	and, therefore, we conclude that $\mathcal{L}$ is geometric.
\end{proof}

In the following result, we aim to provide an intuition for the scenario when the gradient is zero.  In this case, as we will see in the proof, we cannot  define $\mathcal{L}(0,X)$ properly as the first eigenvalue of $-X$. This is due to the fact that $\mathcal{L}_{*}(0,X)=\lambda(-X)$ but in general $\mathcal{L}^{*}(0,X)\not=\lambda(-X)$.

\begin{corollary}
	Let $X\in \mathbb{S}^{N\times N}\setminus\{O\}$. Then $\mathcal{L}_{*}(0,X)=\lambda(-X)$. However,  
	there exists $X\neq O$ such that, $\mathcal{L}^{*}(0,X)\not=\lambda(-X)$.
\end{corollary}

\begin{proof}
	Let $X\in \mathbb{S}^{N\times N}\setminus\{O\}$.
	By \textit{i)} and \textit{ii)} of the previous Proposition \ref{eprop1}, we have
	\begin{align}
		\lambda(-X)\leq  \mathcal{L}_{*}(0,X)=\lim\limits_{\varepsilon\to 0^{+}}\inf\,&\{ \mathcal{L}(q,Y)\,:\;\;0<|q|\leq \varepsilon,\;\;0<|X-Y|\leq \varepsilon\}\\&= \lim\limits_{\varepsilon\to 0^{+}}\inf\,\{ \mathcal{L}(q,X)\,:\;\;0<|q|\leq \varepsilon\}\leq \lim\limits_{\varepsilon\to 0^{+}}  \mathcal{L}(q_{\varepsilon},X)
	\end{align}
	for every family of vectors $\{q_{\varepsilon}\}_{\varepsilon>0}$ such that  $q_{\varepsilon}\neq 0$ and $|q_{\varepsilon}|\leq \varepsilon$. Since $X\in \mathbb{S}^{N\times N}$, there is a unit vector $v_{1}$ that satisfies $\langle -Xv_{1},v_{1}\rangle =\lambda(-X)$. We take a unit vector $v_{2}$ such that $v_{2}\perp v_{1}$ and define $q_{\varepsilon}=\varepsilon v_{2}$. Therefore, we obtain
	\begin{equation*}
		\lambda(-X)\leq  \mathcal{L}_{*}(0,X)\leq \lim\limits_{\varepsilon\to 0^{+}} \inf_{\substack{|v|=1 \\ v\perp \varepsilon v_{2}}}\big \{- \langle X v,v\rangle\big \}=\min_{\substack{|v|=1 \\ v\perp v_{2}}}\big \{- \langle X v,v\rangle\big \}=\lambda(-X)
	\end{equation*}
	due to the fact that $v_{1}\perp v_{2}$.
	
	Now, we will show that $ \mathcal{L}^{*}(0,X)\not=\lambda(-X)$ in general. Let $X \in \mathbb{S}^{N\times N}$ such that the eigenvalue $\lambda(-X)$ has multiplicity one. By definition of the upper semicontinuous envelope and \textit{i)} of Proposition~\ref{eprop1}, it holds that
	\begin{equation}
		 \mathcal{L}^{*}(0,X)=\lim\limits_{\varepsilon\to 0^{+}}\sup\,
		 \Big\{ \mathcal{L}(q,Y)\,:\;\;0<|q|\leq \varepsilon,\;\;0<|X-Y|\leq \varepsilon\Big\}\geq \lim_{\varepsilon \to 0^+}  \mathcal{L}(p_{\varepsilon},X)
	\end{equation}
	for every family of vectors $\{p_{\varepsilon}\}_{\varepsilon>0}$ such that  $p_{\varepsilon}\not =0$ and $|p_{\varepsilon}|\leq \varepsilon$. Let $v_{1}$ be the unit eigenvector associated to the eigenvalue $\lambda(-X)$ as before. We take $p_{\varepsilon}=\varepsilon v_{1}$. Since the eigenvalue $\lambda(-X)$ has multiplicity one, 
	\begin{equation}
		 \mathcal{L}^{*}(p,X)\geq \lim_{\varepsilon \to 0^+}  \mathcal{L}(\varepsilon v_{1},X)=\min_{\substack{|v|=1 \\ v\perp v_{1}}}\Big \{- \langle X v,v\rangle\Big \}>\lambda(-X).
	\end{equation}
	This ends the proof.
\end{proof}

Now, we can state that the two definitions of being a viscosity 
solution to our evolution problem coincide when the test function
has a nonvanishing gradient, that is, at regular points of $\mathcal{L}$.
The proof follows from the fact that the function $\mathcal{L}$ is continuous in $\big(\mathbb{R}^{N}\setminus\{0\}\big)\times \mathbb{S}^{N\times N}$.

\begin{theorem}\label{g6} Since $ \mathcal{L}$ is continuous in $\big(\mathbb{R}^{N}\setminus\{0\}\big)\times \mathbb{S}^{N\times N}$,
	the definitions \ref{edefvisc1} and  \ref{edefvisc2} are equivalent for regular points of $\mathcal{L}$. That is, if $u$ is viscosity subsolution  (resp. supersolution) to \eqref{PE} and $\phi\in C^{2,1}(\mathbb{R}^{N}\times (0,T))$ touches $u^{*}$ ($u_{*}$) strictly from above (from below) and $\nabla\phi(x,t)\not =0$, then 
	\begin{equation}
		\partial_t \phi(x,t)  + \mathcal{L}(\phi(x,t) )\leq0\;\;(\geq 0)
	\end{equation}
	if and only if 
	\begin{equation*}
		\partial_t\phi(x,t)+ \mathcal{L} (\nabla\phi(x,t),D^{2}\phi(x,t))\leq 0\;\;(\geq 0).
	\end{equation*}
\end{theorem}

Finally, we introduce the notion of viscosity sub and supersolution of the obstacle problem \eqref{ParabolicPb}.

\begin{definition}
	We say that $u:\mathbb{R}^{N}\times[0,\infty)\mapsto \mathbb{R}$ is a viscosity subsolution to \eqref{ParabolicPb} 
	 if its upper semicontinuous envelope, $u^{*}$, satisfies that
	 $u^{*}(x,0) \leq u_0(x)$ and for every $\phi\in C^{2,1}(\mathbb{R}^{N}\times (0,\infty))$ such that $\phi$ touches $u^{*}$ at $(x,t)\in\mathbb{R}^{N}\times (0,\infty) $ strictly from above, that is, $\psi-u^{*}$ has a strict minimum at $(x,t)$, we have that, 
	\begin{equation}
		\max \Big\{ - \partial_t \phi(x,t)  - \mathcal{L}_{*}(\phi(x,t)) , \psi (x) -u (x,t) \Big\} \geq 0.
	\end{equation}
	That means, $u^{*}(x,0) \leq u_0(x)$ and, if $\psi (x) <u(x,t)$, then,
	\begin{equation}
		\partial_t \phi(x,t)  + \mathcal{L}_{*}(\phi(x,t) \leq 0.
	\end{equation}
	Equivalently, in terms of the jets $\mathcal{P}^{2,+}u^{*}(x,t)$, $u$ is a viscosity subsolution to the problem \eqref{ParabolicPb} if 
	$u_{*}(x,0) \leq u_0(x)$ and 
	\begin{equation}
		\max \Big\{ - a - \mathcal{L}_{*}(p,X) , \psi (x) -u (x,t) \Big\} \geq 0\qquad\;\textrm{for all}\;\;(a,p,X)\in \mathcal{P}^{2,+}u^{*}(x,t).
	\end{equation}

	Conversely,  $u$ is called a viscosity supersolution to \eqref{ParabolicPb}
 if its lower semicontinuous envelope, $u_{*}$, satisfies the following:
$u_{*}(x,0) \geq u_0(x)$ and for every $\phi\in C^{2,1}(\mathbb{R}^{N}\times (0,\infty))$ such that $\phi$ touches $u_{*}$ at $(x,t)\in\mathbb{R}^{N}\times (0,\infty) $ strictly from below, that is, $u_{*}-\phi$ has a strict minimum at $(x,t)$, we have that, 
	\begin{equation}
		\max \Big\{ - \partial_t \phi(x,t)  - \mathcal{L}^{*}(\phi(x,t)) , \psi (x) -u (x,t) \Big\} \leq 0.
	\end{equation}
	That is,
	\begin{equation}
		u(x,t) \geq \psi (x,t), \;\;\partial_t \phi(x_0,t_0)  +\mathcal{L}^{*}(\phi(x_0,t_0)) \geq 0\;\;\textrm{and}\;\;u_{*}(x,0) \geq u_0(x).
	\end{equation}
	Equivalently, in terms of the jets $\mathcal{P}^{2,-}u(x,t)$, $u$ is a viscosity supersolution to \eqref{ParabolicPb} if $u_{*}(x,0) \geq u_0(x)$
	and
	\begin{equation}
		\max \Big\{ - a - \mathcal{L}^{*}(p,X) , \psi (x) -u (x,t) \Big\} \leq 0\qquad\;\textrm{for all}\;\;(a,p,X)\in \mathcal{P}^{2,-}u_{*}(x,t).
	\end{equation}

	Finally, $u$ is a viscosity solution to \eqref{ParabolicPb}
	 if it is both a viscosity supersolution and a viscosity subsolution.
\end{definition}

That our problem \eqref{ParabolicPb} has a comparison principle is due to the fact that the problem without the obstacle also has a comparison principle. In \cite{Mercier} the author adapted the proof of Theorem~8.2 in~\cite{CIL} incorporating an obstacle to the problem. For completeness, we introduce and prove the version stated in Theorem \ref{comp.pp} of the comparison principle for our problem.

\begin{proof}[Proof of the Comparison Principle, Theorem \ref{comp.pp}]
	We prove this result arguing by contradiction,
	assuming that we have a viscosity supersolution $\overline{u}$ and a viscosity subsolution $\underline{u}$ such that $\overline{u} \not \geq \underline{u}$. The contradiction that we are pursuing comes from the application of \cite[Theorem 8.3]{CIL} combined with the properties of $\mathcal{L}$.
	
	We split the proof into two steps. 
	In the first step of the proof 
	we include the necessary conditions in order to apply \cite[Theorem 8.2]{CIL}  to the function 
	$$w(x,y,t)=\underline{u}_{\eta}(x,t)- \overline{u}(x,t)-\widetilde{\phi}_{\epsilon,\alpha}(x,y,t),$$ with $(x,y,t)\in \mathbb{R}^{N} \times\mathbb{R}^{N}\times (0,T)$.
	Here $\underline{u}_{\eta}$ is a modification of $\underline{u}$ and $\widetilde{\phi}_{\epsilon,\alpha}$ is an admisible test function with doubled variables, both will be defined later. We will look at $w(x,y,t)$ 
	at a family of points $$\{(\hat{x}_{\varepsilon,\alpha},\hat{y}_{\varepsilon,\alpha},\hat{t}_{\varepsilon,\alpha})\}\subset\mathbb{R}^{N} \times\mathbb{R}^{N}\times (0,T).$$ We will also prove that, at these points, $\underline{u}_\eta$ is larger than the obstacle as well as useful bounds for $\hat{x}_{\varepsilon,\alpha}$ and $\hat{y}_{\varepsilon,\alpha}$. Finally, in the second step, we apply \cite[Theorem 8.2]{CIL} and, using the properties of $\mathcal{L}$ stated in Proposition \ref{eprop1}, we obtain the desired  contradiction.

	\textsc{Step 1.}   
	For each $\eta>0$, we take  $$\underline{u}_{\eta} (x,t):=\max\Big\{\underline{u} (x,t)- \frac{\eta}{(T-t)}, \psi (x)\Big\}.$$ Since $\underline{u}$ is a viscosity subsolution to \eqref{ParabolicPb}, $\underline{u}_{\eta}$ is a viscosity subsolution to
	\begin{equation}
		\max\Big\{- \partial_t w(x,t)  - \mathcal{L}(\nabla w(x,t), D^{2}w(x,t))- \frac{\eta}{T^2} ,\psi -w \Big\} = 0
	\end{equation} 
	in $\mathbb{R}^{N}\times (0,T)$.

	We will prove that $\underline{u}_{\eta}\leq \overline{u}$. Once we get this result, taking the limit $\eta\to 0^{+}$ we obtain the desired conclusion,  $\underline{u}
	\leq \overline{u}$. To show this,
	we argue by contradiction. Suppose that there exists a point $(\bar{x},\bar{t})\in \mathbb{R}^{N}\times (0,T)$ such that $$\underline{u}_{\eta}(\bar{x},\bar{t}) - \overline{u}(\bar{x},\bar{t})\geq 2\delta>0.$$ We consider the function $\phi_{\alpha,\varepsilon}:\,\mathbb{R}^{N}\times \mathbb{R}^{N}\times (0,T)\mapsto \mathbb{R}$, given by
	\begin{equation}
		\phi_{\alpha,\varepsilon}(x,y,t):=\underline{u}_{\eta}(x,t)-\overline{u}(y,t)-\frac{\alpha}{4}|x-y|^{4}-\frac{\varepsilon}{2}\Big(|x|^{2}+|y|^{2}\Big).
	\end{equation}
	Since $\phi_{\varepsilon,\alpha}$ is an upper semicontinuous function, it attains maximums in compact sets. In addition, due to the penalization term $-\frac{\varepsilon}{2}(|x|^{2}+|y|^{2})$ and the definition of $\underline{u}_{\eta}$, we get that, for each $\varepsilon>0$ and $\eta>0$, there exist $R_{\varepsilon}>0$ and $0<T_{\eta}<T$ such that 
	\begin{equation}
		\phi_{\alpha,\varepsilon}(\hat{x}_{\varepsilon,\alpha},\hat{y}_{\varepsilon,\alpha},\hat{t}_{\varepsilon,\alpha})=\sup_{\mathbb{R}^{N}\times\mathbb{R}^{N}\times (0,T)}\{\phi_{\alpha,\varepsilon}(\cdot,\cdot,\cdot)\}
	\end{equation}
	 for some $(\hat{x}_{\varepsilon,\alpha},\hat{y}_{\varepsilon,\alpha},\hat{t}_{\varepsilon,\alpha})\in \bar{B}_{R_{\varepsilon}}(0)\times \bar{B}_{R_{\varepsilon}}(0)\times [0,T_{\eta}]$.
Observe that $\phi_{\varepsilon,\alpha}$ and $(\hat{x}_{\varepsilon,\alpha},\hat{y}_{\varepsilon,\alpha},\hat{t}_{\varepsilon,\alpha})$ also depend on $\eta$ but, to simplify the notation, we are not making this dependence
explicit.

	The distance between $\hat{x}_{\varepsilon,\alpha}$ and $\hat{y}_{\varepsilon,\alpha}$ remains uniformly bounded in $\varepsilon$ for each positive $\eta>0$ and $\alpha>0$:
	if $\varepsilon$ is small enough 
	\begin{equation}
		\label{13}	\phi_{\alpha,\varepsilon}(\hat{x}_{\varepsilon,\alpha},\hat{y}_{\varepsilon,\alpha},\hat{t}_{\varepsilon,\alpha}) \geq \phi_{\varepsilon,\alpha}(\bar{x},\bar{x},\bar{t})\geq\delta 
	\end{equation} 
	and then
	\begin{equation}
		\label{cp1}
		\|\underline{u}\|_{\infty}+\|\overline{u}\|_{\infty}\geq \underline{u}_{\eta}(\hat{x}_{\alpha,\varepsilon},\hat{t}_{\alpha,\varepsilon})- \overline{u}(\hat{y}_{\alpha,\varepsilon},\hat{t}_{\alpha,\varepsilon})\geq \delta +\frac{\alpha}{4}|\hat{x}_{\alpha,\varepsilon}-\hat{y}_{\alpha,\varepsilon}|^{4}+\frac{\varepsilon}{2}(|\hat{x}_{\alpha,\varepsilon}|^{2}+|\hat{y}_{\alpha,\varepsilon}|^{2}), 
	\end{equation}
	which implies that there is a positive constant $C$ depending only on the infinity norm of $\underline{u}$ and $\overline{u}$ such that 
	\begin{equation}
		\label{2}
		|\hat{x}_{\alpha,\varepsilon}-\hat{y}_{\alpha,\varepsilon}|\leq C\alpha^{-1/4}. 
	\end{equation} 
	Also, by \eqref{cp1}, the sequences $\{\varepsilon|\hat{x}_{\alpha,\varepsilon}|^{2}\}$ and  $\{\varepsilon|\hat{y}_{\alpha,\varepsilon}|^{2}\}$ are bounded (with respect to both $\varepsilon$ and $\alpha$). Then, for each $\alpha>0$, we have
	\begin{equation}
		\label{6}
		\varepsilon\hat{x}_{\alpha,\varepsilon},\;\varepsilon\hat{y}_{\alpha,\varepsilon}\to 0\qquad \;\textrm{as}\;\;\varepsilon\to 0^{+}.
	\end{equation}
	Moreover, for every $\alpha>0$ large enough and for every
	subsequence of times $\big(\hat{t}_{\alpha,\tilde{\varepsilon}}\big)_{\tilde{\varepsilon}>0}\subset\big(\hat{t}_{\alpha,\varepsilon}\big)_{\varepsilon>0}$, $\hat{t}_{\alpha,\tilde{\varepsilon}}\not \to 0$ as $\tilde{\varepsilon}\to 0$. This property comes from the uniform continuity of the initial datum $u_{0}$. Let us prove this by contradiction. Suppose that for some $\alpha_{0}>0$, $\hat{t}_{\alpha_{0},\tilde{\varepsilon}} \to 0$ as $\tilde{\varepsilon}\to 0$. Then, by \eqref{13}, 
	\begin{equation}
		0<\delta\leq \underline{u}(\hat{x}_{\alpha_{0},\tilde{\varepsilon}},\hat{t}_{\alpha_{0},\tilde{\varepsilon}})- \overline{u} (\hat{y}_{\alpha_{0},\tilde{\varepsilon}},\hat{t}_{\alpha_{0},\tilde{\varepsilon}}).
	\end{equation}
	Suppose that there exists $C'>0$ such that $|\hat{x}_{\alpha_{0},\tilde{\varepsilon}}|\leq C'$ for every $\tilde{\varepsilon}>0$. Then, by \eqref{2}, there exist $x_{\alpha_{0}},y_{\alpha_{0}}\in\mathbb{R}^{N}$ such that $\hat{x}_{\alpha_{0},\tilde{\varepsilon}}\to x_{\alpha_{0}}$ (up to a subsequence) and $\hat{y}_{\alpha_{0},\tilde{\varepsilon}}\to y_{\alpha_{0}}$ and they satisfy
	\begin{equation}
		\label{14}|\hat{x}_{\alpha_{0}}-\hat{y}_{\alpha_{0}}|\leq C\alpha_{0}^{-1/4}.
	\end{equation}
	Here $C$ is the same constant that appears in \eqref{2}. Hence, since $\underline{u}-\overline{u}$ is upper semicontinuous, we get
	\begin{equation}
		0<\delta\leq\limsup_{\tilde{\varepsilon}\to0^{+}}\big( \underline{u}(\hat{x}_{\alpha_{0},\tilde{\varepsilon}},\hat{t}_{\alpha_{0},\tilde{\varepsilon}})-\overline{u}(\hat{y}_{\alpha_{0},\tilde{\varepsilon}},\hat{t}_{\alpha_{0},\tilde{\varepsilon}})\big)\leq \underline{u}(\hat{x}_{\alpha_{0}},0)- \overline{u}(\hat{y}_{\alpha_{0}},0)\leq u_{0}(\hat{x}_{\alpha_{0}})- u_{0}(\hat{y}_{\alpha_{0}})
	\end{equation}
	since $\underline{u}(\cdot,0)\leq u_{0}(\cdot) \leq \overline{u}(\cdot,0)$.
	Therefore, since $u_{0}$ is uniformly continuous, if $\alpha_{0}$ is large enough, we obtain from \eqref{14} the following contradiction,
	\begin{equation}
		0<\delta\leq u_{0}(\hat{x}_{\alpha_{0}})- u_{0}(\hat{y}_{\alpha_{0}})\leq \frac{\delta}{2}.
	\end{equation}

	On the other hand, if $|\hat{x}_{\alpha_{0},\tilde{\varepsilon}}|\to \infty$ as $\tilde{\varepsilon}\to 0^{+}$, by \eqref{2}, $|\hat{y}_{\alpha_{0},\tilde{\varepsilon}}|\to \infty$. Then, since $$\underline{u}(\cdot,0)\leq u_{0}(\cdot) \leq \overline{u}(\cdot,0)$$ and $$\lim_{|x|\to\infty}u_{0}(x)=\mu\in \mathbb{R},$$ we have that 
	\begin{equation}
		0<\delta\leq\limsup_{\tilde{\varepsilon}\to0^{+}}\big( \underline{u}(\hat{x}_{\alpha_{0},\tilde{\varepsilon}},\hat{t}_{\alpha_{0},\tilde{\varepsilon}})- \overline{u}(\hat{y}_{\alpha_{0},\tilde{\varepsilon}},\hat{t}_{\alpha_{0},\tilde{\varepsilon}})\big)\leq \mu-\mu=0,
	\end{equation}
	which is a contradiction. Thus, if $\alpha>0$ is large enough, we have that $\liminf_{\varepsilon \to 0}\hat{t}_{\alpha,\varepsilon} > 0$. In what follows, let us write $(\hat{x},\hat{y},\hat{t})$ insted of $(\hat{x}_{\varepsilon,\alpha},\hat{y}_{\varepsilon,\alpha},\hat{t}_{\varepsilon,\alpha})$.
	
Our next goal is to prove that $\underline{u}_{\eta}$ is bigger that the obstacle at the point $(\hat{x},\hat{t})$, i.e, we aim to show that 
$$\underline{u}_{\eta}(\hat{x},\hat{t})> \psi (\hat{x}).$$ 
Let us establish this inequality arguing by contradiction. Suppose that $\underline{u}_{\eta}(\hat{x},\hat{t})=\psi (\hat{x})$. Since $\psi\leq \overline{u}$ in $\mathbb{R}^{N}\times (0,T)$, we get, 
	\begin{align}
		0<&\delta\leq \phi_{\varepsilon,\alpha}(\hat{x},\hat{y},\hat{t})\leq \underline{u}_{\eta}(\hat{x},\hat{y})-\overline{u}(\hat{y},\hat{t})=\psi(\hat{x})-\overline{u}(\hat{y},\hat{t})\\&\leq \psi(\hat{x})-\psi(\hat{y})+\psi(\hat{y})-\overline{u}(\hat{y},\hat{t})\leq \psi(\hat{x})-\psi(\hat{y})\leq \omega(|x-y|)\leq \omega(C\alpha^{-1/4}),
	\end{align} 
	where $C$ is the constant in \eqref{2}. Then, if $\alpha$ is large enough, we obtain
	\begin{equation*}
		0<\delta \leq \phi_{\varepsilon,\alpha}(\hat{x},\hat{y},\hat{t})\leq\frac{\delta}{2},
	\end{equation*}  
	and we get a contradiction. Then, we conclude that
	$$\underline{u}_{\eta}(\hat{x},\hat{t})> \psi (\hat{x})$$, which implies that for every test function $\varphi$ such that $\varphi-\underline{u}_{\eta}\geq 0$ and attains its maximum at $(\hat{x},\hat{t})$, the following inequality holds, 
	\begin{equation}
		\varphi_{t}(\hat{x},\hat{t})+\mathcal{L} (\nabla \varphi(\hat{x},\hat{t}), D^{2}\varphi(\hat{x},\hat{t})) + \frac{\eta}{T^2} \leq0
	\end{equation}
	since we have that $\hat{t}\not= 0$.

	Now, let us consider 
	\begin{equation}
		w(x,y,t)=	\underline{u}_{\eta}(x,t)-\overline{u}(y,t)-\widetilde{\phi}_{\epsilon,\alpha}(x,y,t),
		\end{equation} 
		where
	\begin{equation}
		\widetilde{\phi}_{\epsilon,\alpha}(x,y,t)=\phi_{\varepsilon,\alpha}(\hat{x},\hat{y},\hat{t})+\frac{\alpha}{4}|x-y|^{4}+\frac{\varepsilon}{2}\Big(|x|^{2}+|y|^{2}\Big).
	\end{equation}
	Notice that $w(x,y,t)\leq 0$ with equality only if $(x,y,t)=(\hat{x},\hat{y},\hat{t})$.

	\textsc{Step 2.} 
	Now, we are ready to apply \cite[Theorem 8.2 ]{CIL} to $w$, leading to the following result:
	for every $\mu>0$ and for every pair of tuples $(a,X), (b,Y)\in \mathbb{R}^N\times \mathbb{S}^{N\times N}$ such that $$(a,\nabla_{x}	\widetilde{\phi}_{\epsilon,\alpha}(\hat{x},\hat{y},\hat{t}),X)\in \mathcal{P}^{2,+} \underline{u}_{\eta}(\hat{x},\hat{t})\quad \mbox{ and }
	\quad (b,\nabla_{y}	\widetilde{\phi}_{\epsilon,\alpha}(\hat{x},\hat{y},\hat{t}),Y)\in \mathcal{P}^{2,-} \overline{u} (\hat{y},\hat{t})$$ 
	we get that $$a-b=0$$ and 
	\begin{equation}
		\label{3}-\big (\frac{1}{\mu}+||A||\big)I_{2N}\leq \begin{pmatrix}
			X&0\\0&-Y
		\end{pmatrix}\leq A+\mu A^{2},
	\end{equation}
	where 
	\begin{equation}
		A:= D^{2}_{x,y}\widetilde{\phi}_{\epsilon,\alpha}(\hat{x},\hat{y},\hat{t})=\varepsilon I_{2N}+\begin{pmatrix}
			P&-P\\-P&P
		\end{pmatrix}
		\;\;\textrm{and}\;\;P:=\frac{\alpha}{2}(\hat{x}-\hat{y})\otimes(\hat{x}-\hat{y})+\frac{\alpha}{4}|\hat{x}-\hat{y}|^{2}I_{N}.
	\end{equation}
	
	We remark that $X$ and $Y$ depend on $\varepsilon$ and $\alpha$. In particular, the inequalities in \eqref{3} give us 
	\begin{equation}
		\label{4}
		\|X\|+\|Y\|\leq C_{\mu}(\alpha|\hat{x}-\hat{y}|^{2}+\varepsilon)
	\end{equation}
	and 
	\begin{equation}
		\label{4.2}
		X-Y\leq 2\varepsilon(1+\mu\varepsilon)I_{N}
	\end{equation}
	for some constant $C_{\mu}$ that only depends on $\mu$. Furthermore, since  $\underline{u}_{\eta}$ is a viscosity subsolution to $$\partial_t w(x,t)  + \mathcal{L}(\nabla w(x,t), D^{2}w(x,t))+ \frac{\eta}{T^2}=0$$ at $(\hat{x},\hat{t})$ and $\overline{u}$ is a  viscosity supersolution to $$\partial_t w(x,t)  + \mathcal{L}(\nabla w(x,t), D^{2}w(x,t))=0$$ at $(\hat{y},\hat{t})$, we obtain 
	\begin{equation}
		a+ \mathcal{L}_{*}(\nabla_{x}	\widetilde{\phi}_{\epsilon,\alpha}(\hat{x},\hat{y},\hat{t}),X)+\frac{\eta}{T^2} \leq 0,
		\end{equation}
		and
		\begin{equation}
		b+ \mathcal{L}^{*}(-\nabla_{y}	\widetilde{\phi}_{\epsilon,\alpha}(\hat{x},\hat{y},\hat{t}),Y)\geq 0.
	\end{equation}
	Since $a=b$, using the above inequalities we obtain
	\begin{equation}
		0<	\frac{\eta}T\leq  \mathcal{L}^{*}(-\nabla_{y}	\widetilde{\phi}_{\epsilon,\alpha}(\hat{x},\hat{y},\hat{t}),Y)- \mathcal{L}_{*}(\nabla_{x}	\widetilde{\phi}_{\epsilon,\alpha}(\hat{x},\hat{y},\hat{t}),X).
	\end{equation}
	In addition, by \eqref{4.2} and the fact that $\mathcal{L}$ is  degenerate elliptic (point \textit{iv)} in Proposition \ref{eprop1}), we have
	\begin{equation}
		0<	\frac{\eta}T\leq \mathcal{L}^{*}(-\nabla_{y}	\widetilde{\phi}_{\epsilon,\alpha}(\hat{x},\hat{y},\hat{t}),X-2\varepsilon(1+\mu\varepsilon)I_{n})-\mathcal{L}_{*}(\nabla_{x}	\widetilde{\phi}_{\epsilon,\alpha}(\hat{x},\hat{y},\hat{t}),X).
	\end{equation}
	
	On the other hand, from a 
	direct computation, we get
	\begin{equation}
		\nabla_{x}	\widetilde{\phi}_{\epsilon,\alpha}(\hat{x},\hat{y},\hat{t})=\hat{p}+\varepsilon \hat{x}\;\;\textrm{and}\;\;-\nabla_{y}	\widetilde{\phi}_{\epsilon,\alpha}(\hat{x},\hat{y},\hat{t})=\hat{p}-\varepsilon \hat{y}\;\;\textrm{where}\;\;\hat{p}:=\alpha|\hat{x}-\hat{y}|^{2}(\hat{x}-\hat{y}).
	\end{equation}

	Thanks to \eqref{2} and \eqref{6}, $	\nabla_{x}	\widetilde{\phi}_{\epsilon,\alpha}(\hat{x},\hat{y},\hat{t})$ and $-\nabla_{y}	\widetilde{\phi}_{\epsilon,\alpha}(\hat{x},\hat{y},\hat{t})$ converge to the same vector $p_{0}\in\mathbb{R}^{N}$ as $\varepsilon$ goes to zero (up to extracting a subsequence if necessary). Moreover, by \eqref{2} and \eqref{4}, the matrices $X$ are uniformly bounded with respect $\varepsilon$ and therefore  $X$ and $(X-2\varepsilon(1+\mu\varepsilon)I_{N})$ converge  to some $X_{0}\in \mathbb{S}^{N\times N}$ as $\varepsilon$ goes to zero (again, up to a subsequence). Therefore, since $\mathcal{L}^{*}$ and $-\mathcal{L}_{*}$ are upper semicontinuous functions, we have
	\begin{equation}
		0< \frac{\eta}{T^2}\leq \limsup_{\varepsilon\to 0^{+}}\big ( \mathcal{L}^{*}(\hat{p}+\varepsilon\hat{y},X-2\varepsilon(1+\mu\varepsilon)I_{n})-\mathcal{L}_{*}(\hat{p}+\varepsilon\hat{x},X)\big)=\big(\mathcal{L}^{*}-\mathcal{L}_{*}\big)(p_{0},X_{0}),
	\end{equation}
	and the contradiction is straightforward. If $p_{0}\not =0$, then $\mathcal{L}^{*}(p_{0},X_{0})=\mathcal{L}_{*}(p_{0},X_{0})$ since $\mathcal{L}$ is continuous in $\big(\mathbb{R}^{N}\setminus\{0\}\big)\times \mathbb{S}^{N\times N}$ as we proved in point {\it i)} of Proposition \ref{eprop1}. If $p_{0}=0$, we have that $|\hat{x}-\hat{y}|\to 0$ as $\varepsilon\to 0^{+}$ which implies, by \eqref{4}, $X_{0}=0$, and, by point \textrm{iii)} in Proposition \ref{eprop1}, we obtain $\mathcal{L}_{*}(0,O)=0=\mathcal{L}^{*}(0,O)$. In both cases, we conclude that $\big(\mathcal{L}^{*}-\mathcal{L}_{*}\big)(p_{0},X_{0})=0$ and we get the contradiction $0<{\eta}/T^2 \leq 0$.
	
	Therefore, we have proved that $\underline{u}_{\eta}\leq \overline{u}$ in $\mathbb{R}^{N}\times (0,T)$ for every $\eta>0$. Thus, taking the limit $\eta\to 0$, we obtain the desired result, it holds that $$\underline{u}\leq \overline{u}$$ in $\mathbb{R}^N\times (0,T)$.  
	\end{proof}

Once that we know that our evolution problem \eqref{ParabolicPb} has a comparison principle (Theorem \ref{comp.pp}), we can prove that this problem has a unique solution. The proof of the existence is based on Perron's method as described in \cite{Mercier}. This proof also works for more general operators $\mathcal{L}$ that have a comparison principle and satisfy items {\it i)}-{\it v)} in Proposition \ref{eprop1}.

 Now, we will prove that there is a unique solution to \eqref{ParabolicPb}, that is, Theorem \eqref{etheoremExistencePerron}. Let us consider $W$ given by 
\begin{equation}
	\label{e9} W(x,t)=\inf\big\{ w(x,t)\,:\;\;w\in \mathcal{A}\big\}
\end{equation}
with 
  \begin{align}
  	\label{e10}	\mathcal{A}:=\Big\{ & w(x,t)\;:\;\;w\;\;\textrm{is a viscosity supersolution } \\ & w_{t}+\mathcal{L}w \geq 0\;\;\textrm{in}\;\;\mathbb{R}^{N}\times(0,\infty),\ w_{*}(\cdot, t)\geq \psi\;\;\textrm{for all }\;\;t>0\;\;\textrm{and}\;\;\;w_{*}(\cdot,0)\geq u_{0}\Big\}.
  \end{align} 
  Our aim is to show that $W$ is also a supersolution of the Obstacle Problem \eqref{ParabolicPb} and, in addition, it
  is a subsolution to $$u_{t}+\mathcal{L} u=0$$
  inside the set $\{ (x,t) : W (x,t) > \psi (x) \}$. 
 
Before doing that, we need to establish that $W$ is well defined and that it is the smallest supersolution of the Obstacle Problem \eqref{ParabolicPb} using the following lemma.

\begin{lemma}\label{elemma1}
	Let $\mathcal{F}$ be a family of viscosity supersolutions to $u_{t}+\mathcal{L}u=0$ in $\mathbb{R}^{N}\times (0,\infty)$ and assume that
	the family is bounded below. Assume 
	\begin{equation*}
		U(x,t)=\inf\Big\{u(x,t)\,:\;\;u\in\mathcal{F}\Big\}
	\end{equation*}
	is a well defined function with $U(x,t)>-\infty$ for every 
	$(x,t)\in\mathbb{R}^N\times (0,\infty).$ Then, $U$ is also a supersolution to $u_{t}+\mathcal{L}u=0$ in $\mathbb{R}^{N}\times (0,\infty)$.
\end{lemma}

\begin{proof}
Fix $(x_{0},t_{0})\in \mathbb{R}^N \times (0,\infty)$. Suppose that $(a,p,X)\in \mathcal{P}^{2,-}U_{*}(x_{0},t_{0})$. We want to see that 
	\begin{equation}
		a+ F^{*}(p,X)\geq 0.
	\end{equation}
	Thanks to the way in which $U$ is defined, we can take a sequence $\{(x_{n},t_{n},u_{n})\}_{n\in\mathbb{N}}$ such that $u_{n}\in \mathcal{F}$ and $(x_{n},t_{n},u_{n}(x_{n},t_{n}))\to (x_{0},t_{0},U_{*}(x_{0},t_{0}))$. By a parabolic version of \cite[Proposition 4.3]{CIL}, there exists a sequence $\{(\hat{x}_{n},\hat{t}_{n},a_{n},p_{n},X_{n})\}_{n\in\mathbb{N}}\subset\mathbb{R}^{N}\times (0,\infty)\times\mathbb{R}\times \mathbb{R}^{N}\times \mathbb{S}^{N\times N}$ such that $(a_{n},p_{n},X_{n})\in \mathcal{P}^{2,-}u_{n*}(\hat{x}_{n},\hat{t}_{n})$  for all $n\in\mathbb{N}$ and
	\begin{equation*}
		\big(a_{n},p_{n},X_{n}\big )\to (a,p,X)\;\;\textrm{as}\;\;n\to\infty.
	\end{equation*}
	Therefore, we have 
	\begin{equation}
		a_{n}+ \mathcal{L}^{*}(p_{n},X_{n})\geq 0
	\end{equation}
	due to the fact that  $u_{n}$ is a viscosity supersolution in $\mathbb{R}^{N}\times (0,\infty)$ for every $n\in\mathbb{N}$.	Since $\mathcal{L}^{*}$ is upper semicontinuous, we obtain
	\begin{equation}
		a+\mathcal{L}^{*}(p,X)\geq0
	\end{equation}
	as we wanted to show.
\end{proof}

In the following proposition, we show that $W$, defined in \eqref{e9}, is the smallest supersolution of \eqref{PE} that is larger than the obstacle and with initial datum larger than $u_{0}$. As a consequence, $W$ is a supersolution to the Obstacle Problem \eqref{ParabolicPb}. To prove this result, we strongly use Lemma \ref{elemma1} and the fact  that $\mathcal{L}$ is a geometric function (see point \textit{v)} in Proposition \ref{eprop1}).

\begin{proposition}\label{eprop2} The function $W:\mathbb{R}^N\times (0,\infty)\mapsto \mathbb{R}$ given by 
	\begin{equation}
		W(x,t)=\inf\big\{ w(x,t)\,:\;\;w\in \mathcal{A}\big\}
	\end{equation}
	with $\mathcal{A}$ as in \eqref{e10} is  well defined and
	belongs to $\mathcal{A}$. That is, $W$ a supersolution to \eqref{PE} in $\mathbb{R}^{N}\times (0,\infty)$ such that $W(\cdot,0)\geq u_{0}$ and $W\geq \psi$ or, equivalently, $W$ is a supersolution to \eqref{ParabolicPb}.  Furthermore, $W(x,0)=u_{0}(x)$ for all $x\in\mathbb{R}^{N}$ and $W$ is a lower semicontinuous function in $\mathbb{R}^{N}\times(0,\infty)$. 
\end{proposition}

\begin{proof}
We divide the proof in several steps.

	\textsc{First step.} We begin by constructing an upper barrier $w^{+}.$ We look for a function $w^{+}\in\mathcal{A}$ such that $w^{+}(\cdot,0)=u_{0}$. Consider two constants $\alpha$ and $\beta$ such that  $\beta>2\alpha>0$. Fix $z\in\mathbb{R}^N$. We define the function
	\begin{equation}
		h_{z}(x,t):=\alpha|x-z|^{2}+\beta t
	\end{equation}
	for $(x,t)\in\mathbb{R}^{N}\times (0,\infty)$. Observe that, since $D^{2}h_{z}=-2\alpha I_N$, we have that $\mathcal{L}_{*}(\nabla h_{z},D^{2}h_{z})=\mathcal{L}^{*}(\nabla h_{z},D^{2}h_{z})=-2\alpha$ in $\mathbb{R}^{N}\times (0,\infty)$ due to the item \textit{ii)} in Proposition \ref{eprop1}. Then, since 
	\begin{equation}
		\partial_t h_{z}(x,t)+ \mathcal{L}(\nabla h_{z}(x,t),D^{2}h_{z}(x,t))=\beta-2\alpha>0	\end{equation}
for every $(x,t)\in\mathbb{R}^{N}\times (0,\infty)$,
	we have that  $h_{z}$ is a supersolution to \eqref{PE} in  $\mathbb{R}^{N}\times (0,\infty)$ for every $z\in\mathbb{R}^{N}$. Now, our goal is to construct a supersolution to \eqref{PE}  based on $h_z$  that is ordered with the initial datum at $t=0$. Since $u_{0}$ is bounded and continuous, we consider
	\begin{equation*}
		H_{z}(s):=\max \Big\{u_{0}(y)\,:\;\;y\in\mathbb{R}^{N}\;\;\textrm{satisfies that}\;\; \alpha|y-z|^{2}\leq s \Big\},\qquad\;s\in\mathbb{R}.
	\end{equation*}
	 Since  $\mathcal{L}$ is degenerate elliptic and geometric (see Proposition ~\ref{eprop1}) and $H_{z}$ is a nondecreasing continuous function and $h_{z}$ is a supersolution to  \eqref{PE} for all $z\in\mathbb{R}^{n}$, then, by \cite[Theorem 5.2.]{CGG}, the function 
	 	\begin{equation}
	 	H_{z}\circ h_{z}(x,t)=\max\Big\{u_{0}(y): \;\;\alpha(|y-z|^{2}-|x-z|^{2})\leq \beta t \Big\}, \quad (x,t)\in\mathbb{R}^{n}\times (0,\infty)
	 \end{equation} is, endeed, a supersolution to  \eqref{PE} for all $z\in\mathbb{R}^{n}$.  
	 Moreover,  since $H_{z}\circ h_{z}\geq u_{0}$  and $u_{0}\geq \psi$ , we have that $H_{z}\circ h_{z}\geq \psi$ for all $z\in\mathbb{R}^{n}$. Therefore, $H_{z}\circ h_{z}\in \mathcal{A}$ for every $z\in\mathbb{R}^{N}$, so $\mathcal{A}$ is not empty.
	
	Now, to construct a function in $\mathcal{A}$ that takes the initial datum $u_{0}$ at $t=0$, we define
	\begin{equation*}
		w^{+}(x,t):=\inf_{z\in\mathbb{R}^{N}}\Big\{H_{z}\circ h_{z}(x,t)
		\Big\},
	\end{equation*}
	for $(x,t)\in\mathbb{R}^{N}\times[0,\infty)$. Observe that, by construction, $\|H_{z}\circ h_{z}\|_{\infty}\leq \|u_{0}\|_{\infty}$ for every $z\in\mathbb{R}^{N}$. Moreover, since $H_{z}\circ h_{z}$ is a supersolution to \eqref{PE} in $\mathbb{R}^{N}\times (0,\infty)$ for every $z\in \mathbb{R}^{N}$, by Lemma \ref{elemma1}, we get that $w^{+}$ is also a supersolution of the parabolic equation \eqref{PE} in $\mathbb{R}^{N}\times (0,\infty)$.
	
	Finally, since that $H_{z}\circ h_{z}\geq \psi$ and $H_{z}\circ h_{z}(\cdot,0)\geq u_{0}(\cdot)$ for every $z\in\mathbb{R}^{N}$, we obtain that $w^{+}$ satisfies
		\begin{equation}
		w^{+}(x,t)\geq \psi (x)\;\;\textrm{and}\;\;w^{+}(\cdot,0)\geq u_{0}(\cdot).
	\end{equation}
	In addition, $w^{+}(\cdot,0)=u_{0}(\cdot)$ due to the fact that
	\begin{equation*}
		u_{0}(x)\leq w^{+}(x,0)\leq H_{x}\circ h_{x}(x,0)=\max\{u_{0}(y)\,:\;\;\alpha|y-x|^{2}\leq 0\}=u_{0}(x)
	\end{equation*}
	for every $x\in\mathbb{R}^{N}$. Hence, we conclude that $w^{+}\in\mathcal{A}$ and $w^{+}(\cdot, 0)=u_{0}$.
	
	\textsc{Second step.}  We will prove that $W$ is well defined, $W\in\mathcal{A}$ and $W(\cdot,0)=u_{0}(\cdot)$.
	
	First, observe that since $w^{+}\in \mathcal{A}$, $w^{+}\leq \|u_{0}\|_{\infty}$ and $w\geq\psi\geq -\|\psi\|_{\infty}$ for every $w\in\mathcal{A}$, the infimum in $\mathcal{A}$, $W$, is well defined and satisfies $-\|\psi\|_{\infty}\leq W\leq \|u_{0}\|_{\infty}$. Furthermore, by Lemma~\ref{elemma1}, $W$ is also a supersolution to \eqref{PE} in $\mathbb{R}^{N}\times (0,\infty)$.
	
	On the other hand, since $w\geq\psi$  and $w(\cdot,0)\geq u_{0}
	(\cdot)$ for every $w\in\mathcal{A}$, we have that $W\geq\psi$  and $W(\cdot,0)\geq u_{0}(\cdot)$. In addition, it holds that
	\begin{equation}
		u_{0}(x)\leq W(x,0)\leq w^{+}(x,0)=u_{0}(x)
	\end{equation}
	for every $x\in\mathbb{R}^{N}$.
\end{proof}

Now, we show that $W$ is also a subsolution to our obstacle problem.

\begin{proposition}\label{eprop3'}
	The function defined by \eqref{e9}, $W$, is a  subsolution to the obstacle problem \eqref{ParabolicPb} in $\mathbb{R}^{N}\times(0,\infty)$ with $W(\cdot,0)=u_{0}(\cdot)$.
\end{proposition}

\begin{proof}
	Let $(x_{0},t_{0})\in \mathbb{R}^{N}\times (0,\infty)$ such that $W(x_{0},t_{0})-\psi(x_{0})\geq \gamma >0.$ We want to prove that, for every $(a,p,X)\in \mathcal{P}^{2,+}W^{*}(x_{0},t_{0})$, it holds that 
	\begin{equation*}
		a+\mathcal{L}_{*}(p,X)\leq 0.
	\end{equation*}
	We will argue by contradiction. Suppose that there exists  $(a_{0},p_{0},X_{0})\in \mathcal{P}^{2,+}W^{*}(x_{0},t_{0})$ such that 
	\begin{equation}
		\label{e11'} a_{0}+\mathcal{L}_{*}(p_{0},X_{0})\geq \gamma >0.
	\end{equation}
Our goal is to construct a supersolution $U$ to \eqref{PE} that is strictly smaller than $W$ at $(x_{0},t_{0})$. Notice that, by the definition of $W$
as the infimum of supersolutions, this leads to the desired contradiction. 
	
	First, we define
	\begin{align}
		u_{\mu,\sigma}(x,t)&:=W(x_{0},t_{0})-\frac{\gamma}{3}+ a_{0}(t-t_{0})+\langle p_{0},x-x_{0}\rangle\\& \qquad +\frac{1}{2}\langle X_{0}(x-x_{0}),x-x_{0}\rangle+\sigma(|t-t_{0}|+|x-x_{0}|^{2}).
	\end{align}	
	where $\mu, \sigma>0$ and they will be chosen later.
	Observe that, since $W(x_{0},t_{0})>\psi(x_{0})+\gamma$, we have that
	there is a $r_{0}$ depending on $\gamma$, $|a_{0}|$, $|p_{0}|$, $|X_{0}|$ and $\psi$ such that
	\begin{equation}
		W(x_{0},t_{0})-\frac{\mu}{3}+ a_{0}(t-t_{0})+\langle p_{0},x-x_{0}\rangle+\frac{1}{2}\langle X_{0}(x-x_{0}),x-x_{0}\rangle\geq\psi(x)+\frac{\gamma}{3}
	\end{equation}
	for every $(x,t)\in B_{r_{0}}(x_{0},t_{0})$ (here we use the assumption that $\psi$ is continuous). Moreover, thanks to the fact that $\sigma>0$, the above implies that
	\begin{equation}
		u_{\mu,\sigma}(x,t)\geq \psi(x)+\frac{\gamma}{3}\;\;\textrm{for all}\;\;(x,t)\in B_{r_{0}}(x_{0},t_{0}).
	\end{equation}
	In addition, by \eqref{e11'}, we have
	$$a_{0}+\mathcal{L}^{*}(p_{0},X_{0})\geq a_{0}+\mathcal{L}_{*}(p_{0},X_{0})\geq\gamma>0.$$ Then,
	there exists $0<r_{1}\leq r_{0}$ where $r_{1}$  does not depend on $\sigma$ such that  $u_{\mu,\sigma}$ is supersolution to $u_{t}+\mathcal{L}u=0$ in $B_{r_{1}}(x_{0},t_{0})$.
	
	On the other hand, we claim that we can choose $\sigma>0$ 
	large enough such that  
	\begin{equation}
		W(x,t)\leq u_{\mu,\sigma}(x,t)\qquad \;\textrm{for all}\;\;\frac{r_{0}}{2}\leq |(x,t)-(x_{0},t_{0})|\leq r_{0}.
	\end{equation}
	Considering 
	\begin{equation*}
		\tau:=\max\limits_{B_{r_{0}}(x_0,t_0)}W^{*}-\min\limits_{B_{r_{0}}(x_0,t_0)}W\geq 0,
	\end{equation*}
	we have that  $W(x,t)\leq W(x_{0},t_{0})+\tau$ and, as a consequence, 
	\begin{equation*}
	\begin{array}{l}
	\displaystyle
		W(x,t)\leq u_{\mu,\sigma}(x,t)+\frac{\gamma}{3}+\tau-a_{0}(t-t_{0})-\langle p_{0},x-x_{0}\rangle \\[8pt]
		\qquad \qquad \displaystyle -\frac{1}{2}\langle X_{0}(x-x_{0}),x-x_{0}\rangle-\sigma(|t-t_{0}|+|x-x_{0}|^{2})
		\end{array}
	\end{equation*}
	for every $(x,t)\in B_{r_{0}}(x_{0},t_0)$. Now, suppose that $r_{0}/2\leq |(x,t)-(x_0,t_0)|\leq r_0$.  Then, taking 
	\begin{equation*}
		\sigma= \frac{2(\gamma /3+\tau+ |a_{0}|r_0+|p_{0}|r_{0}+|X_{0}|r_{0}^{2}/2)}{r_{0}+r_{0}^{2}}
	\end{equation*}
	 it holds that 
	\begin{equation}
		W(x,t)\leq u_{\mu,\sigma}(x,t)+\frac{\gamma}{3}+\tau+ |a_{0}|r_0+|p_{0}|r_{0}+\frac{|X_{0}|}{2}r_{0}^{2}-\frac{\sigma}{2}(r_{0}+r_{0}^{2})=u_{\gamma,\sigma}(x,t).
	\end{equation}
	Then the claim is proven.

	Now, we define 
	\begin{equation}
		U(x,t):=\begin{cases}
			\min\{W(x,t),u_{\delta,\sigma}(x,t)\}\;\;&\textrm{if}\;\;|(x,t)-(x_{0},t_{0})|\leq r_{1},\\[8pt]
			W(x,t) \;\;&\textrm{otherwise.}
		\end{cases}
	\end{equation}
	Observe that, by Lemma \ref{elemma1} and the previous steps,  $U$ is a supersolution to $u_{t}+\mathcal{L}u=0$ in $\mathbb{R}^{N}\times (0,\infty)$ such that $U\geq \psi$ and $U(\cdot,0)=u_{0}(\cdot)$ (taking $r_{1}$ small enough such that $U(\cdot,0)=W(\cdot,0)$). That is, $U\in \mathcal{A}$ with $\mathcal{A}$ defined in \eqref{e10}. As a consequence, by the definition of $W$ in \eqref{e9} and using that $W$ is lower semicontinuous (see Proposition \ref{eprop2}), we get
	\begin{equation}
		W(x_{0},t_{0})\leq U_{*}(x_{0},t_{0}).
	\end{equation}
	 However,  
	\begin{align}
		U_{*}(x_{0},t_{0})&=\min \Big\{W(x_{0},t_{0}),u_{\gamma,\sigma}(x_{0},t_{0})\Big\}_{*}\\&=\min \Big\{W(x_{0},t_{0}),W(x_{0},t_{0})-\frac{\gamma}{3}\Big\}=W(x_{0},t_{0})-\frac{\gamma}{3}<W(x_{0},t_{0})
	\end{align}
	which is a contradiction with the previous inequality. Then,
	we conclude that $W$ is a subsolution to the obstacle problem \eqref{ParabolicPb}.
\end{proof}

Now we are ready to conclude the proof of Theorem \ref{etheoremExistencePerron}.

\begin{proof}[Proof of Theorem \ref{etheoremExistencePerron}.]
	By Proposition \ref{eprop2}, the function  $W$ is a viscosity supersolution of \eqref{ParabolicPb} such that $W(x,0)=u_{0}$ for all $x\in\mathbb{R}^{N}.$ On the other hand, by Proposition~\ref{eprop3'}, $W$ is a subsolution to~\eqref{ParabolicPb}. Therefore,  thanks to the Comparison Principle~\ref{comp.pp},  $W^{*}\leq W_{*}$ in $ \mathbb{R}^{N}\times [0,\infty)$. Hence, $W$ is a continuous function in $ \mathbb{R}^{N}\times [0,\infty)$ and $W$ is the unique solution to the obstacle problem \eqref{ParabolicPb}.
\end{proof}

\section{The limit of the value functions for the game is the solution
to the parabolic equation}\label{sect.GameAndEq}

In this section, we will obtain that the unique solution $u$ to the obstacle problem \eqref{ParabolicPb} is the locally uniform limit of the value functions $u^{\varepsilon}$ of the $\varepsilon$-game as $\varepsilon$ goes to zero, see Theorem  \ref{maintheorem.juego.conv}.
Our strategy is to prove that the functions 
\begin{equation} \label{subsuperviscositysols}
	\overline{u}(x,t) = \limsup_{\substack{\varepsilon \to 0^+ \\ (y,s) \to (x,t)}} u^\varepsilon(y,s), \quad \underline{u}(x,t) = \liminf_{\substack{\varepsilon \to 0^+ \\ (y,s) \to (x,t)}} u^\varepsilon(y,s),
\end{equation} 
are,  respectively, a viscosity subsolution and a viscosity supersolution to the problem \eqref{PE} with $\overline{u}(x,0) = \underline{u}(x,0) =u_{0}(x)$. Therefore, by the Comparison Principle, Theorem \ref{comp.pp}, we get that $\overline{u}\leq\underline{u}.$ Since by definition $\underline{u}\leq\overline{u}$, we then obtain that $\overline{u}=\underline{u}$. Hence, the limit of the family $u^{\varepsilon}$ exists and it is the unique solution $u$ to \eqref{PE}. Finally, as a consequence of the convergence
$u^{\varepsilon}\to u$, for each $t>0$ we will be able to approximate the set where $u(\cdot,t)$ is positive by the set where $u^{\varepsilon}(\cdot,t)$
is positive for $\varepsilon$ small enough, see Corollary \ref{cor1}. 

Notice that the final payoff is uniformly bounded,
$$
\sup_{x\in \mathbb{R}^N} \max \Big\{ |u_0 (x) |, |\psi_\eps (x) | \Big\} \leq C
$$
for a constant $C$ independent of $\eps$. Therefore, we
conclude that the value function of the game, $u^\eps$ is uniformly bounded. 

\begin{lemma} \label{lema.cota.uniforme}
There exists a constant $C$, independent of $\eps$, such that
\begin{equation} \label{cota.uniforme}
\sup_{(x,t)\in \mathbb{R}^N\times [0,\infty) } |u^\eps (x,t) |\leq C.
\end{equation}
\end{lemma}

 Next, we prove that the half-relaxed limits 
 of the family $u^\eps$, $\overline{u}$ and $\underline{u}$, given by \eqref{subsuperviscositysols},
 attain the initial condition.

\begin{proposition} \label{ppDatoInicial}
	Let $\overline{u}$ and $\underline{u}$ be defined as in \eqref{subsuperviscositysols}. Then  $$\overline{u}(x,0) = \underline{u}(x,0) =u_{0}(x), \qquad \mbox{ for every } x\in \mathbb{R}^N.$$
\end{proposition}

\begin{proof} We give strategies for both Paul and Carol that will yield the two inequalities needed for the proof. 
	
	\textsc{Paul's strategy.}
	In this strategy Paul will never stop playing and collect the amount given by the obstacle. Fix $z \in \RR^N$. Paul chooses vectors $v_k$ perpendicular to $x_{k-1}-z$ for $k = 1, \cdots n$. We let Carol choose the sign freely. This implies that, independently of Carol's choice,
	\begin{equation}
		|x_n - x_0|^2 = |x_0 - z|^2 + n \varepsilon^2 .
	\end{equation}
	Let $t = n\varepsilon^2/2$. Then, we have that 
	\begin{equation}
		|x_n - x_0| = \sqrt{|x_0 - z|^2 + 2t}.
	\end{equation}
	Since Paul has been choosing the vectors, we get a bound from below 
	\begin{equation}
		u_0(x_n) \leq u^\varepsilon (x_0, t), \qquad \mbox{ with }
		|x_n - x_0| = \sqrt{|x_0 - z|^2 + 2t}.
	\end{equation}
	
	\textsc{Carol's strategy.}
	Now we let Paul choose freely the possibility to stop and obtain the payoff given by the obstacle and the direction he wishes to point to, but Carol chooses the sign $b_k$ for $k = 1,\cdots, n$ so that each $x_k$ stays as near as possible to $x_0$. The furthest that each position $\widetilde{x}_k$ can be from $x_0$ is, following the previous computation, $\sqrt{|x_0 - z|^2 + 2t}$. So we get
	\begin{equation}
		|\widetilde{x}_n - x_0| \leq \sqrt{|x_0 - z|^2 + 2t}.
	\end{equation}
	Notice that Paul could choose to stop at the obstacle at any moment. Since we selected a particular strategy for Carol, what we get in this case is that 
	\begin{equation}
		u^\varepsilon(x_0, t) \leq \max\bigg\{\psi(x_0), \max\Big\{\psi(\widetilde{x}_1), \cdots \max\big\{\psi(\widetilde{x}_{n-1}), u_0(\widetilde{x}_n) \big\}\Big\}\bigg\}.
	\end{equation} 
	
	\textsc{Attaining the initial condition.}
	From our previous computations we have that
	\begin{equation}
		u_0(x_n) \leq u^\varepsilon(x_0, t) \leq \max\bigg\{\psi(x_0), \max\Big\{\psi(\widetilde{x}_1), \cdots \max\big\{\psi(\widetilde{x}_{n-1}), u_0(\widetilde{x}_n) \big\}\Big\}\bigg\}
	\end{equation}
	with 
	\begin{equation}
		|x_n-z| = \sqrt{|x_0 - z|^2 + 2t}, \qquad  |\widetilde{x}_k-z| \leq \sqrt{|x_0 - z|^2 + 2t}.
	\end{equation}
	for $k = 1,\cdots n$. Now, taking the $\limsup$ and $\liminf$ when $\varepsilon \to 0$, $x_0\to z$ and $t\to 0$, using the continuity of $u_0$ and $\psi$ the we get 
	\begin{equation}
		u_0(z) \leq \underline{u}(z,0) \leq \overline{u}(z,0) \leq \max\big\{\psi(z), u_0(z) \big\}
	\end{equation}
	To conclude, observe that $u_0 (z)\geq \psi(z)$ for every $z\in \RR^N$. So, we have obtained that $$\underline{u}(z,0) = \overline{u}(z,0) = u_0(z)$$ for an arbitrary $z\in \RR^N$, that is what we wanted to prove.
\end{proof}

\begin{remark} {\rm
When $u_0$ is a Lipschitz function, from the previous estimates
in the proof of Proposition \ref{ppDatoInicial}, we get that
\begin{equation}
		u_0(x_n) \leq u^\varepsilon(x_0, t) \leq \max\bigg\{u_0(x_0), u_0(\widetilde{x}_1), \cdots u_0(\widetilde{x}_{n-1}), u_0(\widetilde{x}_n) \bigg\},
	\end{equation}
	with
	\begin{equation}
		|x_n-z| = \sqrt{|x_0 - z|^2 + 2t}, \qquad  |\widetilde{x}_k-z| \leq \sqrt{|x_0 - z|^2 + 2t}.
	\end{equation}
	for $k = 1,\cdots n$. Hence, using that $u_0$ is Lipschitz we obtain
	an explicit bound of the form
	$$
	|u_0(z) - u^\varepsilon(x_0, t)| \leq 
	C (|x_0 - z| + \sqrt{2t}).	
	$$
	}
\end{remark}

Now, let us show that the half-relaxed limits are sub and supersolutions to our parabolic obstacle problem. Recall that $\psi$ is a continuous and bounded function such that 
\[K= \{x\in \RR^N : \psi(x) > 0\},\]
and that we have characterized the enlarged set $K_\varepsilon = K + {B_{N\varepsilon}} (0)$ in a similar way, that is, there exists a function $\psi_\varepsilon$, such that
we have 
\[K_\varepsilon = \{x\in \RR^N : \psi_\varepsilon(x) > 0\}\]
and $\psi_\varepsilon \rightarrow \psi$ uniformly when $\varepsilon \rightarrow 0$.

\begin{theorem} \label{11}
	The functions $\overline{u}$ and $\underline{u}$  defined in  \eqref{subsuperviscositysols} are respectively viscosity subsolution and supersolutions to the obstacle problem \eqref{ParabolicPb}.
\end{theorem}

\begin{proof} First, remark that
$\overline{u}$ is upper semicontinuous and $\underline{u}$ is lower semicontinuous. 
	We know that $u^\varepsilon(x,t) \geq \psi(x)$ thanks to the DPP. This implies that both $\overline{u}$ and $\underline{u}$ satisfy the same inequality. 
	
	Let us write the Dynamic Programming Principle of our parabolic game as
	\[0 = \max \left\{ \psi_\varepsilon (x) - u^\varepsilon(x,t) , \sup_{|v|=1} \min_{b = \pm 1} \left\{  u^\varepsilon(x + bv\varepsilon, t- \frac12\varepsilon^2 ) - u^\varepsilon(x,t) \right\}  \right\}. \]
	The use of this equation is the key to pass to the limit
	in the viscosity sense.
	
	Let us show that $\overline{u}$ is a viscosity subsolution. 
	To this end, 
	let $\phi \in C^{2,1}(\mathbb{R}^N \times \mathbb{R})$ be a function such as  $\phi-\overline{u}$ achieves a strict minimum at $(x_{0},t_{0})$ in $\overline{B}_{R_0}(x_0,t_0)$. We want to prove that 
	\[\min\Big\{ \partial_t \phi(x_0,t_0)  + \mathcal{L}(\phi(x_0,t_0)) , \overline{u}(x_0,t_0) - \psi(x_0) \Big\} \leq  0.\]
	For every $\varepsilon>0$, we can choose $(x_\varepsilon, t_\varepsilon) \in \overline{B}_{R_0}(x_0,t_0)$ such that 
	\[\phi(x_\varepsilon, t_\varepsilon) - u^\varepsilon(x_\varepsilon, t_\varepsilon) - \varepsilon^3\leq \inf_{\overline{B}_{R_0}(x_0,t_0)}\{\phi(x,t)-u^\varepsilon(x,t)\} \leq \phi(x,t) - u^\varepsilon(x,t), \]
	for every $(x,t) \in \overline{B}_{R_0}(x_0,t_0)$.
	As $\overline{B}_{R_0}(x_0,t_0)$ is a compact set, $(x_\varepsilon, t_\varepsilon)$ converge, up to a subsequence, to some $(\overline{x}, \overline{t})$ as $\varepsilon \to 0^+$. We have that $(\overline{x}, \overline{t}) = (x_0, t_0)$, since $(x_0,t_0)$ is the only maximum of $(\overline{u}-\phi)$ in $B_{R_0}(x_0, t_0)$ by construction. Rearranging the previous expression, we get
	\[u^\varepsilon(y,s)  - u^\varepsilon(x_\varepsilon, t_\varepsilon) \leq \phi(y,s) - \phi(x_\varepsilon, t_\varepsilon)  + \varepsilon^3,  \]
	for every $(y,s) \in \overline{B}_{R_0}(x_0,t_0)$.
	
	Evaluating the DPP expression at $(x,t) = (x_\varepsilon, t_\varepsilon)$ we obtain 
	\[\begin{aligned}
		0 &= \max \left\{ \psi_\varepsilon (x_\varepsilon) - u^\varepsilon(x_\varepsilon,t_\varepsilon) , \sup_{|v|=1} \min_{b = \pm 1} \left\{  u^\varepsilon(x_\varepsilon + bv\varepsilon, t_\varepsilon - 
		\frac12\varepsilon^2 ) - u^\varepsilon(x_\varepsilon,t_\varepsilon) \right\}  \right\} \\
		&\leq \max \left\{ \psi_\varepsilon (x_\varepsilon) - u^\varepsilon(x_\varepsilon,t_\varepsilon) , \sup_{|v|=1} \min_{b = \pm 1} \left\{  \phi(x_\varepsilon + bv\varepsilon, t_\varepsilon - \frac12 \varepsilon^2 ) - \phi(x_\varepsilon,t_\varepsilon) \right\}  + \varepsilon^3 \right\} .
	\end{aligned}\]
	This says that either
	\[\psi_\varepsilon(x_\varepsilon) \geq u^\varepsilon(x_\varepsilon, t_\varepsilon) \quad \text{or} \quad \sup_{|v|=1} \min_{b = \pm 1} \left\{  \phi(x_\varepsilon + bv\varepsilon, t_\varepsilon - \frac12 \varepsilon^2 ) - \phi(x_\varepsilon,t_\varepsilon) \right\}  + \varepsilon^3 \geq 0. \]
	Since $u^\varepsilon(x,t) \geq \psi_\varepsilon(x)$, we can rewrite this as 
	\[\psi_\varepsilon(x_\varepsilon) = u^\varepsilon(x_\varepsilon, t_\varepsilon) \quad \text{or} \quad \sup_{|v|=1} \min_{b = \pm 1} \left\{  \phi(x_\varepsilon + bv\varepsilon, t_\varepsilon - \frac12 \varepsilon^2 ) - \phi(x_\varepsilon,t_\varepsilon) \right\}  + \varepsilon^3 \geq 0, \]
	which, dividing by $\varepsilon^2$ the inequality, is equivalent to 
	\[\psi_\varepsilon(x_\varepsilon) = u^\varepsilon(x_\varepsilon, t_\varepsilon), \quad \text{or} \quad \sup_{|v|=1} \min_{b = \pm 1} \left\{ \frac{ \phi(x_\varepsilon + bv\varepsilon, t_\varepsilon - \frac12\varepsilon^2 ) - \phi(x_\varepsilon,t_\varepsilon) }{\varepsilon^2} \right\}  + \varepsilon \geq 0, \]
	that is, 
	\[0 \leq \max \left\{ \psi_\varepsilon (x_\varepsilon) - u^\varepsilon(x_\varepsilon,t_\varepsilon) , \sup_{|v|=1} \min_{b = \pm 1} \left\{ \frac{ \phi(x_\varepsilon + bv\varepsilon, t_\varepsilon - \frac12 \varepsilon^2 ) - \phi(x_\varepsilon,t_\varepsilon)}{\varepsilon^2} \right\}  + \varepsilon \right\}. \]
	
	Now, we divide the proof into two cases. First, assume that 
	there is a sequence $(x_\varepsilon, t_\varepsilon)$ that satisfies that $\psi_\varepsilon(x_\varepsilon) = u^\varepsilon(x_\varepsilon, t_\varepsilon)$. In this case, we have that $\psi_\varepsilon(x_\varepsilon) = u^\varepsilon(x_\varepsilon, t_\varepsilon)$ for a sequence $\varepsilon \to 0$. Thus, passing to the limit in
	\[u^\varepsilon(y,s)  \leq
	\psi (x_\varepsilon) + \phi(y,s) - \phi(x_\varepsilon, t_\varepsilon)  + \varepsilon^3,  \]
that holds 	for every $(y,s) \in \overline{B}_{R_0}(x_0,t_0)$,
	we obtain $\psi(x_0) = \bar u(x_0, t_0)$. 
	
	For the second case, we can assume that $\psi_\varepsilon(x_\varepsilon, t_\varepsilon)< u^\varepsilon(x_\varepsilon, t_\varepsilon)$ and, therefore, we have
	\begin{equation} \label{subepsilon}
		\sup_{|v|=1} \min_{b = \pm 1} \left\{ \frac{ \phi(x_\varepsilon + bv\varepsilon, t_\varepsilon - \frac12\varepsilon^2 ) - \phi(x_\varepsilon,t_\varepsilon)}{\varepsilon^2} \right\}  + \varepsilon \geq 0.
	\end{equation}
	Let us analyze this expression. We have
	\[\begin{aligned}
		&\phi \Big(x_\varepsilon +bv\varepsilon, t_\varepsilon - \frac12 \varepsilon^2 \Big) - \phi(x_\varepsilon,t_\varepsilon)   \\
		& = \phi \Big(x_\varepsilon, t_\varepsilon - \frac12 \varepsilon^2\Big) - \phi(x_\varepsilon,t_\varepsilon) + \phi\Big(x_\varepsilon + bv\varepsilon, t_\varepsilon - \frac12\varepsilon^2 \Big) - \phi \Big(x_\varepsilon, t_\varepsilon - \frac12 \varepsilon^2 \Big).
	\end{aligned}\]
	The first couple of terms can be estimated with a first order Taylor expansion
	\[\begin{aligned}
		&\phi \Big(x_\varepsilon, t_\varepsilon - \frac12\varepsilon^2 \Big) - \phi(x_\varepsilon, t_\varepsilon ) = - \frac12 \varepsilon^2\partial_t \phi(x_\varepsilon,t_\varepsilon ) + o_\eps (\varepsilon^2),
	\end{aligned}\]
	while the second couple requires a second order Taylor expansion
	\[\begin{aligned}
		&\phi \Big(x_\varepsilon + bv\varepsilon, t_\varepsilon - \frac12 \varepsilon^2 \Big) - \phi(x_\varepsilon,t_\varepsilon - \frac12\varepsilon^2) \\
		&= b \varepsilon \Big\langle \nabla \phi \Big(x_\varepsilon, t_\varepsilon- \frac12\varepsilon^2\Big) , v \Big\rangle + \frac12 \varepsilon^2 \langle D^2\phi(x_\varepsilon, t_\varepsilon- \frac12\varepsilon^2)v,v \rangle + o_\varepsilon(\varepsilon^2).
	\end{aligned}\]
	Choose $v_\varepsilon \in \mathbb{S}^{N-1}$ such that 
	\[\min_{b = \pm 1} \left\{ \frac{ \phi(x_\varepsilon +bv_\varepsilon\varepsilon, t_\varepsilon - \frac12 \varepsilon^2 ) - \phi(x_\varepsilon,t_\varepsilon)}{\varepsilon^2} \right\} = \sup_{|v|=1} \min_{b = \pm 1} \left\{ \frac{ \phi(x_\varepsilon + bv\varepsilon, t_\varepsilon - \frac12\varepsilon^2 ) - \phi(x_\varepsilon,t_\varepsilon)}{\varepsilon^2} \right\}, \]
such a $v_\varepsilon \in \mathbb{S}^{N-1}$ exists due to the continuity of $\phi$. Then,
	\[\begin{aligned}
		&\sup_{|v|=1} \min_{b = \pm 1} \left\{ \frac{ \phi(x_\varepsilon + bv\varepsilon, t_\varepsilon - \frac12 \varepsilon^2 ) - \phi(x_\varepsilon,t_\varepsilon)}{\varepsilon^2} \right\} \\
		&= - \frac12\partial_t \phi(x_\varepsilon,t_\varepsilon) + \frac12\langle D^2\phi(x_\varepsilon, t_\varepsilon- \frac12\varepsilon^2)v_\varepsilon,v_\varepsilon \rangle + \min_{b = \pm 1} \left\{\frac{b}{\varepsilon} \Big\langle \nabla \phi \Big(x_\varepsilon, t_\varepsilon- \frac12 \varepsilon^2\Big) , v_\varepsilon \Big\rangle \right\} + o_\varepsilon(1).
	\end{aligned}\]
	
	Since $|v_\varepsilon|=1$, up to a subsequence, we can suppose that $v_\varepsilon \to v_0$ for some $v_0 \in \mathbb{S}^{N-1}$ when $\varepsilon \to 0^+$. This limit direction $v_0$ is going to satisfy that $v_0 \perp \nabla \phi (x_0, t_0)$. If it did not, for $\varepsilon$  small enough we have that
	\[c\langle \nabla \phi(x_0, t_0) , v_0\rangle \leq \Big\langle \nabla \phi \Big(x_\varepsilon, t_\varepsilon- \frac12 \varepsilon^2\Big) , v_\varepsilon \Big\rangle \leq C\langle \nabla \phi(x_0, t_0) , v_0\rangle \]
	for two positive constants $c,C>0$ independent of $\varepsilon$. But if this was the case and, for instance, $\langle \nabla \phi(x_0, t_0) , v_0\rangle>0$, we would have that 
	\[\lim_{\varepsilon \to 0^+} \min_{b = \pm 1} \left\{\frac{b}{\varepsilon} \Big\langle \nabla \phi \Big(x_\varepsilon, t_\varepsilon- \frac12\varepsilon^2\Big) , v_\varepsilon \Big\rangle \right\} \leq \lim_{\varepsilon \to 0^+} -c\frac{1}{\varepsilon} \langle \nabla \phi(x_0, t_0) , v_0\rangle = -\infty . \]
	Assuming that $\langle \nabla \phi(x_0, t_0) , v_0\rangle < 0$ we arrive to the same conclusion (since we can choose the opposite sign, $b=+1$). This contradicts \eqref{subepsilon}. 
	
	With the perpendicularity condition in mind, $\langle \nabla \phi(x_0, t_0) , v_0\rangle = 0$, and taking into account that 
	\[\min_{b = \pm 1}\left\{\frac{b}{\varepsilon} \Big\langle \nabla \phi \Big(x_\varepsilon, t_\varepsilon- \frac12\varepsilon^2\Big) , v_\varepsilon \Big\rangle \right\} \leq 0,\]
	we can take limits and obtain
	\[0 \leq -\frac12\partial_t \phi(x_0,t_0) + \frac12\langle D^2\phi(x_0, t_0)v_0,v_0 \rangle \leq  - \frac12\partial_t \phi(x_0,t_0)  + \sup_{\substack{|v|=1 \\ v \perp \nabla \phi(x_0,t_0)}} \frac12\langle D^2\phi(x_0, t_0)v,v \rangle.  \]
	Hence, since $\overline{u}(\cdot,0)= u_{0}(\cdot)$ (see Proposition \ref{ppDatoInicial}), we have that $\overline{u}$ is a viscosity subsolution for the parabolic obstacle problem, \eqref{ParabolicPb}. 
	
	To show that $\underline{u}$ is a supersolution, the structure of the proof remains basically the same (with the appropriate changes in the inequalities), but the last step in this case is simpler, since we arrive to
	\[\begin{aligned}
		& 0 \geq \sup_{|v|=1} \min_{b = \pm 1} \left\{ \frac{ \phi(x_\varepsilon + bv\varepsilon, t_\varepsilon - \frac12\varepsilon^2 ) - \phi(x_\varepsilon,t_\varepsilon)}{\varepsilon^2} \right\}  \\
		& \qquad \geq \sup_{\substack{|v|=1 \\ v \perp \nabla \phi(x_\varepsilon,t_\varepsilon - \frac12\varepsilon^2)}} \min_{b = \pm 1} \left\{ \frac{ \phi(x_\varepsilon + bv\varepsilon, t_\varepsilon - \frac12\varepsilon^2 ) - \phi(x_\varepsilon,t_\varepsilon)}{\varepsilon^2} \right\} \\
		&\qquad = -\frac12 \partial_t \phi(x_\varepsilon,t_\varepsilon) +  \sup_{\substack{|v|=1 \\ v \perp \nabla \phi(x_\varepsilon,t_\varepsilon - \frac12\varepsilon^2)}} \Big\{ \frac12 \langle D^2\phi(x_\varepsilon, t_\varepsilon- \frac12 \varepsilon^2)v,v \rangle  + o_\varepsilon(1) \Big\}.
	\end{aligned}\]
	Choose $v_0$ such that 
	\[\langle D^2\phi(x_0, t_0)v_0,v_0 \rangle = \sup_{\substack{|v|=1 \\ v \perp \nabla \phi(x_0,t_0)}}  \langle D^2\phi(x_0, t_0)v,v \rangle. \]  
	Define the vector $w_\varepsilon$ as
	\[w_\varepsilon = \left\langle v_0, \nabla \phi\left(x_\varepsilon, t_\varepsilon - \frac12 \varepsilon^2\right) \right\rangle \frac{\nabla \phi\left(x_\varepsilon, t_\varepsilon - \frac12 \varepsilon^2\right)}{\left|\nabla \phi\left(x_\varepsilon, t_\varepsilon - \frac12 \varepsilon^2\right)\right|^2}  \]
	if $\nabla \phi\left(x_\varepsilon, t_\varepsilon - \frac12 \varepsilon^2\right)  \neq 0$, and $w_\varepsilon = 0$
	otherwise. Observe that $$(v_0 - w_\varepsilon) \perp \nabla \phi\left(x_\varepsilon, t_\varepsilon - \frac12 \varepsilon^2\right).$$ With this in mind,  we get
	\[\begin{aligned}
		 & \sup_{\substack{|v|=1 \\ v \perp \nabla \phi(x_\varepsilon,t_\varepsilon - \frac12 \varepsilon^2)}} \Big\{ \frac12 \Big\langle D^2\phi \Big(x_\varepsilon, t_\varepsilon- \frac12\varepsilon^2\Big)v,v \Big\rangle + o_\varepsilon(1)\Big\} \\
		 & \qquad \geq \left\langle \frac12 D^2\phi\Big(x_\varepsilon, t_\varepsilon- \frac12 \varepsilon^2\Big)\frac{v_0-w_\varepsilon}{\|v_0-w_\varepsilon\|}, \frac{v_0-w_\varepsilon}{\|v_0-w_\varepsilon\|} \right\rangle + o_\varepsilon(1) \\
		&\qquad \qquad \to \frac12 \langle D^2\phi(x_0, t_0)v_0,v_0 \rangle 
		= \sup_{\substack{|v|=1 \\ v \perp \nabla \phi(x_0,t_0)}} \frac12  \langle D^2\phi(x_0, t_0)v,v \rangle,
	\end{aligned}\]
	since $w_\varepsilon \to 0$. Therefore, we found what we wanted, that is,
	\[
		 0 \geq 
		  - \frac12 \partial_t \phi(x_0 ,t_0) +  \sup_{\substack{|v|=1 \\ v \perp \nabla \phi(x_0,t_0 )}} \Big\{ \frac12 \langle D^2\phi(x_0, t_0)v,v \rangle \Big\}.
\] 
	Thus, since  $\underline{u}(\cdot,0)= u_{0}(\cdot)$ (see Proposition~\ref{ppDatoInicial}) and that we have $\underline{u}\geq \psi$, we obtain that $\underline{u}$ is a viscosity supersolution to \eqref{ParabolicPb}.
\end{proof}

Our aim is to establish the convergence of the functions $u^{\varepsilon}$ given by the deterministic game to the solution to the parabolic problem \eqref{ParabolicPb}. Having proved that the the half relaxed limits satisfy that $\overline{u}$ is a viscosity subsolution and $\underline u$ is a viscosity supersolution to \eqref{ParabolicPb}, our next goal is to prove that   $\overline u$ and $\underline u$ coincide.  We will achieve this using the Comparison Principle, Theorem \ref{comp.pp}.

\begin{proof}[Proof of Theorem \ref{maintheorem.juego.conv} ] 
	Consider the upper semicontinuous function $\overline{u}$ and  lower semicontinuous function $\underline{u}$ defined by \eqref{subsuperviscositysols}. 
	
	On the one hand, by definition, we have that $\underline{u}(x,t)\leq \overline{u}(x,t)$  for all $(x,t)\in\mathbb{R}^{N}\times [0,\infty)$.
	
	By Proposition \ref{11}, $\overline{u}$ is subsolution and $\underline{u}$ is a supersolution of problem \eqref{ParabolicPb}.  Then, $\underline{u}\geq \overline{u}$  in $ \mathbb{R}^{N}\times [0,\infty)$ due to the Comparison Principle, Theorem \ref{comp.pp}. Therefore, we 
	conclude that $$\overline{u}\equiv\underline{u}$$ which implies that $u:=\overline{u}=\underline{u}$ is a continuous function in $ \mathbb{R}^{N}\times [0,\infty)$. Moreover, $u$ is the unique solution to the obstacle problem \eqref{ParabolicPb} and $u^{\varepsilon}$ converges to $u$ pointwise in $\mathbb{R}^{N}\times (0,\infty)$. This convergence is locally uniform.
	Let us show this by a contradiction argument. Assume there is a compact set $A\subset \mathbb{R}^{N}\times [0,\infty)$ such that the uniform convergence is not satisfied there. Then, there exist a positive number $\mu>0$, and sequences $\{\varepsilon_{j}\}_{j\in\mathbb{N}}$, $\{(x_j, t_j)\}_{j\in\mathbb{N}} \subset A$ such that $\varepsilon_j \to 0$ as $j\to \infty$ and
	\[|u^{\varepsilon_{j}}(x_{j},t_j)-u(x_j,t_j)|\geq \mu \quad \text{for all }j\in \mathbb{N}.\]
	Then, taking a subsequence if necessary, there exists $(x_{\infty},t_{\infty})\in A$ so that $(x_j,t_j) \to (x_{\infty},t_{\infty})$ as $j \to \infty$. Since $u$ is continuous at $(x_\infty, t_\infty)$, and we know by the previous part of this result that 
	\begin{equation}
		\lim
		_{\substack{\varepsilon \to 0^+ \\ (y,s) \to (x_0,t_0)}} u^\varepsilon(y,s) = u(x_0,t_0),
	\end{equation}
	we obtain that
	\[\mu \leq |u^{\varepsilon_{j}}(x_{j},t_j)-u(x_j,t_j)| \leq |u^{\varepsilon_{j}}(x_{j},t_j) - u(x_\infty, t_\infty)| + |u(x_j, t_j) - u(x_\infty, t_\infty) | \to 0 \]
	when $j\to \infty$, which is a contradiction. 
	\end{proof}

Now, recall that we we want to study the positivity sets 
of $u^{\varepsilon}$ and $u$,
\begin{equation}
	\label{OmegaSets.22}	\Omega^{\varepsilon}_{t}:=\{x\,:\;\;u^{\varepsilon}(x,t)>0\}\;\;\textrm{and}\;\;\Omega_{t}:=\{x\,:\;\;u(x,t)>0\}.
	\end{equation}
 Next,  as a consequence of the convergence of $u^{\varepsilon}$ to $u$, we will prove that for each fixed $t>0$,
 	\begin{equation*}
 	\Omega_t\subset \liminf_{\varepsilon \to 0} \Omega^\varepsilon_t \subset \limsup_{\varepsilon \to 0} \Omega^\varepsilon_t\subset \overline{\Omega}_t.
 \end{equation*}

\begin{proof}[Proof of Corollary \ref{cor1}]
	Fix $t>0$. We start proving that $$\Omega_t\subset \liminf_{\varepsilon \to 0} \Omega^\varepsilon_t.$$ Suppose that $x\in \Omega_t $. Since 
	\begin{equation*}
		\liminf_{\varepsilon \to 0} \Omega^\varepsilon_t=\bigcup_{\varepsilon_{0}>0}\bigcap_{0<\varepsilon\leq \varepsilon_{0}} \Omega^{\varepsilon}_{t},
	\end{equation*}
	we want to prove that there exists $\varepsilon_{0}>0$ such that $x\in \Omega^{\varepsilon}_{t}$ for all $0<\varepsilon\leq\varepsilon_{0}.$ 
	
	Using that $x\in\Omega_{t}$, we get that there is $\mu>0$ such that $u(x,t)\geq \mu. $ By Theorem \ref{maintheorem.juego.conv} ,   $u^{\varepsilon}\to u$. Then, there is $\varepsilon_{0}>0$ such that $u^{\varepsilon}(x,t) \geq \mu/2$ for each $0<\varepsilon\leq \varepsilon_{0}$. Thus, $x\in  \Omega^{\varepsilon}_{t}$ for all $0<\varepsilon\leq\varepsilon_{0}$.

	On the other hand,  suppose that  $x\not\in \overline{\Omega_{t}}$. Our aim is to prove that   $$x\not \in \limsup_{\varepsilon\to 0}\Omega^{\varepsilon}_{t}.$$ Since 
	\begin{equation*}
		\limsup_{\varepsilon \to 0} \Omega^\varepsilon_t=\bigcap_{\varepsilon_{0}>0}\bigcup_{0<\varepsilon\leq \varepsilon_{0}} \Omega^{\varepsilon}_{t},
	\end{equation*} the above is equivalent to prove that there exists $\varepsilon_{0}>0$ for which $x_{0}\not\in \Omega^{\varepsilon}_{t}$ for all $0<\varepsilon\leq \varepsilon_{0}$.
	Due to the fact that $x\not\in\overline{\Omega_{t}}$, we get that there exists $\mu>0$ such that $u(x,t)\leq-\mu$. Using again the convergence of $u^{\varepsilon}$ to $u$, we obtain that there exists $\varepsilon_{0}>0$ such that $u^{\varepsilon}(x,t)\leq -\mu/2$ for every $0<\varepsilon\leq \varepsilon_{0}$, that is, for each $0<\varepsilon\leq \varepsilon_{0}$,  $x\not\in\Omega^{\varepsilon}_{t}$.
\end{proof}

\section{Long time behaviour of the value function of the game and of the solution to the parabolic problem} \label{sect.LongTimeBehaviour}

The main goal of this section is to prove Theorem \ref{cor2}, that is, to prove that in the the asymptotic limit as $t\to +\infty$ we 
recover the convex hull of the obstacle as the limit of the positivity set 
of the solution to the parabolic problem, that is,
$$\textrm{co}(K) = \lim_{t \to \infty} \Omega_t.$$
In other words, we expect that the set where the solution stays positive is close to the convex hull of $K$ for large times. This holds for any 
open set $K$ as long as the condition on $\Omega_0$ and $K$
that we called \eqref{(G)} is satisfied.
However, we can also work with a more general $K$, not necessarily open
(but always assuming \eqref{(G)}), and in this case the best that we can conclude is
\[ \mbox{co}(K) \subset \liminf_{t\to \infty} \Omega_t \subset \limsup_{t\to \infty} \Omega_t \subset \overline{\mbox{co}(K)}.\]

The game is going to be our main tool in this section, so we will also study the value of the game $u^\varepsilon(x,t)$ for large times before taking the limit in $\varepsilon$. To this 
end, we need to introduce a different notion of convexity that takes into account that there is a fixed length for the steps that one can make in the game. With this notion of convexity (that we will call $\eps$-convexity), we define a better suited concept of ``convex hull'' that is adapted to the discrete nature of the game: the $\varepsilon$-convex hull of a set.

\subsection{Defintion of $\varepsilon$-convexity and preliminaries}

Let us first present some previous definitions to understand $\varepsilon$-convexity and see how they reproduce the expected properties of convexity. We will write $\ell_{x,y}$ as the standard segment between $x$ and $y$ (including both $x$ and $y$),
\[\ell_{x,y}:= \Big\{ z=\theta x + (1-\theta)y\;:\;\; 0\leq \theta \leq 1\Big\}.  \]

\begin{definition}
	For $0<\varepsilon<1$, we define the $\varepsilon$-segment between $x$ and $y$ as
	\begin{equation}
		\ell_{x,y}^{\varepsilon}:=
		\begin{cases}
			\displaystyle \Big\{z= \frac{k}{M} x + (1-\frac{k}{M} )y\,:\;\; j = 0, 1,...,M  \Big\}& \displaystyle  \;\;\textrm{if}\;\;\frac{|x-y|}{\varepsilon}=M\in\mathbb{N},\\
			\displaystyle  \Big\{x,y \Big\}& \displaystyle \;\;\textrm{if}\;\;\frac{|x-y|}{\varepsilon}\not\in\mathbb{N}.
		\end{cases}
	\end{equation}
\end{definition}
Notice that when $|x-y|/\varepsilon\not\in\mathbb{N}$ the $\varepsilon$-segment contains just the two points $x$ and $y$. 

With this definition of $\varepsilon$-segment, we can define the concept of $\varepsilon$-convexity.

\begin{definition} {\rm
	$A\subset\mathbb{R}^{N}$ is a \emph{$\varepsilon$-convex} set if for any $x,y \in A$, $\ell^\varepsilon_{x,y} \subset A$.
	}
\end{definition}

\begin{remark}
	The arbitrary intersection of $\varepsilon$-convex sets is $\varepsilon$-convex.
\end{remark}

The reason behind using $\varepsilon$-convexity and not usual convexity is that due to the discrete nature of the game, Carol could force Paul to exit the convex hull of $K$, but not the $\varepsilon$-convex hull. Observe, for example, that if the obstacle was $K = \{a,b\}$ (without enlarging it) and $\ell^\varepsilon_{a,b} \subset \Omega_0$, Paul would win if the game starts at $x \in \ell^\varepsilon_{a,b}$
no matter how many plays the game last, but would loose if $x\in \ell_{a,b} \setminus \ell^\varepsilon_{a,b}$
if they play a large number of times.

Thanks to the fact that 
the intersection of $\varepsilon$-convex sets is $\varepsilon$-convex we can introduce the $\varepsilon$-convex hull of a set.

\begin{definition}
	We define the $\varepsilon$-convex hull of a set $A$, $co_{\varepsilon}(A)$, as the smallest $\varepsilon$-convex set that contains $A$, that is,
	\begin{equation}
		\mbox{co}_{\varepsilon}(A):=\bigcap\limits_{\{C\;\;\varepsilon\text{-convex}\,:\;\;A\subset C\}} C.
	\end{equation}
\end{definition}
Note that, thanks to the previous remark, the set $\mbox{co}_{\varepsilon}(A)$ is $\varepsilon$-convex. 

With these definitions at hand, what we aim to prove is Theorem \ref{Th.Cadenas}, that is, 
\begin{equation} \label{eq:reducedchain}
	\mbox{co}(K) \subset \liminf_{t\to \infty} \Omega^\eps_t \subset \limsup_{t\to \infty} \Omega^\eps_t \subset \mbox{co}(K_\eps) \subset \mbox{co}(K) + B_{N\eps} (0).
\end{equation}
Next, we will deduce Theorem \ref{cor2} from this result. 

Under the hypothesis that $\mbox{co}(K) \subset \Omega_0$ we will be able to prove a stronger statement, in fact, under this condition, we have
\begin{equation} \label{eq:extendedchain}
	\mbox{co}(K) \subset \mbox{co}_\varepsilon (K_\eps) \subset \liminf_{t\to \infty} \Omega^\eps_t \subset \limsup_{t\to \infty} \Omega^\eps_t \subset \mbox{co}(K_\eps) \subset \mbox{co}(K) + B_{N\eps} (0).
\end{equation}

Now, let us state a result concerning classical convex hulls that proves the last inclusion in the above chain.

\begin{lemma}  It holds that
	$$ co(K_\varepsilon) \subset co(K) + B_{N\varepsilon}(0). $$ 
	\end{lemma}

 \begin{proof}
 	Let us call $U:=co(K)+B_{N\eps}(0)$. We will prove that $U$ is convex and $K_\eps\subset U$. Then $\textrm{co}(K_\eps)\subset U$ because $\textrm{co}(K_\eps)$ is the intersection of all convex sets that contains $K_\eps$.
	
 	On the  one hand, $U$ is convex. By definition $U=\{y=x+z : x\in \mbox{co}(K) \ \mbox{and} \ z\in B_{N\eps}\}$. Given $y_1 , y_2\in U$, $y_1=x_1+z_1$ and $y_2=x_2+z_2$ with $x_1,x_2\in \mbox{co}(K)$ and $z_1,z_2\in B_{N\eps}$. Given $0\le \theta\le 1$ let us consider 
 	$$
 	\theta y_1+(1-\theta)y_2=\big(\theta x_1+(1-\theta)x_2\big)+\big(\theta z_1+(1-\theta)z_2\big).
 	$$
 	Using that $\mbox{co}(K)$ and $B_{N\eps}(0)$ are convex we get that $(\theta y_1+(1-\theta)y_2)\in co(K)+B_{N\eps}(0)$.
	
 	On the other hand, it holds that $K_{\varepsilon}\subset U$ because $K_\eps=\{y=x+z : x\in K \ \mbox{and} \ z\in B_{N\eps} (0) \}$ and $K\subset \mbox{co}(K)$.
 \end{proof}

Next, we want to characterize the $\varepsilon$-convex hull of a set by means of the $\varepsilon$-segments that we introduced. We need the following auxiliary definitions. 

\begin{definition}
	Let $A \subset \RR^N$. Let us consider $\ell^0(A) := A$ and $\ell^{0,\varepsilon}(A) := A$. We define $\ell^j(A)$ and $\ell^{j,\varepsilon}(A)$ for $j \in \mathbb{N}$ as 
	\begin{align}
		\ell^{j}(A) &:= \Big\{x\in\mathbb{R}^{N}:\;\;\exists \,a,b\in \ell^{j-1}(A)\;\;\textrm{such that}\;\;x\in\ell_{a,b} \Big\}, \\
		\ell^{j,\varepsilon}(A) &:=\Big\{x\in\mathbb{R}^{N}:\;\;\exists \,a,b\in \ell^{j-1,\varepsilon}(A)\;\;\textrm{such that}\;\;x\in\ell^{\varepsilon}_{a,b} \Big\}.
	\end{align}
\end{definition}

\begin{remark} \label{re:increasingsets}
	Observe that the sequences $\left\{ \ell^j(A) \right\}^\infty_{j=1}$ and $\left\{ \ell^{j, \varepsilon}(A) \right\}^\infty_{j=1}$ are increasing sequences. We also have that $\ell^j(A) = \ell^1\left( \ell^{j-1}(A)\right)$ and $\ell^{j,\varepsilon}(A) = \ell^{1,\varepsilon}\left( \ell^{j-1,\varepsilon}(A)\right)$.
\end{remark}

Notice that what the previous definitions do is add to a set the segments whose endpoints lie in the set. A naive first impression may suggest that $\mbox{co}(A) = \ell^1 (A)$, but if the dimension of the ambient space is bigger that one, this is not true in general. Next, we are going to characterize both $\mbox{co}(A)$ and $\mbox{co}_\varepsilon(A)$ of a set $A\subset \RR^N$ with the families $\{\ell^j(A)\}$ and $\{\ell^{j,\varepsilon}(A)\}$, $j\in \mathbb{N}$.

\begin{proposition} \label{re:ellcharac}
	Let $A \subset \RR^N$. Then
	\begin{equation} \label{eq:ccharacterizacion}
	\mbox{co}_{\varepsilon}(A)=\bigcup_{j=0}^{\infty}\ell^{j,\varepsilon}(A), \qquad 	\mbox{co}(A)=\ell^{N}(A).
	\end{equation}
\end{proposition}

\begin{proof} 
	We begin proving the result for the $\varepsilon$-convex hull. First, note that the family $\{\ell^{j,\varepsilon}(A)\}_{j=1}^{\infty}$ is an increasing sequence of sets as we mentioned in Remark \ref{re:increasingsets}. Then, if $x,y\in 	\bigcup_{j=0}^{\infty}\ell^{j,\varepsilon}(A)$ are such that $|x-y|/\varepsilon\in \mathbb{N}$, there exists $k\in \mathbb{N}$ such that $x,y\in \ell^{k,\varepsilon}(A)$ which implies $\ell^{\varepsilon}_{x,y}\in\ell^{k+1,\varepsilon}(A)$ and therefore, $\bigcup_{j=0}^{\infty}\ell^{j,\varepsilon}(A)$ is $\varepsilon$-convex. Since $A=\ell^{0,\varepsilon}(A)$, we obtain
	\begin{equation}
		\textrm{co}_{\varepsilon}(A)\subset\bigcup_{j=0}^{\infty}\ell^{j,\varepsilon}(A).
	\end{equation}
	
	Next, we want prove that $\ell^{n,\varepsilon} (A)\subset \textrm{co}_{\varepsilon}(A)$ for every $n\in\mathbb{N}$. For $n=1$, if $x\in \ell^{1,\varepsilon}(A)$ there exist $a,b\in A$ such that $x\in\ell^{\varepsilon}_{a,b}$. This implies that $x\in\textrm{co}_{\varepsilon}(A)$ because $\textrm{co}_{\varepsilon}(A)$ is $\varepsilon$-convex. Now, arguing inductively, suppose that $\ell^{n-1,\varepsilon}(A)\subset\textrm{co}_{\varepsilon}(A)$. Let $x\in \ell^{n,\varepsilon}(A)\setminus \ell^{n-1,\varepsilon}(A).$ Therefore, there exists $a,b\in \ell^{n-1,\varepsilon}(A)$ such that $x\in \ell^{\varepsilon}_{a,b}$ and, applying the inductive hypothesis, we obtain that $x\in\textrm{co}_{\varepsilon}(A)$ because $a,b\in \textrm{co}_{\varepsilon}(A)$ and $\textrm{co}_{\varepsilon}(A)$ is $\varepsilon$-convex. Thus,
	\begin{equation}
		\textrm{co}_{\varepsilon}(A)\supset\bigcup_{j=0}^{\infty}\ell^{j,\varepsilon}(A).
	\end{equation}

	The last part of the proposition, namely, the fact that $\textrm{co}(A) = \ell^N(A)$ (notice that the dimension of the ambient space plays a role here), is just a simple consequence of Caratheodory's Theorem. This is a standard result for the convex hull of a set, which states that for any element $x \in \textrm{co}(A)$, $x$ is the result of a convex combination of at most $N+1$ points in $A$. 
\end{proof}

Now, under our hypothesis on the relationship between $\textrm{co}(K)$ and $\Omega_0$ we will be able to prove the chain of inclusions \eqref{eq:reducedchain}, and, with further hypotheses, \eqref{eq:extendedchain}. 

First, we introduce two important strategies (a version for each player) that  will help us to understand the asymptotic behaviour of the value of the game.
 
\begin{definition} \label{ConcentricStrategie}Let 
	\begin{equation}
	\{(v_{n},b_{n})\}_{n\in\mathbb{N}}\subset\mathbb{S}^{N-1}\times\{1,-1\}
	\end{equation} 
be a family of choices for Paul and Carol. 

We say that the choice $\{(v_{n},b_{n})\}_{n\in\mathbb{N}}$ follows a concentric strategy at $x$ for the $\varepsilon$-game if for some  $z\in\mathbb{R}^{N}$ the sequence 
\begin{equation*}
	x_{1}:= x+b_{1}v_{1},\quad x_{n}=x_{n-1}+b_{n}v_{n},\quad n\in\mathbb{N}
\end{equation*}
satisfies that 
\begin{equation}
	\label{g7}|x_{n}-z|^{2}\geq |x-z|^{2}+n\varepsilon^{2}, \quad\textrm{for every }n\in\mathbb{N}.
\end{equation}

For each starting point $x$ and centre $z$, each player can implement a concentric strategy regardless the choices of the other player
(if Paul plays without stopping at the obstacle).

\begin{itemize}
	 \item \emph{Concentric strategy for Paul:} Here, at each round
	 of the game, Paul 
	 chooses to continue playing and a vector $v_{n}$ such that $v_{n}\perp x_{n-1}-z$ for $n\geq 2$ and  $v_{1}\perp x-z$.  With these choices, no matter which election of sign Carol makes, we have that 
	 \begin{equation*}
	 	|x_{1}-z|^{2}=|x-z|^{2}+\varepsilon^{2},\quad 	|x_{n}-z|^{2}=|x_{n-1}-z|^{2}+\varepsilon^{2}\quad n\geq 2
	 \end{equation*}
	 Then,  \eqref{g7} hols for the sequence of positions of the game $\{x_{n}\}_{n\in\mathbb{N}}$.
	 
	 \medskip
	 
	 	\item \emph{Concentric strategy for Carol:} 
		Assume that Paul chooses to play at every round. Let $v_{1}$ be the choice of Paul at the first round. Carol chooses $b_{1}=\pm$ such that $x_{1}\not \in B_{R_{1}}(z)$
	 with $R_{1}=|x-z|$. Then,  $|x_{1}-z|^{2}\geq |x-z|^{2}+\varepsilon^{2}$ with equality if and only if $v_{1}\perp x-z$. For the 
	 $n$-th round, given Paul's selection of a vector, $v_{n}$, Carol chooses again  the sign $b_{n}=\pm$ that satisfies that $x_{n}\not \in B_{R_{n}}(z)$ with $R_{n}=|x_{n-1}-z|.$ By an iterative argument, one can see that the sequence $\{x_{n}\}_{n\in \mathbb{N}}$ satisfies \eqref{g7}.
	\end{itemize}
\end{definition}

Although Paul and Carol can use a concentric strategy for each $x$, this approach may not always optimal. However, in some cases, merely having the option to execute this strategy determines the winner of the game. The final outcome heavily depends on the position of the starting point $x$ relative to $\textrm{co}(K)$ and the time $t$ at which the game
starts (this determines the maximum number of plays). For an example, see for instance the proof of the following result.

\begin{lemma} \label{glemma2} 
	For every $\varepsilon>0$, it holds that
	\begin{equation}
	 \limsup_{t\to \infty} 	\Omega^{\tilde{\eps}}_{t}\subset
		\overline{\text{co}(K_{\varepsilon})}\qquad \;\textrm{for all}\;\;0<\tilde{\varepsilon}\leq \varepsilon.
	\end{equation}
\end{lemma}

	\begin{proof}
Fix $\varepsilon>0$ and choose $x \notin \overline{\mbox{co}(K_{\varepsilon})}$. We will prove that 
		\begin{equation}
		\label{g2} x\notin \limsup_{t\to \infty} 	\Omega^{\tilde{\eps}}_{t}=\bigcap_{t\geq 0} \bigcup_{s\geq t}\Omega^{\tilde{\eps}}_{s}\;\;\textrm{for every}\;\;0<\tilde{\eps}\leq \varepsilon.
		\end{equation}
		That is, we want to prove that there exists $t\geq0$ such that  $x\notin\Omega^{\tilde{\eps}}_{s}$ for all $s\ge t$ and all $0<\tilde{\eps}\leq \varepsilon$.
		
		We claim that there exists a time $\tau_{x}\geq 0$, 
		that may depend on $x$,  such that 
		\begin{equation}
		\label{g3}	u^{\tilde{\eps}}(x,t)\leq -\mu/2\;\;\textrm{for all}\;\;t\geq \tau_{x}\;\;\textrm{and}\;\;0<\tilde{\eps}\leq\varepsilon,
		\end{equation}
		where $\mu>0$ is such that $u_{0}(x)\to-\mu$ as $|x|\to\infty.$

		Using the claim, since $\Omega_{s}^{\tilde{\eps}}=\{x\in\mathbb{R}^{N}:\;u^{\tilde{\eps}}(x,s)>0\}$, we have that $x\not\in\Omega^{\tilde{\eps}}_{s}$ for all $s\geq \tau_{x}$ and $0<\tilde{\eps}<\varepsilon$. Then, we have  just proved \eqref{g2}.
		
		Now, let us prove the claim.  Since $u_{0}$ is a continuous function such that  $u_{0}(x)\to -\mu$ as $|x|\to\infty$ for some $\mu>0$, there exists  $R_{1}>0$ such that $\{x\in\mathbb{R}^{N}:\;u_{0}(x)\geq -\mu/2\}\subset B_{R_{1}}(0)$. 
	
	On the other hand, since $x \notin \overline{\mbox{co}(K_{\varepsilon})}$, there exist $z \in \RR^N$ and $R_2>0$ such that $\mbox{co}(K_{\varepsilon}) \subset B_{R_2} (z)$ and $x \in \partial B_{R_2}(z)$ due to the fact that  $\overline{\mbox{co}(K_{\varepsilon})}$ is a bounded, closed and convex set. In this situation,  using the concentric strategy, playing the $\tilde{\eps}$-game, (with  $0<\tilde{\eps}<\varepsilon$), Carol wins. No matter which choice of  $v\in \mathbb{S}^{N-1}$ Paul makes,  Carol can choose the next position $x_{1}=x\pm v\tilde{\eps}$ such that $|x_1-z|^{2}\ge R_{2}^{2}+\tilde{\eps}^2$ with equality if and only if $ x-z\perp v $.  For the next position, $x_{2}$,  for any vector that Paul could choose, Carol choose the sign that increases the distance between the position and $z$, that is, she chooses the sign that ensures that  $|x_{2}-z|^{2}\geq |x_{1}-z|^{2}+\tilde{\eps}^{2}$. Then, $|x_{2}-z|^{2}\geq R_{2}^{2}+2\tilde{\eps}^{2}$. Iterating this strategy, after $n$ rounds, we have that $ |x_{n}-z|^{2}\geq R_{2}^{2}+n\tilde{\eps}^{2}$
and then we get
\begin{equation*}
	|x_{n}|^{2}\geq R_{2}^{2}+n\tilde{\eps}^{2}-|z|^{2}.
\end{equation*}
Taking, $\tau_{x}$ such that $\tau_{x}\geq R_{1}^{2}+|z|^{2}$, we have 
\begin{equation}
		|x_{n}|^{2}\geq R_{2}^{2}+n\tilde{\eps}^{2}-|z|^{2}\geq n\tilde{\eps}^{2}-|z|^{2}\geq R_{1}^{2}\;\;\textrm{for all}\;\; n\tilde{\eps}^{2}\geq \tau_{x}.
\end{equation}
Therefore, $x_{n}\not\in B_{R_{1}}(0)$ for all $n\tilde{\eps}^{2}\geq \tau_{x}$ which implies, by the choice of $R_{1}$, that
\begin{equation*}
	u^{\tilde{\eps}}(x,t)\leq \sup\big\{u_{0}(y)\,:\;\;y\not\in B_{1}(0)\big\}\leq -\mu/2\;\;\textrm{for all}\;\;t\geq \tau_{x}\;\;\textrm{and}\;\;0<\tilde{\eps}\leq \varepsilon.
\end{equation*}
This ends the proof.
\end{proof}

Let us improve the previous result and show that if $x \not \in \textrm{co}(K_{\varepsilon})$, then Carol has a strategy that allows her to leave the domain $\Omega_0$
if they play a large number of times (and hence Paul looses).

\begin{lemma}
It holds that	
\begin{equation}
		\limsup_{t \to \infty} \Omega^\eps_t \subset 
		\textrm{co}(K_{\varepsilon}).
	\end{equation}
\end{lemma}

\begin{proof}
	We are going to describe a strategy for Carol so that if $x$ is in $\partial \mbox{co} (K)$, then it forces Paul to eventually move to one position that is in $\RR^N \setminus \overline{\mbox{co}(K)}$. Then, we can apply our previous result to 
	conclude.

Since $x \not \in \mbox{co}(K)$, there exists a hyperplane that separates $x$ from $\mbox{co}(K)$. If Paul chooses a vector $v$ not included in the hyperplane, Carol may choose the direction that will make $x_1$ 
go to a positive distance to $\mbox{co}(K)$. Then, from that point, Carol may apply a concentric strategy just like in the previous result and exit the domain $\Omega_0$. 
	
Now, if we assume that Paul always chooses vectors included in the hyperplane $H$, Carol may perform a concentric strategy inside the hyperplane so that she can also exit the domain $\Omega_0$. In this case, Carol treats $H\cap\Omega_0$ as the new domain and $H \cap \overline{co(K)}$ as the new obstacle. Carol just needs to apply the concentric strategy to move out of $H \cap \Omega_0$. 
\end{proof}

\begin{remark} \label{remarksinengordado}
	In this proof we only used that $K_\varepsilon$ was an open set, but we are not using that it is an enlarged version of $K$. Recall that we are assuming $K$ to be open. Then, if we consider the same game but using $K$ instead of the enlarged version of $K$, this result would hold just the same, that is, we can deal with $x \in \partial \text{co}(K)$
	playing the game without enlarging the obstacle.
	This fact will be used to obtain the inclusion
	\begin{equation}
		\limsup_{t \to \infty} \Omega_t \subset 
		\textrm{co}(K)
	\end{equation}
	that we prove next. 
	
	Notice that this inclusion cannot be deduced taking the limit as $\eps\to 0$ in the previous result, $ \displaystyle\limsup_{t \to \infty} \Omega^\eps_t \subset 
		\textrm{co}(K_{\varepsilon})$. 
\end{remark}

\begin{theorem}\label{SecondPartTheorem1.5}
	We have that
		\begin{equation*}
		\limsup_{t\to\infty}\Omega_{t}\subset \mbox{co}(K).
	\end{equation*}
\end{theorem}

\begin{proof}
	 Let $x\not\in \overline{\mbox{co}(K)}$. Then, there exists $\varepsilon_{0}>0$,  depending on $x$, such that  $x\not\in \overline{\mbox{co}(K_{\varepsilon_{0}})}$.
	By  the claim~\eqref{g3} in the proof of  Lemma \ref{glemma2}, there is a time $\tau_{x}>0$ such that
$u^{\varepsilon}(x,t)<-\mu/2$ for all $t\geq \tau_{x}$ and $0<\varepsilon<\varepsilon_{0}$ where $\mu>0$ is the constant such that $u_{0}(x)\to -\mu$ as $|x|\to\infty$.
	Since $u^{\varepsilon}$ converges locally uniformly to $u$ (see Theorem \ref{maintheorem.juego.conv}), we have that $u(x,t)\leq -\mu/2$ for all $t\geq \tau_{x}$ and every $\eps$ small enough and therefore, $\displaystyle x\not\in \limsup_{t\to\infty}\Omega_{t}$.

	If $x \in \partial \text{co}(K)$ then the result follows from Remark \ref{remarksinengordado} and the previous computations.
\end{proof}

Now we describe some strategies that Paul may use to win. As we already mentioned, it is convenient to divide our arguments into different situations. We present our ideas from the most to the least favorable
situation for Paul.

\subsection{Convex hull of $K$ included in $\Omega_0$}

For this whole subsection we assume that $$\mbox{co}(K) \subset \Omega_0 \qquad 
\mbox{and} \qquad K_\eps \subset \Omega_0 \mbox{ (for $\eps$ small)}.$$ 
Notice that $K_\eps \subset \Omega_0$ for $\eps$ small follows from $\overline{K} \subset \Omega_0$.
Observe that when $\Omega_0$ is bigger Paul has
more opportunities to win. 
We are going to prove \eqref{eq:extendedchain}. 

We need the next result for the first part of the chain of inclusions \eqref{eq:extendedchain}. 
In the next lemma we use our choice of the enlarged obstacle as 
$K_\eps = K + B_{N\eps}(0)$.

\begin{lemma}
	Let $K \subset \RR^N$ be an open set and let
	$K_\eps = K + B_{N\eps} (0)$. Then, it holds that
	\begin{equation}
        \mbox{co}(K) \subset \mbox{co}_\varepsilon(K_\eps).
    \end{equation}
\end{lemma}

\begin{proof}
    First, we show that $\ell^1(K) \subset \ell^{1,\varepsilon}(K_\eps)$. Choose $a,b \in K$.  Let us assume that $|a-b|/\varepsilon \not\in \mathbb{N}$, since the proof in the case that $|a-b|/\varepsilon \in \mathbb{N}$ is simpler and follows exactly the same ideas.
    
    If $|a-b|$ is not a natural multiple of $\varepsilon$, there exists $\hat b$ in the line formed by $a,b$ such that $|a-\hat b|$ is a multiple of $\varepsilon$ and the distance between $b$ and $\hat b$ is at most $\varepsilon/2$ (if it was larger than $\varepsilon/2$ forward, it would be less than $\varepsilon/2$ backwards and vice versa). See Figure \ref{Fig5}. Notice that if $|a-b|$ was a natural multiple of $\varepsilon$, the only necessary change to the proof would be choosing $\hat b = b$. 
    
    \begin{figure}[H]
        \centering
        \includegraphics[width=9cm]{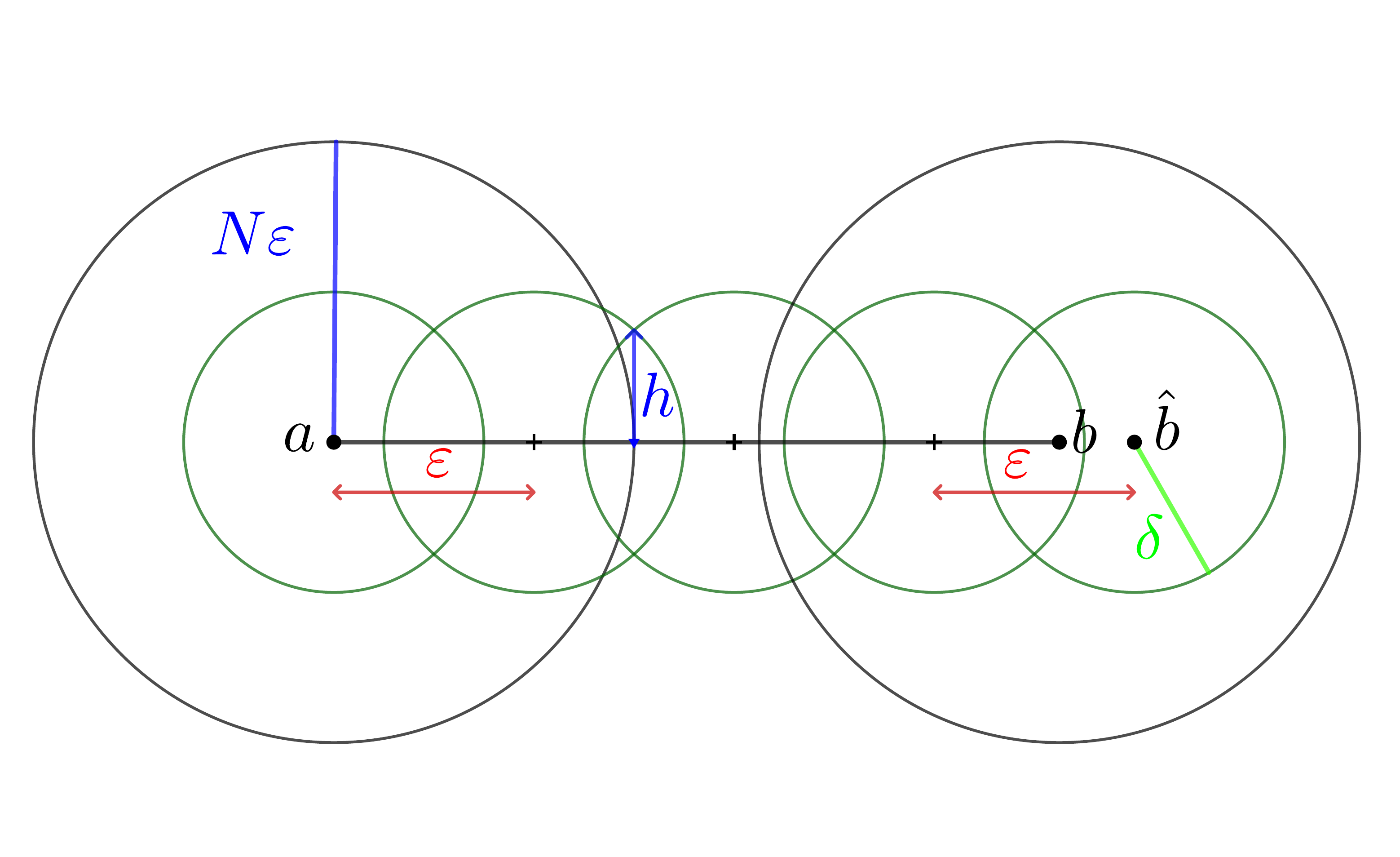}
        \caption{ }
        \label{Fig5}
    \end{figure}

 Since $a,b \in K$, then $B_{N\eps}(a), B_{N\eps}(b) \subset K_\eps$. We know that the maximum distance between $b$ and $\hat b$ is $\varepsilon/2$, so choose $\delta = {N\eps}-\varepsilon/2< {N\eps}$ and we get that $B_\delta(\hat b) \subset B_{N\eps}(b)$. From the construction of $\ell^{1,\varepsilon}(K_{\eps})$, we know that it contains copies of the set $B_\delta$ from $a$ to $\hat b$ centered at each natural multiple of $\varepsilon$ between them. Then, we have that these $B_\delta$ balls overlaps. Let us compute this overlap. 

    \begin{figure}[H]
        \centering
        \includegraphics[width=9cm]{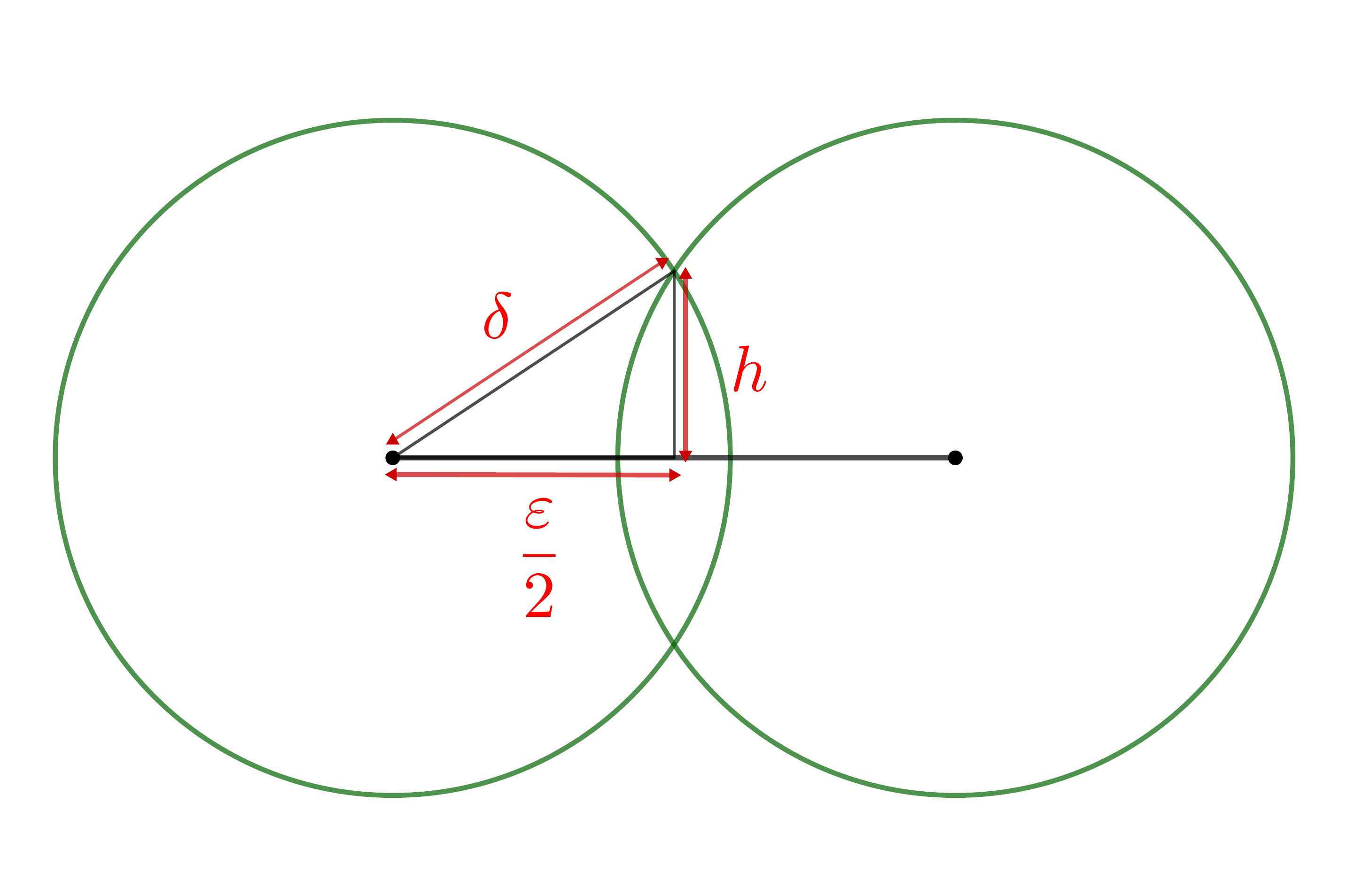}
        \caption{ }
        \label{Fig6}
    \end{figure}

    Define $h$ as the radius of the overlap (see Figure \ref{Fig6}). Then, using Pythagoras Theorem we get that 
    \begin{equation}
        h^2 = \delta^2 - \left(\frac{\varepsilon}{2} \right)^2 =\left( {N\eps} - \frac{\varepsilon}{2} \right)^2 - \left(\frac{\varepsilon}{2} \right)^2 =  N(N-1)\eps^{2}>0.
    \end{equation}
    Then, we have that the segment $[a,\hat b]$ is 
    included in $\ell^{1,\varepsilon}(K_{\eps})$, and in particular the segment $[a,b]$ we were interested in is inside $\ell^{1,\varepsilon}(K_\eps)$. This proves that $\ell^1(K) \subset \ell^{1,\varepsilon}(K_\eps)$. 

    To prove that $\mbox{co}(K) \subset \mbox{co}_\varepsilon(K_\eps)$, we are going to iterate the previous argument to obtain that $\ell^k(K) \subset \ell^{k,\varepsilon}(K_\eps)$. Since by Proposition \ref{re:ellcharac} we know that $\mbox{co}(K) = \ell^N(K)$, there is a finite amount of iterations and our choice of $K_\eps$ as
    the set $K + B_{N\eps}(0)$ is enough (notice that in $\RR^N$ we need $N-1$ iterations).
\end{proof}

Let us treat first the situation in which $\Omega_0$ is $\varepsilon$-convex. This particular case is interesting because the set $\Omega_0$ preserves the $\varepsilon$-convexity through its evolution via the game,
that is, the positivity set of the value of the game after $k$ rounds of the game is
still $\varepsilon$-convex. To simplify the notation let us call $\Omega^\eps_{k}$  (instead of $\Omega^{\varepsilon}_{t=k\eps^{2}/2}$) the positivity set of the value of the game, $u^\eps$, after $k$ plays.

\begin{remark}\label{remark1}
	Note that, when $\overline{K}\subset \Omega_{0}$,  it holds that $K_{\varepsilon}\subset \Omega^{\varepsilon}_{k}$ for all $\varepsilon>0$ small enough and all $k\in\mathbb{N}$. It is due to the fact that, if $x\in K_{\varepsilon}$, Paul can always choose to stop the game and receives the positive payment $\psi_{\varepsilon}(x)$.
\end{remark}

\begin{lemma}[Preservation of $\varepsilon$-convexity]\label{9} 
	Let $\Omega_0$ be $\varepsilon$-convex. Then, $\Omega^\eps_{k}$ is $\varepsilon$-convex for any $k \in \mathbb{N}$ and moreover, 
	\[K_{\eps}\subset \mbox{co}_{\varepsilon}(K_{\varepsilon}) \subset \cdots \subset \Omega^\eps_{k} \subset \Omega^\eps_{(k-1)} \subset \cdots \subset \Omega_0.  \]
\end{lemma}

\begin{proof}
 Assume $\Omega^\eps_{k-1}$ is $\varepsilon$-convex. Let us show that $\Omega^\eps_k$ is $\varepsilon$-convex and that $\Omega^\eps_k \subset \Omega^\eps_{k-1}$. Let us start with the latter. If $x \in \Omega^\eps_{k}$, we have that $u(x, k\varepsilon^2/2)>0$. According to the DPP, this implies that either
	\[0<\sup_{|v|=1} \min_{\pm } u(x\pm \varepsilon v, (k-1)\varepsilon^2/2) \]
	or 
	\[0<\psi_\varepsilon(x). \]
	This second case directly implies that $x \in K_\varepsilon \subset \Omega^{\eps}_{k}$ due to Remark \ref{remark1}.  Then, let us assume that we are in the first case. If that supremum is bigger than zero, we have that there is some vector $v_0$ such that 
	\[0 <\min_{\pm } u(x\pm \varepsilon v_0, (k-1)\varepsilon^2/2).\]
	According to this expression, both $x+v\varepsilon,x-\varepsilon v$ belong to $\Omega^\eps_{k-1}$. Since $\Omega^\eps_{k-1}$ is $\varepsilon$-convex by assumption, we also have that $x \in \Omega^\eps_{k-1}$. Thus, $\Omega^\eps_{k} \subset \Omega^\eps_{k-1}$. 
	
	Now let us check that $\Omega^\eps_k$ is also 
	$\varepsilon$-convex. Choose $x,y \in \Omega^\eps_k$, and $z\in \ell^\varepsilon_{x,y}  \setminus  ( \{x\} \cup \{y\} )$ (if this set is empty, the conclusion is trivial). We want to check that $z \in \Omega^\eps_k$. Since $\Omega^\eps_{k} \subset \Omega^\eps_{k-1}$, both $x,y \in \Omega^\eps_{k-1}$. Using the $\varepsilon$-convexity of this set, we get that $\ell^\varepsilon_{x,y} \subset \Omega^\eps_{k-1}$. Choosing $v_0 = \frac{x-y}{|x-y|}$, this implies that
	\[ u(z,k\varepsilon^2) \geq \max_{|v|=1} \min_{\pm} u(z \pm \varepsilon v,(k-1)\varepsilon^2/2) \geq \min_{\pm} u(z \pm \varepsilon v_0,(k-1)\varepsilon^2/2) > 0 \]
	Hence, we conclude that $z \in \Omega^\eps_k$, making the set $\Omega^\eps_k$ $\varepsilon$-convex. 
\end{proof}

Now we prove the main result of this subsection.
\begin{lemma}\label{8}
	Suppose that $\textrm{co}_{\varepsilon}(K_{\varepsilon})\subset\Omega_{0}.$ Then,
	\begin{equation}
		\textrm{co}_{\varepsilon}(K_{\varepsilon})\subset\Omega_{k}^{\varepsilon}\;\;\textrm{for every}\;\;k\in\mathbb{N}.
	\end{equation} 
\end{lemma}

\begin{proof}
To prove that $\text{co}_{\varepsilon}(K_{\varepsilon})\subset\Omega_{k}^{\varepsilon}$ for each $k\in\mathbb{N}$ is 
equivalent to show that for every time $t=k\varepsilon^{2}/2$, there exists a strategy for Paul such that he receives a positive payment regardless of Carol's choices for $x\in \textrm{co}_{\varepsilon}(K_{\varepsilon})$, as a consequence of the existence of this strategy, playing with the optimal strategy Paul also wins which implies that  $u^{\varepsilon}(x,k\varepsilon^{2}/2)>0.$

	Suppose $x\in\textrm{co}_{\varepsilon}(K_{\varepsilon})\setminus  K_{\varepsilon}$. Otherwise, if $x\in K_{\varepsilon}$, Paul automatically wins (he can stop the game at this point) and receives the payment of the obstacle. Then, if $x\in\textrm{co}_{\varepsilon}(K_{\varepsilon})\setminus  K_{\varepsilon}$, by the characterization in Proposition~\ref{re:ellcharac} of the $\varepsilon$-convex hull, there exists an integer $n>0$ such that $x\in\ell^{n,\varepsilon}(K_{\varepsilon}).$ Then, there are $a,b\in \ell^{n-1,\varepsilon}(K_{\varepsilon})\subset\textrm{co}_{\varepsilon}(K_{\varepsilon})$ such that $x\in\ell^{\varepsilon}_{a,b}$. Based on this information, Paul selects the vector $v=(b-a)/|b-a|$ because that choice ensures the next point $x_{1}\in \ell^{\varepsilon}_{a,b}$ no matter what Carol does. Then, if Paul and Carol only play once, $k=1$, Paul will receive a positive payoff because $x_{1}\in \textrm{co}_{\varepsilon}(K_{\varepsilon})$. That implies Paul will either get money from the obstacle if $x_{1}\in K_{\varepsilon}$ or from the initial data, since $\textrm{co}_{\varepsilon}(K_{\varepsilon})\subset\Omega_{0}.$
	
	If $k\not = 1$, Paul repeats the same strategy for $x_{1}$ choosing $v=(b-a)/|b-a|$. Suppose that Carol always chooses the sign for the next the movement such that $x_{j}\in \ell^{\varepsilon}_{a,b}\setminus\{a,b\}$ for each $j\in\{1,...,k\}.$ Consequently, at the end of the game Paul will receive a positive payment since $x_{j}\in\ell^{n}_{a,b}\subset\Omega_{0}.$ On the other hand, suppose that for some $j\in\{1,....,k-1\}$, $x_{j}=a$ (or $x_{j}=b$), then $x_{j}\in \ell^{n-1,\varepsilon}(K_{\varepsilon}).$ Thus, in the next turn, Paul's choice will be to select the vector $v=(b_{j}-a_{j})/|b_{j}-a_{j}|$ where $a_{j},b_{j}\in\ell^{n-2,\varepsilon}(K_{\varepsilon}$ and $x_{j}\in\ell^{\varepsilon}_{a_{j},b_{j}}\subset \textrm{co}_{\varepsilon}(K_{\varepsilon})$. Once again, if the game concludes at this stage, Paul will obtain a positive reward since $x_{j+1}\in \ell^{\varepsilon}_{a_{j},b_{j}}\subset \textrm{co}_{\varepsilon}(K_{\varepsilon})$. However, if the game continues and $x_{j+1}\not\in K_{\varepsilon}$ either Paul can keep the same strategy ensuring he wins when the time expires or Paul reaches $K_{\varepsilon}$ and he also wins in this case.
\end{proof}

As a consequence we get the following corollary.

\begin{corollary} Assume that 
$\textrm{co}_{\varepsilon}(K_{\varepsilon})\subset\Omega_{0},$	then
$$
		\textrm{co}_{\varepsilon}(K_{\varepsilon})\subset
		\liminf_{k\to\infty}\Omega^\eps_{k}.
	$$ 
Moreover, when $\Omega_{0}$ is $\varepsilon$-convex, it holds that
	\begin{equation}
	\textrm{co}_{\varepsilon}(K_{\varepsilon})\subset	 \lim_{k\to\infty}\Omega^\eps_{k}.
	\end{equation}
\end{corollary}

\begin{proof} Since we have
$\textrm{co}_{\varepsilon}(K_{\varepsilon})\subset\Omega_{0},$ 
thanks to Lemma \ref{8}, we get
	$$
		\textrm{co}_{\varepsilon}(K_{\varepsilon})\subset 	\liminf_{k\to\infty}\Omega^\eps_{k}. $$
Moreover, if $\Omega_{0}$ is $\varepsilon$-convex, by Lemma \ref{9}, using that $\{\Omega_{k}^{\eps}\}_{k\in\mathbb{N}}$ is a monotone sequence of sets, we obtain that
	\begin{equation}
		 \liminf_{k\to\infty}\Omega^\eps_{k}=
		 \limsup_{k\to\infty}\Omega^\eps_{k}.
	\end{equation}
\end{proof}
Observe that we are not able to prove that $\displaystyle \limsup_{k\to\infty}\Omega^{\varepsilon}_{k}\subset \textrm{co}_{\varepsilon}(K_\eps)$ with a fixed $\varepsilon$. However, in the next subsection, we will be able to prove that $\displaystyle \limsup_{t\to\infty}\Omega_{t} \subset \textrm{co}(K)$ and conclude $\displaystyle \lim_{t\to\infty}\Omega_{t} =\textrm{co}(K)$.

\subsection{Convex hull of $K$ not included in $\Omega_0$} For an open obstacle $K$ and an initial open set $\Omega_0$, we 
defined in the introduction of this paper the graph $G=(\mathcal{V},\mathcal{E})$ as follows:  the set of vertices is given by
$$
	\mathcal{V}:= \Big\{ V \subset \mathbb{R}^N : V \text{ is a connected component of } K \Big\},
	$$
	and the set of edges by 
	$$
	\mathcal{E}:= \Big\{ \langle V_1, V_2 \rangle : V_1, V_2 \in \mathcal{V} \text{ and } \ell_{x,y} \subset \Omega_0 \text{ for some } x\in V_1 \text{ and some } y\in V_2  \Big\}. 
$$
Let us also define the open set 
\begin{equation}
	L = \Big\{ x \in \ell_{a,b} : a,b \in K, \ell_{a,b} \subset \Omega_0 \Big\},
\end{equation}
which is a set analogue of the graph $G$. It is clear that $\mbox{co}(K) = \mbox{co}(L)$, since we have that $K \subset L\subset \ell^1(K)$. Observe that if $G$ is connected, then $L$ is connected. That the graph $G$ is connected is going to be the main hypothesis of this subsection.

We start by showing the following elementary result. 

\begin{lemma} \label{lemmacountablecomponents}
	Let $K \subset \RR^N$ be an open set. Then, $K$ can only have at most a countable amount of connected components. 
\end{lemma}
\begin{proof}
	For each connected component of $K$, which are open, choose an element of $\mathbb{Q}^N$ belonging to the connected component. Since every one of these rational elements can only belong to one connected component, there can be at most a countable amount of connected components.
\end{proof}

The proof of the result in \cite{Misu} requires to choose a sufficiently small $\varepsilon$, depending on the starting position $x$: the reason for doing that is that if the jumps are large enough, it may happen that you jump over the set you want to land on and then the proof would not work anymore. This fact (that one has to choose $\varepsilon$ depending on the starting position $x$) prevents the proofs in \cite{Misu} to shed some light on the asymptotic behaviour for large times of the sets $\Omega^\eps_t$.
At this crucial point we have to take into account that we enlarged the obstacle set, since this will help Paul to win.

What we are going to show now is that 
\begin{equation} \label{first}
    \mbox{co}(K) \subset \liminf_{t \to \infty} \Omega_t^\varepsilon.
\end{equation}
What we would actually like to prove is that 
\begin{equation}
    \mbox{co}_{\varepsilon}(K_{\varepsilon}) \subset \liminf_{t \to \infty} \Omega_t^\varepsilon,
\end{equation}
but there are several difficulties in order to show this.  
We have proved \eqref{first} in the previous subsection under the additional assumption that $\mbox{co}(K) \subset \Omega_0$. When this hypothesis is not in force, we want an argument that takes into the account both the concentric strategy used in \cite{Misu} and the $\varepsilon$-convex hull of $K_{\varepsilon}$, something that is missing for the moment.  

\begin{remark}
    Once we prove that Paul has a winning strategy in a certain region, we can may assume without loss of generality that that region is part of the obstacle. For example, we may assume the connected components of the obstacle $K$ to be convex, or that $L$ is the obstacle, instead of $K$, and thus we can consider the whole obstacle as a connected
    set. 
\end{remark}

From this point onwards, we have to differentiate between the proof for $N=2$ and the proof for $N\geq 3$. This is due to the fact that the proof for $N=2$ is more natural for this game, while the proof for $N\geq 3$ is a generalization of the two-dimensional one. In fact, when $N=2$ we will be able to say more than in higher dimensions, since in this case we will be able to choose a fixed $\varepsilon$, uniform for every $x \in \textrm{co}(K)$. The two-dimensional proof is a slight refinement of what is done in \cite{Misu}, while the proof in higher dimensions is completely new. 

Let us start with the case of $N=2$. We state a well-known property of the convex hull of a path-connected set in two dimensions.

\begin{lemma}\label{re:pathconnectedconvex}
    If $L \subset \RR^2$ is a path-connected set, then $\mbox{co}(L) = \ell^1(L)$.
\end{lemma}

This next auxiliary result will be needed in the proof. 

\begin{lemma} \label{re:cortasegmentos}
    Let $K_1, K_2 \subset \RR^2$ be two connected sets, and let $(a,b) \in K_1 \times K_2$. If $r$ is a line that intersects $\ell_{a,b} \setminus (K_1 \cup K_2)$, then for any pair $(x,y) \in K_1 \times K_2$ the line $r$ intersects the segment $\ell_{x,y}$. 
\end{lemma}

\begin{proof}
    Assume without loss of generality that the line $r$ is the $x$-axis, $\{(x,y) \subset \RR^2: y=0\}$. 
    \begin{figure}[H]
        \centering
        \includegraphics[width=11cm]{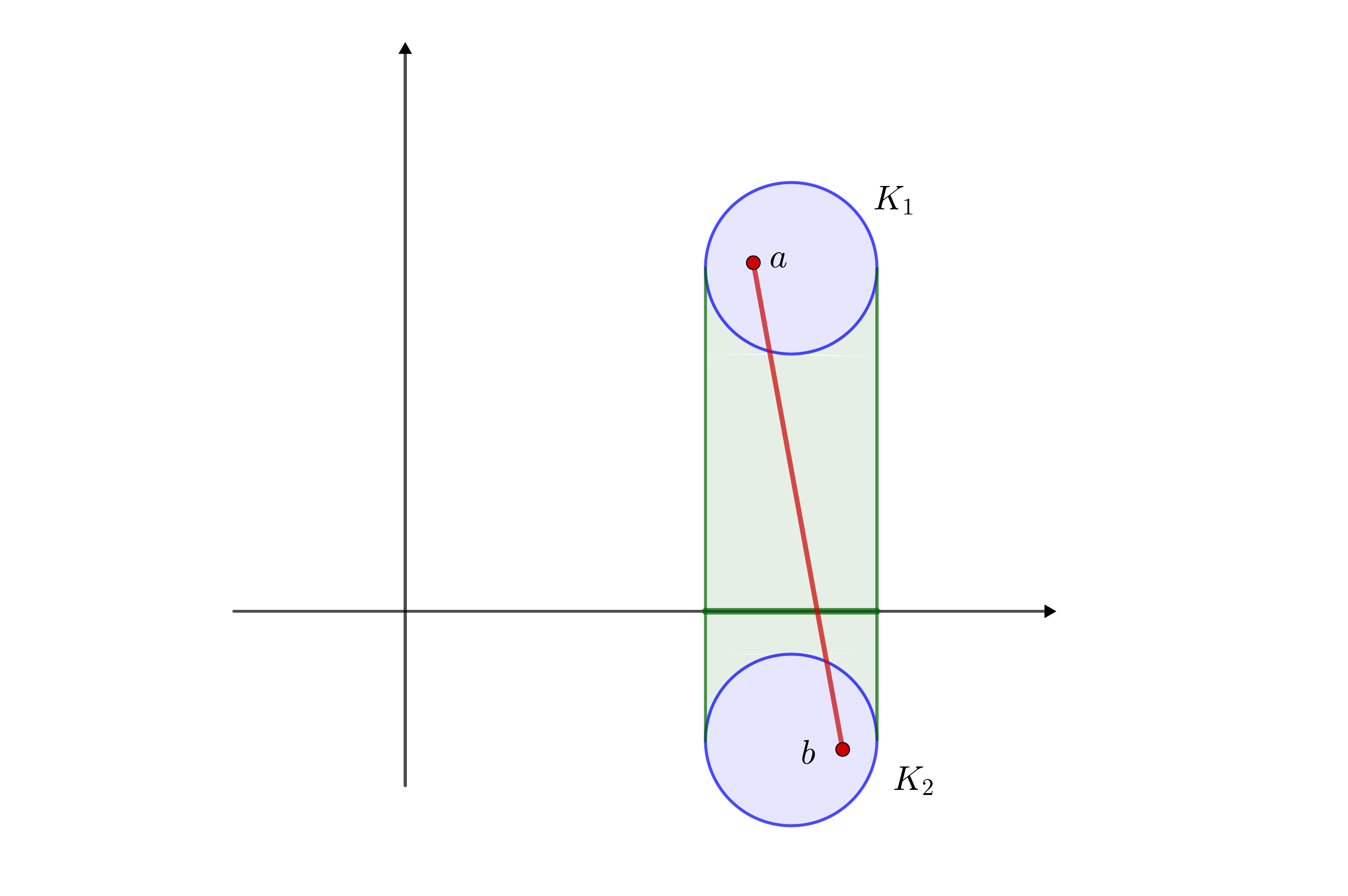}
        \caption{Scheme of the proof.}
    \end{figure}
    Then, if $K_1$ and $K_2$ do not touch the $x$-axis, they must lie either inside $H^+ = \{(x,y) \in \RR^2: y>0\}$ or inside $H^- \{(x,y) \in \RR^2: y<0\}$. If both of them are in the same set, it is impossible that any segment  would intersect the $x$-axis. So we can assume that $K_1 \subset H^+$ and $K_2 \subset H^-$. Therefore, due to Bolzano's Theorem, for any couple of points $(a,b) \in K_1 \times K_2$, the segment $\ell_{a,b}$ joining them must intersect the $x$-axis. 
\end{proof}

Let us now write the main result of this section for two space dimensions, $N=2$. 
\begin{proposition} \label{misudim2}
    Assume we are in $\RR^2$ and condition \eqref{(G)} be satisfied, that is, let the graph $G$ be connected. Then, 
    \begin{equation}
        \mbox{co}(K) \subset \liminf_{t \to \infty} \Omega_t^\varepsilon,
    \end{equation}
    for every $\eps$ small enough.
\end{proposition}
\begin{proof} 
    The proof follows the ideas in \cite{Misu}, but with a few changes to take into account that the enlarged obstacle set, $K_\varepsilon$, is going to help us. The idea is that we want to prove that the asymptotic limit $t\to \infty$ can be taken for a fixed $\varepsilon$ small enough, independently of the starting position $x \in \textrm{co}(K)$. 

    \textsc{Case $x \in K$.} In this case it is clear that 
    \[u^\varepsilon(x,t)\geq \min_{x \in K}\psi_\varepsilon(x)>0\quad \text{ for } t>0.\] 

   \textsc{Case  $x \in L\setminus K$.} In this case, by the definition of $L$, there exists $a,b$ in two different connected components of $K$, such that $x\in \ell_{a,b}\subset \Omega_0$. With this in mind, an clear strategy for Paul is to move along the segment $\ell_{a,b}$. The only consideration needed is to observe that, since $\varepsilon< 2\eps$, Paul will not jump over the enlarged set $K_\eps$ and can always stop inside. We get that 
    \[u^\varepsilon(x,t) \geq \min\left\{\min_{y \in K + B_\varepsilon} \psi_\varepsilon(x), \min_{y \in \ell_{a,b}} u_0(y)  \right\} > 0 \quad \text{ for } t>0.\] 
    So far it has not been necessary to take $\varepsilon$ to be small.

	\textsc{Case  $x \in \mbox{co}(K) \setminus L$.} In this the case we have to slightly alter the arguments from the ones used in~\cite{Misu}. Recall that $\mbox{co}(K) = \mbox{co}(L) = \ell^1(L)$, see~Lemma~\ref{re:pathconnectedconvex}. 
	
	Let us give a sketch of the argument. Since $x \in \ell^1(L)$, there exists $a,b \in L$ such that $x \in \ell_{a,b}$. Now, the essential part of the argument lies in the fact that you can build a polygonal path in $L$ between $a$ and $b$ and, this polygonal, together with the segment $\ell_{a,b}$, form a closed curve. We now perform a concentric argument for Paul so that he moves away from the segment $\ell_{a,b}$ and towards the polygonal. When the position lies close enough to the polygonal, Paul wins, since the problem reduces to one of the previous cases. 
	
	This approach has one disadvantage, which is that for every $x$ the polygonal chosen may be different. And if this polygonal lies too close to the boundary of $\Omega_0$, we have to take the step $\varepsilon$ appropiately small enough so that Paul does not jump over the set. Therefore, the step $\varepsilon$ depends on the initial position $x$. If we chose a polygonal path that is valid for every $x$, then we avoid this problem and the $\varepsilon$ step, if chosen small enough, can be valid for every $x$. 

    \begin{figure}[H]
        \centering
        \includegraphics[width=9.6cm]{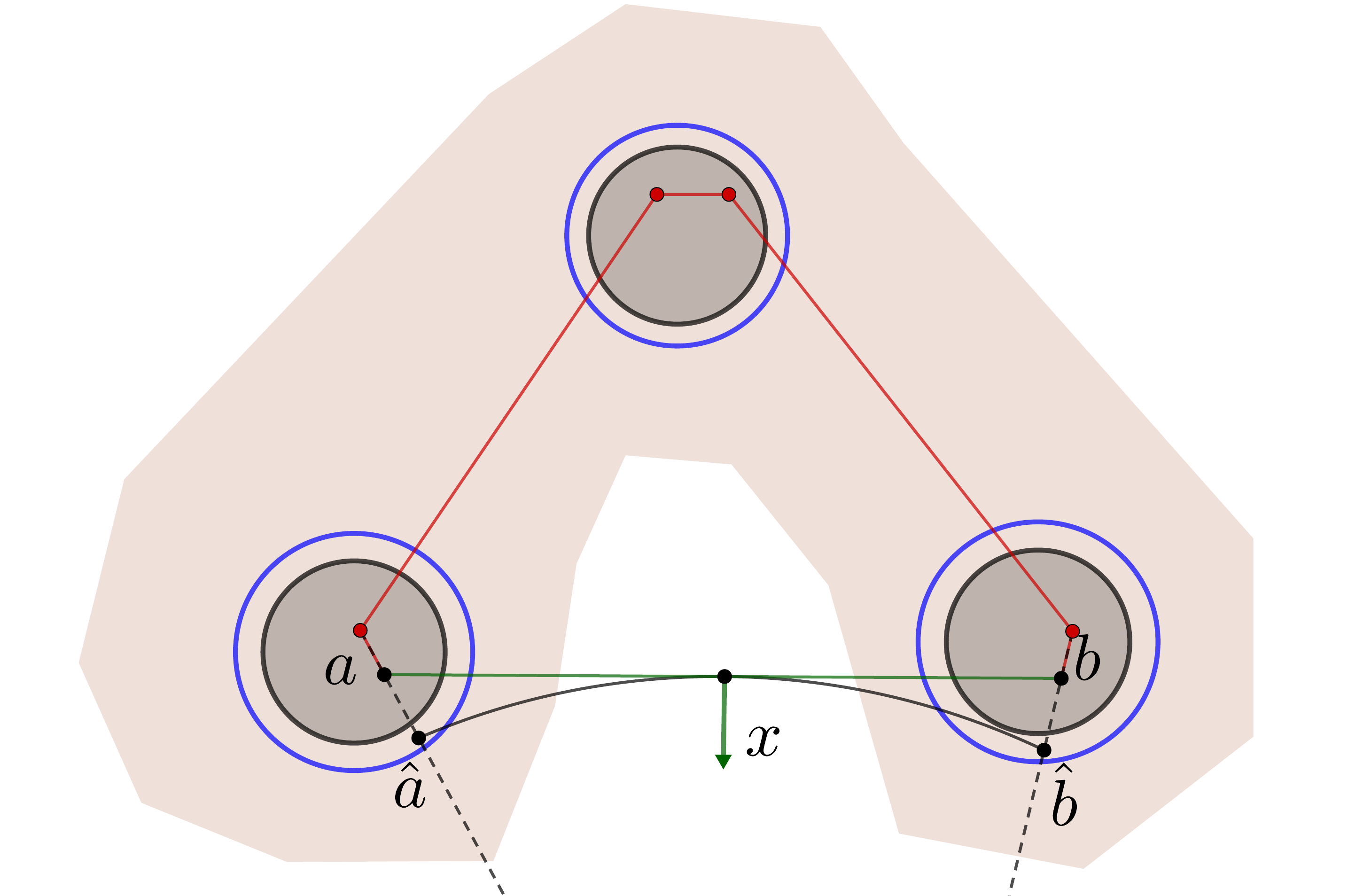}
        \caption{Scheme of the proof. }
    \end{figure}

    To build this polygonal, recall that $K$ only has a countable amount of connected components $K_j$ with $j \in \mathbb{N}$ (see Lemma \ref{lemmacountablecomponents}). Then for each pair of connected components $K_i$, $K_j$ with $i \neq j$ such that there is a segment $\ell_{y_i, y_j} \subset \Omega_0$ with $y_i \in K_i$ and $y_j \in K_j$, take this segment $\ell_{y_i, y_j}$ and consider the set $\Gamma_L$ of all the points in these segments, that is, $\Gamma_L = \cup \ell_{y_i, y_j}$.

	This $\Gamma_L$ connects all of the connected components $K_j$, but inside of each $K_j$ it is not necessarily connected. We add the required polygonal paths inside of each $K_j$, which we call $\Gamma_K$, to make $\Gamma \coloneq \Gamma_L  \cup \Gamma_K$ a connected polygonal path. 
	The polygonal $\Gamma$ is a compact subset of $L$, so there 
    exists $\delta>0$ such that $B_{3\delta}(\Gamma) \subset L$. Now we have to choose $\varepsilon < \delta$ to make sure Paul cannot jump over the set. Let us also assume that $\delta < 2
	\eps$. Notice that
    $\delta $ is choosed depending on $\eps$ and that such a choice is possible for every $\eps$ small enough.
    
    As we mentioned before, for any $x \in \mbox{co}(K)\setminus L$ there exists $a,b \in L$ such that $x \in \ell_{a,b}$ due to $\textrm{co}(K)=\ell^{1}(L)$. We are going to show that $a,b$ can be chosen either in $\Gamma$ or in $K$, that is, $x\in \ell^1(\Gamma \cup K)$. Assume, $a$ is in $L \setminus K$. Then $a$ belongs to a segment joining two connected components without intersecting them. Hence, thanks to Lemma \ref{re:cortasegmentos}, we can change $a$ to be  in $\Gamma$, and we still have $x \in \ell_{a,b}$. We perform the same change for $b$ if necessary. If either $a$ or $b$ lie in $K$, we add the necessary segments inside of $K$ to the polygonal $\Gamma$ so that it connects with them, and discard unnecessary segments so that the polygonal path starts in $a$ and ends in $b$. 

	Now we have the polygonal path joining $a,b$ and lying inside of $L$ ready. Notice that, without loss of generality, we may assume that $\Gamma$ and $\ell_{a,b}$ only intersect at $a$ and $b$. If we modify $\Gamma$ extend the segment joining $\Gamma$ with $a$ and $b$ with a length $\delta$, we obtain two new corresponding endpoints, $\hat a$ and $\hat b$. We define $\hat \Gamma$ as the modified polygonal and $C$ as the arc that passes through the points $\hat a$, $x$ and $\hat b$. Bear in mind that if $a$ or $b$ were in $K$, $\hat a$ or $\hat b$ may lie in $K_\varepsilon$. Observe that $\hat \Gamma \cup C$ form a closed curve. Therefore, there is a Jordan closed curve $\gamma \subset \hat \Gamma \cup C$, such that $x \in \gamma$. Then, thanks to the Jordan curve theorem, there exists a bounded open set $D_\gamma$ enclosed by the curve $\gamma$. Note that the exterior normal vector of $C$ points towards the set $D_\gamma$.

  	From this point the proof follows using a concentric argument for Paul. The arc $C$ has a center $z$ which we take as the center for the concentric strategy. Since $D_\gamma$ is a bounded set, it takes at most finite time $\tau$ such that $D_\gamma \subset B_{R_\tau}(z)$, for 
	\begin{equation}
		R_\tau = \sqrt{|x-z|^2 + 2\tau}
	\end{equation}
	Before this time $\tau$, we can be sure that Paul has arrived at a neighbourhood of $\Gamma$ and thus lies inside of $L$. Therefore, we are now in one of the previous cases and have concluded the proof. 
\end{proof}

Now, let us try to generalize the previous proof for $N=2$ to any dimension $N\geq 3$. The fundamental idea lies in the fact that Paul can always play in a two-dimensional plane as if we were in $\RR^2$ by choosing vectors contained in that plane. In order to use this, we develop a different approach to the convex hull of a set that allows us to extend the proof from $\RR^2$ to $\RR^N$.

Let us try to give a sketch of the proof in $\RR^3$. The simplest geometrical example in $\RR^2$ where we can see how the previous proof works is a triangle. If we lie inside a triangle without one of the sides, then Paul will be able to move from the interior of the triangle to one of the other sides. Hence, starting from the convex hull of a triangle without one of the sides, Paul can find a strategy to win. In $\RR^3$, if we think of a tetrahedron, the idea is to move from the interior to one of the faces, and then reduce to two dimensions and move to the sides. Observe that thinking about the previous argument in terms of triangles greatly simplifies the geometry of the problem. The intuition behind our adaptation relies on this idea.  

We mention again one important difference between the results of $N=2$ and $N\geq 3$: in two dimensions we can give the result for a fixed $\varepsilon>0$, but for higher dimensions we cannot. This is the main difference between Theorems \ref{Th.Cadenas} and \ref{cor2}, and it occurs because in the proof of Proposition~\ref{misudim2} we can choose the same polygonal path for every $x$. This type of argument clearly does not work for $N \geq 3$.  

Let us define, for a set $\Gamma$, the new set
\begin{equation}
	\mathcal{T}^1(\Gamma) := \Big\{z\in \mathbb{R}^N: z \in \ell_{x,y}, \text{ with } x \in \ell_{a,b}\subset \Gamma \text{ and } y\in \ell_{a,c} \subset \Gamma, \text{ for } a\neq b \text{ and } a \neq c \Big\}.
\end{equation}
This definition is relevant when the set $\Gamma$ is polygonally connected. It is easy to observe from this construction that $\mathcal{T}(\Gamma)$ is open whenever $\Gamma$ is open. 

 \begin{figure}[H]
	\centering
	\includegraphics[width=13cm]{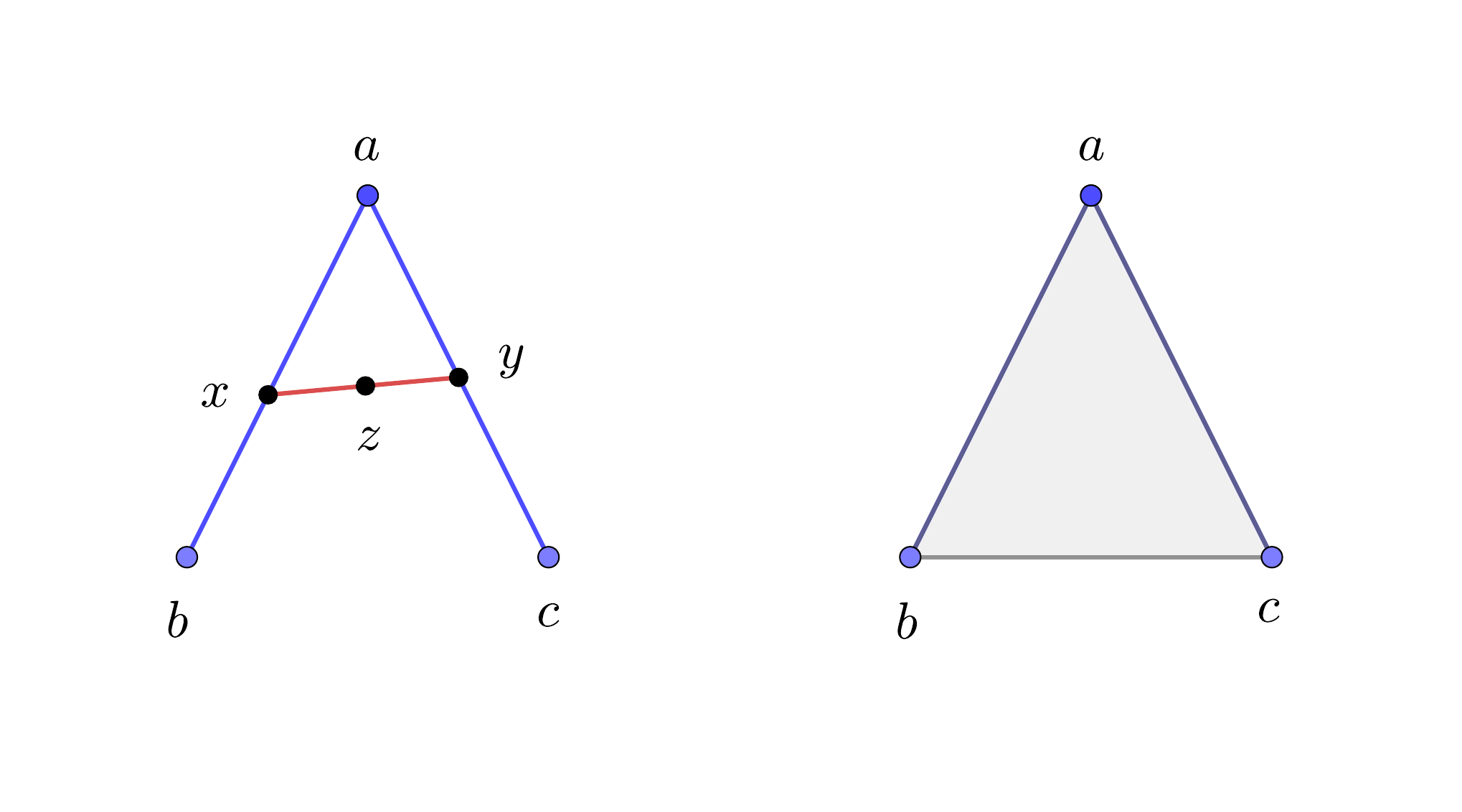}
	\caption{Construction of $\mathcal{T}^1(\Gamma)$.}
\end{figure}

Let us define $\mathcal{T}^2 (\Gamma) = \mathcal{T}^1(\mathcal{T}^1(\Gamma))$. Also define $\mathcal{T}^0(\Gamma) = \Gamma$. Let 
$$\mathcal{T}(\Gamma) =  \bigcup_{n=0}^\infty \mathcal{T}^n(\Gamma).$$ 
We will show that if Paul starts in $\textrm{co}(K)$ plays with a strategy that bears in mind the iterations $\mathcal{T}^n(K)$, he will be able to win the game. It is immediate to observe that $\mathcal{T}^n(\Gamma) \subset \mbox{co}(\Gamma)$, since this is a more restrictive construction than the one of the convex hull, so $\mathcal{T}(\Gamma) \subset \mbox{co}(\Gamma)$. What we need to prove now is that $\mbox{co}(\Gamma) \subset \mathcal{T}(\Gamma)$. 

First, let us show that these sets preserve polygonal connectedness.

\begin{lemma} \label{polygonallemma}
	Let $\Gamma$ be a polygonally connected set. Then, $\mathcal{T}^1(\Gamma)$ is poligonally connected.
\end{lemma}
\begin{proof}
	Let $x,y \in \mathcal{T}^1(\Gamma)$. This implies that there is a segment joining $x$ with a point $b_x \in \Gamma$ and a segment joining $y$ with a point $b_y$ in $\Gamma$. By hypothesis, there is a polygonal joining both $b_x$ and $b_y$, so we can extend this polygonal to join $x$ and $y$.
\end{proof}

Now let us prove an auxiliary lemma that will be useful. 

\begin{lemma} \label{firstllemma}
	Let $\Gamma$ be a polygonally connected set. Then, $$\ell^1(\Gamma) \subset \mathcal{T}(\Gamma).$$
\end{lemma}

\begin{proof}
	Let $x, y \in \Gamma$. Since $\Gamma$ is polygonally connected, there is a polygonal joining $x$ and $y$. Let $k$ be the number of segments in the polygonal joining both points. We can see by an inductive argument that the segment joining $x$ and $y$ is necessarily in $\mathcal{T}^{k-1}(\Gamma) \subset \mathcal{T}(\Gamma)$. Thus, $\ell^{1}(\Gamma) \subset \mathcal{T}(\Gamma)$. 
\end{proof}

\begin{lemma} \label{dobleiteracion}
	Let $\Gamma$ a polygonally connected set. Then, 
	$$\mathcal{T}(\Gamma) = \mathcal{T}(\mathcal{T}(\Gamma)).$$
\end{lemma}

\begin{proof}
	That $\mathcal{T}(\Gamma) \subset \mathcal{T}(\mathcal{T}(\Gamma))$ is obvious. We want to prove the reciprocal. It suffices to prove that $\mathcal{T}^1(\mathcal{T}(\Gamma)) = \mathcal{T}(\Gamma)$. If $x\in \mathcal{T}^1(\mathcal{T}(\Gamma))$, then this implies there exist $y,z \in \mathcal{T}(\Gamma)$ such that $x$ is in $\ell_{y,z}$, and there are $a,b,c \in \mathcal{T}(\Gamma)$, $b\not =a\not =c$, such that $y \in \ell_{a,b} \subset \mathcal{T}(\Gamma)$ and $z\in \ell_{a,c} \subset \mathcal{T}(\Gamma)$. But if $a,b \in \mathcal{T}(\Gamma)$, then there exist  $k_{1},k_{2} \in \mathbb{N}$ such that $a,b \in \mathcal{T}^{k_{1}}(\Gamma)$ and $a,c \in \mathcal{T}^{k_{2}}(\Gamma)$ . Recall that Lemma ~\ref{polygonallemma} implies $\mathcal{T}^{k}(\Gamma)$ is polygonally connected for any $k\in\mathbb{N}.$ Then, just like in the proof of  Lemma~\ref{firstllemma}, there is a polygonal in $\mathcal{T}^{k_{1}}(\Gamma)$ joining $a$ and $b$ and another joining $a$ and $c$ in $\mathcal{T}^{k_{2}}(\Gamma)$, so we know that for $m\in\mathbb{N}$ such that $m\geq  \max\{k_{1},k_{2}\}$ both $\ell_{a,b}$ and $\ell_{a,c}$ lie inside $\mathcal{T}^m(\Gamma)$.  Hence, $x\in \mathcal{T}^1\left( \mathcal{T}^m(\Gamma) \right) = \mathcal{T}^{m+1}(\Gamma) \subset \mathcal{T}(\Gamma)$. 
\end{proof}

With these two lemmas we are ready to show that $\mbox{co}(\Gamma) \subset \mathcal{T}(\Gamma)$.

\begin{corollary} \label{re:convexhullsubset}
	Let $\Gamma$ be a polygonally connected set. Then, $$\textrm{co}(\Gamma) \subset \mathcal{T}(\Gamma).$$
\end{corollary}

\begin{proof}
	We know from Proposition \ref{re:ellcharac} that $\mbox{co}(\Gamma) = \ell^{N}(\Gamma)$. Recall $\ell^2(\Gamma) = \ell^1(\ell^1(\Gamma))$. Notice that, thanks to  Lemma \ref{firstllemma} and Lemma \ref{dobleiteracion} we get $\mathcal{T}(\Gamma) = \mathcal{T}(\ell^1(\Gamma))$, using that $\mathcal{T}(\ell^1(\Gamma)) \subset \mathcal{T}\left(\mathcal{T}(\Gamma) \right)$. Then, thanks to  Lemma \ref{firstllemma} again, we get that $\ell^2(\Gamma) \subset \mathcal{T}(\Gamma)$. Iterating this argument we obtain that $\ell^k(\Gamma) \subset \mathcal{T}(\Gamma)$, so $\mbox{co}(\Gamma)\subset \mathcal{T}(\Gamma)$. 
\end{proof}

\begin{figure}[H]
	\centering
	\includegraphics[width=16cm]{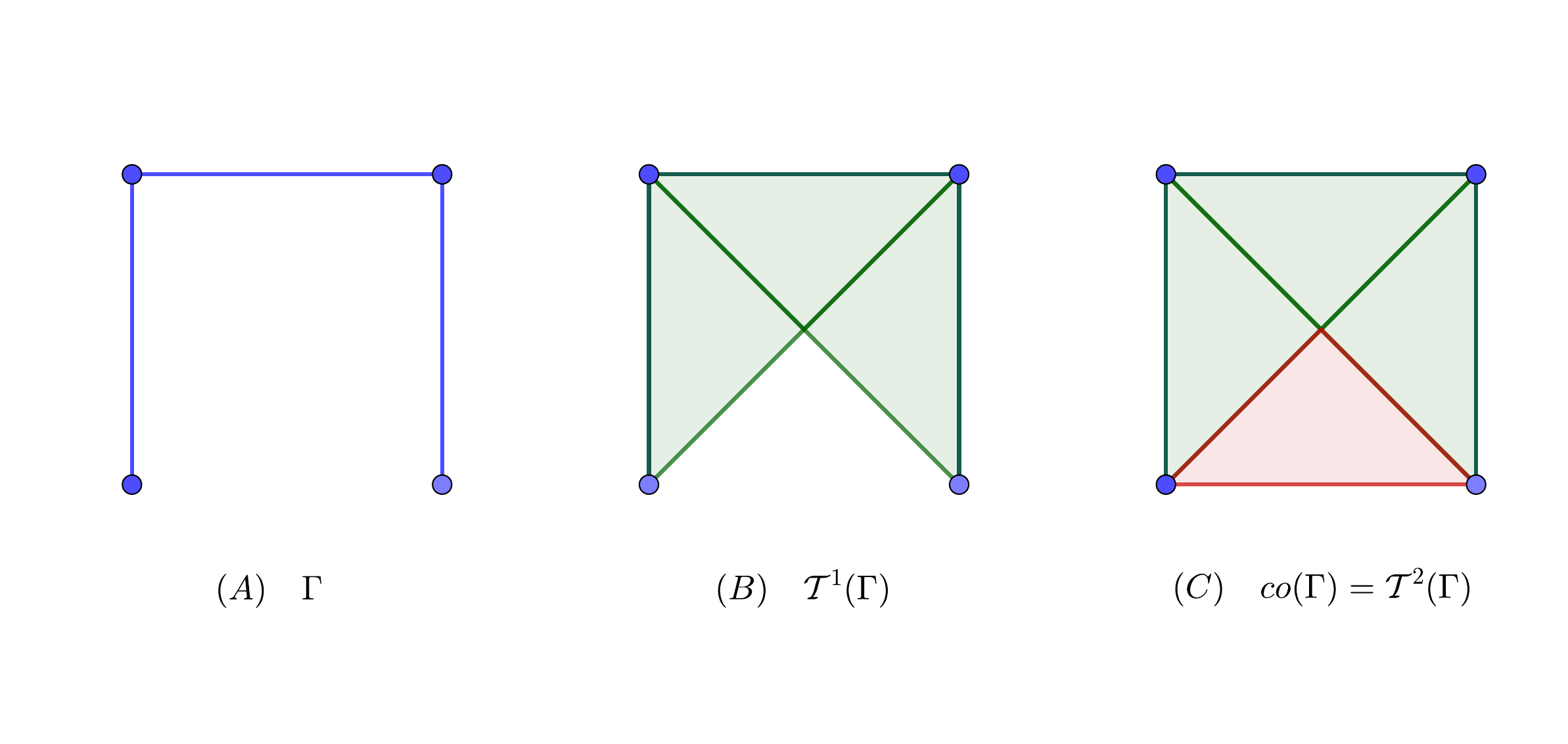}
\caption{Recovering $\text{co}(\Gamma)$ through iterations of $\mathcal{T}^k(\Gamma)$.}
\end{figure}

We now have the tools to prove the main result concerning the asymptotic behaviour of our problem, that is, $\textrm{co}(K)\subset\liminf_{t \to \infty}\Omega_{t}$, under condition \eqref{(G)}. Recall that in this case we cannot fix a small $\varepsilon$ for every starting point $x$, since the argument in Proposition \ref{misudim2} relies on choosing a polygonal valid for every $x \in \text{co}(K)$, but this can only be done in $\RR^2$.   

\begin{theorem} \label{ThmMisu1}
	Let condition \eqref{(G)} be satisfied, that is, let the graph $G$ be connected. Then 
	\begin{equation}
		co(K) \subset \liminf_{t \to \infty} \Omega_t.
	\end{equation}
\end{theorem}   

\begin{proof}
	First, recall that $G$ being connected implies that $L$ is connected. Since $L$ is also open, it follows that it is polygonally connected. Let $x \in \mbox{co}(K) = \mbox{co}(L)$. This implies, due to  Corollary~\ref{re:convexhullsubset} that $x\in \mathcal{T}(L)$, which means $x\in \mathcal{T}^k(L)$ for some large enough $k \in \mathbb{N}$. Let us show that if $x\in \mathcal{T}^k(L)$ for $k \in \mathbb{N}$, then the Proposition \ref{misudim2} shows that starting from $x$ Paul can win. We just need to prove it for $\mathcal{T}^{1}(L)$, since $\mathcal{T}^{k}(L) = \mathcal{T}^{1}\left( \mathcal{T}^{k-1}(L) \right)$. If $x \in \mathcal{T}^{1}(L)$, $x$ is inside a segment with extreme points $y$ and $z$, where $y \in \ell_{a,b}$ and $z \in \ell_{a,c}$, with $a,b,c \in L$. Then the arguments done for $\RR^2$ imply that from $x$ Paul can move a finite number of steps and end up as close as needed to either $\ell_{a,b}$ or $\ell_{a,c}$. Recall that $L$ is open, so we can end up in $L$ after a finite amount of steps choosing $\varepsilon$ sufficiently small. Since $L\subset \Omega_0$, we just proved Paul has a strategy to win starting from $x$, provided that the time is large enough.
\end{proof}

\begin{proof}[Proof of Theorem \ref{cor2}]
Since  $\limsup_{t \to \infty}\Omega_{t}\subset\textrm{co}(K)$ was already proved in Theorem \ref{SecondPartTheorem1.5}, the previous Theorem  \ref{ThmMisu1} concludes the whole proof of Theorem \ref{cor2}.  Concerning the asymptotic behaviour of the minimal curvature flow with an obstacle, under the hypotheses of the Comparison Principle  Theorem \ref{comp.pp} and condition \eqref{(G)}, we have that 
\begin{equation}
	\text{co}(K)\subset\liminf\limits_{t\to\infty}\Omega_{t}\subset\limsup\limits_{t\to\infty}\Omega_{t}\subset \text{co}(K),
\end{equation}
and we conclude that the limit exists,
\begin{equation}
	\lim_{t\to \infty} \Omega_{t} = \text{co}(K),
\end{equation}
as we wanted to show.
\end{proof}

	{\bf Acknowledgments}
	
	I. Gonzálvez and J. Ruiz-Cases were partially supported by the European Union's Horizon 2020 research and innovation program under the Marie Sklodowska-Curie grant agreement No.\,777822, and by grants CEX2019-000904-S, PID-2019-110712GB-I00, PID-2020-116949GB-I00, and RED2022-134784-T, by MCIN/AEI (Spain).

	A. Miranda and J. D. Rossi were partially supported by 
			CONICET PIP GI No 11220150100036CO
(Argentina), PICT-2018-03183 (Argentina) and UBACyT 20020160100155BA (Argentina).


\begin{thebibliography}{99} \itemsep10pt 

\bibitem{Almeida}
L. Almeida, A. Chambolle and M. Novaga. {\it Mean curvature flow with obstacles}. Ann. Inst.
H. Poincare Anal. Non Lineaire, 29(5), (2012), 667--681.


\bibitem{BR} P. Blanc and J. D. Rossi. {\it Games for eigenvalues of the Hessian and concave/convex envelopes.} 
J. Math. Pures Appl., 127, (2019), 192--215.

\bibitem{BRLibro} P. Blanc and J. D. Rossi. {Game Theory and Partial Differential Equations.}
De Gruyter Series in Nonlinear Analysis and Applications. Vol. 31. 2019.
ISBN 978-3-11-061925-6.
ISBN 978-3-11-062179-2 (eBook).


\bibitem{2} K. A. Brakke. The motion of a surface by its mean curvature. Princeton University Press, Princeton, N.J., 1978.
	
	
	\bibitem{CIL} 
	M.G. Crandall,  H. Ishii and  P.L. Lions,
	{\it User's guide to viscosity solutions of second order partial differential equations}. 
	Bull. Amer. Math. Soc., 27 (1992), 1--67.
	
	

	\bibitem{ES} 
	L.C. Evans, J. Spruck, 
	{\it Motion of level sets by mean curvature. I}. 
	J. Differential Geom., 33(3) (1991), 635--681.


	\bibitem{CGG} 
	Y.G. Chen,  Y. Giga and  S. Goto,
	{\it Uniqueness and existence of viscosity solutions of generalized
	mean curvature flow equations}. 
	J. Differential Geom., 33(3) (1991), 749--786.
	
	\bibitem{5} M. Gage and R. S. Hamilton. {\it The heat equation shrinking convex plane curves}. J. Differential Geom., 26 (1986), 69--96.
	
	\bibitem{Giga} Y. Giga. Surface evolution equations. A level set approach. Birkhauser Verlag, Basel, 2006.

	\bibitem{GGIS} 
	Y. Giga and  S. Goto, H. Ishii, M.-H. Sato
	{\it Comparison principle and convexity preserving properties for
	singular degenerate parabolic equations on unbounded domains}. 
	Indiana Univ. Math. J., 40(2) (1991), 443--470.
	
	\bibitem{6} Y. Giga and Q. Liu. {\it A billiard-based game interpretation of the Neumann problem for the
curve shortening equation}. Adv. Differential Equations, 14(3-4) (2009),
201--240.

\bibitem{8} Y. Giga, H. Mitake, and H. V. Tran. {\it On asymptotic speed of solutions to level-set mean curvature flow equations with driving and source terms}. SIAM J. Math. Anal., 48, (2016), 3515--3546.

\bibitem{9} Y. Giga, H. V. Tran, and L. Zhang. {\it On obstacle problem for mean curvature flow with driving force}. Geometric flows, 4, (2019), 9--29.
	
	\bibitem{10} M. A. Grayson. {\it The heat equation shrinks embedded plane curves to round points}. J. Differential Geom., 26 (1987), 285--314.
	
	\bibitem{KS} R.V. Kohn and S. Serfaty.
	{\it A deterministic-control-based approach to motion by curvature}. 
	Comm. Pure Appl. Math., 59(3), (2006), 344--407.

\bibitem{KS2} R.Kohn and S. Serfaty. {\it A deterministic-control-based approach to fully nonlinear parabolic and elliptic equations}. Comm. Pure Appl. Math., 
10, (2010), 1298--1350.

\bibitem{Lewicka} M. Lewicka. A Course on Tug-of-War Games with Random Noise.
Introduction and Basic Constructions. Universitext book series. Springer, (2020).

\bibitem{ML} M. Lewicka and J. J. Manfredi. {\it The obstacle problem for the $p-$laplacian via optimal stopping of tug-of-war games.}
Prob. Theory Rel. Fields, 167(1-2), (2017), 349--378.


\bibitem{15} Q. Liu. {\it Fattening and comparison principle for level set equations of mean curvature type}. SIAM J. Control Optim., 49, (2011), :2518--2541.


\bibitem{16} Q. Liu and N. Yamada. {\it An obstacle problem arising in large exponent limit of power mean curvature flow equation}. Trans. Amer. Math. Soc., 372 (2019),
2103--2141.

\bibitem{Luiro}
H. Luiro, M. Parviainen, and E. Saksman.
{\it Harnack's inequality for p-harmonic functions via stochastic games.}
Comm. Partial Differential Equations, 38(11), (2013), 1985--2003.


	\bibitem{MS2}
A.~P. Maitra and W.~D. Sudderth.
\newblock {\em Discrete gambling and stochastic games}, volume~32 of {\em
  Applications of Mathematics (New York)}.
\newblock Springer-Verlag, New York, 1996.


\bibitem{MPR} J. J. Manfredi, M. Parviainen and J. D. Rossi. {\it An asymptotic mean value characterization for $p-$harmonic functions}. Proc. Amer. Math. Soc., 138(3), (2010), 881--889. 


\bibitem{MPR2} J. J. Manfredi, M. Parviainen and J. D. Rossi.
\textit{On the definition and properties of p-harmonious functions.}
Ann. Scuola Nor. Sup. Pisa, 11, (2012), 215--241.


\bibitem{17} J. J. Manfredi, J. D. Rossi, and S. J. Somersille. {\it An obstacle problem for tug-of-war games}. Commun. Pure Appl. Math., 14 (2015), 217--228.


	\bibitem{Mercier}
	G. Mercier
	{\it Mean curvature flow with obstacles: a viscosity approach}. 
	https://arxiv.org/abs/1409.7657v3, (2014).
	
	\bibitem{19} G. Mercier and M. Novaga. {\it Mean curvature flow with obstacles: Existence, uniqueness and regularity of solutions}. Inter. Free Bound. 17,
	(2015), 399--426.
	
	\bibitem{MiRo} A. Miranda and J. D. Rossi. {\it Games for the two membranes problem.}
Orbita Math. 1(1), (2024), 59--101.


	\bibitem{Misu}
	K. Misu
	{\it A game-theoretic approach to the asymptotic behavior of solutions to an obstacle problem for the mean curvature flow
	equation}. 
	http://hdl.handle.net/2115/87822, (2023).
	
	\bibitem{PSSW} Y. Peres, O. Schramm, S. Sheffield and D. Wilson,
{\it Tug-of-war and the infinity Laplacian.} J. Amer. Math. Soc.,
22, (2009), 167--210.


\bibitem{PS} Y. Peres and S. Sheffield, {\it Tug-of-war with noise:
a game theoretic view of the $p$-Laplacian}, Duke Math. J., 145(1), (2008), 91--120.

\bibitem{R} J. D. Rossi. {\it Tug-of-war games and PDEs.} Proc.
Royal Soc. Edim. 141A, (2011), 319--369.


\end{thebibliography}
\end{document}